\documentclass[11pt, reqno]{amsart}

\usepackage{amsmath,amssymb,latexsym}
\usepackage{color}
\usepackage{xcolor}
\usepackage{graphicx}
\usepackage{cite}
\usepackage[colorlinks=true]{hyperref}
\hypersetup{urlcolor=blue, citecolor=red, linkcolor=blue}

\newcommand{\diff}{\mathrm{d}}

\newtheorem{theorem}{Theorem}[section]
\newtheorem{lemma}{Lemma}

\newtheorem{remark}{Remark}
\newtheorem{definition}{Definition}

\newtheorem*{main-theorem}{Main Theorem}
\newtheorem*{remark*}{Remark}
\newtheorem*{lemma*}{Lemma A.1}

\numberwithin{equation}{section}

\begin{document}

\title[the viscous Saint-Venant system
for shallow waters]{Well-posedness of classical solutions to the vacuum free boundary
problem of the viscous Saint-Venant system
for shallow waters}

\author{Hai-Liang Li}

\author{Yuexun Wang}

\author{Zhouping Xin}

\address{ School of Mathematics and CIT, Capital Normal University, Beijing 100048, P. R. China.}
\email{hailiang.li.math@gmail.com}

\address{	
	School of Mathematics and Statistics,
	Lanzhou University, 730000 Lanzhou, China.}

\email{yuexunwang@lzu.edu.cn}

\address{The Institute of Mathematical Sciences, The Chinese University of Hong Kong, Hong Kong}

\email{zpxin@ims.cuhk.edu.hk}

\thanks{}

\begin{abstract} We establish the local-in-time well-posedness of classical solutions to the vacuum free boundary
problem of the viscous Saint-Venant system
for shallow waters derived rigorously from incompressible
Navier-Stokes system with a moving free surface by Gerbeau-Perthame \cite{MR1821555}. Our solutions (the height and velocity) are smooth (the solutions satisfy the equations point-wisely) all the way to the moving boundary, although the height degenerates as a singularity of the distance to the vacuum boundary. The proof is built on some new higher-order weighted energy functional and weighted estimates associated to the degeneracy near the moving vacuum boundary.
\end{abstract}
\maketitle

\section{Introduction }

The one-dimensional compressible isentropic Navier-Stokes equations with the density-dependent
viscosity coefficient are given by
\begin{equation}\label{eq:intro-1}
	\left\{
	\begin{aligned}
		&\rho_t+(\rho u)_x=0,\\
		&(\rho u)_t+(\rho u^2)_x+p_x
=(\mu(\rho)u_x)_x,
	\end{aligned}
	\right.
\end{equation}
where $(x,t)\in\mathbb{R}\times\mathbb{R}_+ $, and $\rho(x,t)\geq 0,u(x,t)$ and $p=\rho^\gamma\ (\gamma>1)$ stand for the  density, velocity, and pressure, respectively. And $\mu(\rho)=\rho^\alpha\ (\alpha\geq 0)$ is the viscosity coefficient.

There is a vast body of literature on the long time existence and asymptotic behavior of solutions to the system \eqref{eq:intro-1} in the case that the viscosity $\mu(\rho)$ is constant, i.e., \(\alpha=0\). When the initial density is strictly away from vacuum (\(\inf_{x\in\mathbb{R}} \rho_0(x) > 0\)),
the global existence of strong solutions was addressed for sufficiently smooth data by Kazhikhov et
al. \cite{MR0468593}, and for discontinuous initial data by Serre \cite{MR870700} and Hoff \cite{MR896014}, respectively.
The crucial point to establish such global existence of strong solutions lies in the fact that if the initial density is positive, then the density is positive for any later-on time as well. This fact is also proved to be true
for weak solutions by Hoff and Smoller \cite{MR1814847}, namely weak solutions
do not contain vacuum states in finite time as long as there is no vacuum initially.
When the initial density contains vacuum, the problem becomes subtle.
In fact, the appearance of vacuum indeed leads to some singular behaviors of solutions, such as the failure of continuous dependence of weak solutions on
initial data \cite{MR1117422} and the finite time blow-up of smooth solutions \cite{MR3360663, MR1488513}, and even non-existence of classical solutions with
finite energy \cite{MR3925527}.

Thus, when the solutions may contain vacuum states, it seems natural to investigate
the compressible Navier-Stokes equations with density-dependent viscosity.
Indeed, in the derivation of the
compressible Navier-Stokes equations from the Boltzmann equation by the Chapman-Enskog expansions,
as pointed out and investigated by Liu-Xin-Yang \cite{MR1485360}, the viscosity shall depend on the temperature and thus correspondingly depend on the density for isentropic flows.
Moreover, Gerbeau-Perthame \cite{MR1821555} derived rigorously a viscous Saint-Venant system for the shallow waters
which is expressed exactly to \eqref{eq:intro-1} with \(\alpha=1\) and \(\gamma=2\),  from the incompressible Navier-Stokes equation with a moving free surface. Such viscous compressible models with density-dependent viscosity coefficients and its variants also appear in geophysical flows \cite{MR1978317, MR1989675, MR2257849} (see also P.-L. Lions's book \cite{MR1637634}).
There are also extensive studies on the compressible Navier-Stokes equations with density-dependent viscosity.
When the initial density was assumed to be connected to vacuum with
discontinuities,
the local well-posedness of weak solutions to this problem was first established by Liu-Xin-Yang \cite{MR1485360}, and
the global existence of weak solutions for \(0<\alpha<1\) was considered by many authors,
see \cite{MR2254008} and the references therein. The above analysis relies heavily
on the fact that the density 
of the approximate solutions has a uniform positive lower bound in the non-vacuum region. 
When the density connects to
vacuum continuously, the density has no positive lower bound and thus the viscosity
coefficient vanishes at vacuum. This degeneracy in the viscosity coefficient gives rise to
some new difficulties. Despite of this, there is still much progress, for instance, one may refer to \cite{MR1929151} when \(\alpha>1/2\) for the local existence result, and \cite{MR1990849} for the global existence results of
weak solutions when \(0<\alpha<1/2\), in the free boundary
setting.
For \(\alpha>1/2\), some phenomena of vacuum vanishing and blow-up of solutions were found by Li-Li-Xin \cite{MR2410901}, more precisely, the authors proved that for any global entropy weak solution, the vacuum state must vanish within finite time, and the velocity blows up in finite time as the vacuum states vanish.
For the study on the asymptotic stability of rarefaction waves to this problem, one may refer to \cite{MR3116645} and the references therein.

Since P.-L. Lions' breakthrough work \cite{MR1226126, MR1637634},
there have also been much important progress for the multi-dimensional isentropic Navier-Stokes equations with the constant coefficients or density-dependent viscosity coefficients, see \cite{MR1779621, MR2914266, MR3862947, MR1978317, MR1867887, MR2877344,
LX2015, MR1810944, MR0564670, MR3573976, MR4011697, MR4195559} and the references therein.

The vacuum free boundary problem
of \eqref{eq:intro-1} had attracted a vast of attractions in the past two decades.
In the case that the viscosity is constant, Luo-Xin-Yang \cite{MR1766564} studied the global regularity and behavior of the weak solutions near the interface
when the initial density connects to vacuum states in a very smooth manner.  Zeng \cite{MR3303172} showed that the global existence of smooth solutions for which the smoothness extends all the way to
the boundary. In the case that the viscosity is density-dependent, the global existence of weak solutions was studied by many authors, see \cite{MR1990849} without external force, and \cite{MR2771263, MR2259332, MR2068308} with external force and the references therein.
By taking the effect of external force into account, Ou-Zeng \cite{MR3397339} obtained the global well-posedness of strong solutions and the global regularity uniformly up to the vacuum boundary.

Although there have been much important progress as aforementioned, it is still not clear whether the above solutions are smooth or not even locally in time when the viscosity coefficient vanishes at vacuum. In the present paper, we study the local well-posedness of classical solutions to the vacuum free boundary problem of the viscous Saint-Venant system
for shallow waters derived rigorously from the incompressible
Navier-Stokes system with a moving free surface by Gerbeau-Perthame \cite{MR1821555}, which corresponds to \eqref{eq:intro-1} with \(\alpha=1\) and \(\gamma=2\), i.e.,
\begin{equation}\label{eq:intro-2}
\begin{cases}
\rho_t+(\rho
u)_x=0 &\quad \text{in}\ I(t),\\
(\rho u)_t+(\rho u^2+\rho^2)_x=(\rho u_x)_x &\quad \text{in}\ I(t),\\
\rho>0 &\quad \text{in}\ I(t),\\
\rho=0 &\quad \text{on}\ \Gamma(t),\\
\mathcal{V}(\Gamma(t))=u,\\
(\rho,u)=(\rho_0,u_0) &\quad \text{on}\ I(0),\\
I(0)=I=\{x: 0<x<1\}.
\end{cases}
\end{equation}

To solve the system \eqref{eq:intro-2}, we need to solve the four pairs \((\rho, u, I(t), \Gamma(t))\) (in fact it suffices to solve the triple \((\rho, u, \Gamma(t))\)).
Here \(\rho\) denotes the \emph{height} of the fluid (we use this terminology
from its original meaning), and  \(u\) denotes the Eulerian velocity, respectively.
The open, bounded interval \(I(t)\) denotes the changing domain occupied by the
fluid, \(\Gamma(t)=\colon \partial I(t)\) denotes the moving vacuum boundary, and \(\mathcal{V}(\Gamma(t))\) denotes the velocity of \(\Gamma(t)\), respectively.
Equation \(\eqref{eq:intro-2}_1\) stands for the conservation of mass, and Equation \(\eqref{eq:intro-2}_2\) describes the conservation of
momentum, the condition \(\eqref{eq:intro-2}_3\) means that there is no vacuum inside of fluid,
the conditions \(\eqref{eq:intro-2}_4\) tell the dynamical boundary conditions to be investigated, \(\eqref{eq:intro-2}_5\) states that the vacuum boundary is moving with the fluid velocity, and \(\eqref{eq:intro-2}_6\) are the initial conditions for the \emph{height}, velocity, and domain.

The initial \emph{height} profile we are interested in this paper connects to vacuum as follows:
\begin{align}\label{eq:intro-3}
\rho_0\in H^5(\bar{I}(0))\ \text{and}\  C_1d(x)\leq \rho_0(x)\leq C_2d(x) \quad \text{for\ all}\ x\in \bar{I}(0),
\end{align}
for some positive constants \(C_1\) and \(C_2\), where \(d(x)=\colon d(x,\Gamma(0))\) is the distant function from \(x\) to the initial boundary.

We also explain a little bit on the condition \eqref{eq:intro-3}. The condition \eqref{eq:intro-3} is equivalent to the following so-called "physical vacuum singularity". Let \(c(x,t)=\sqrt{\rho(x,t)}\) be the sound speed, and hence \(c_0=c(x,0)\) is the initial sound speed.
The physical vacuum singularity (see, for example, \cite{MR2779087, MR1485360}) is determined by the following condition
\begin{equation}\label{eq:intro-4}
\begin{aligned}
 0<\bigg|\frac{\diff c_0^2}{\diff x}\bigg|<\infty \quad \mbox{on}\ \Gamma(0).
\end{aligned}
\end{equation}
It is straightforward to check that \eqref{eq:intro-3} is equivalent to \eqref{eq:intro-4} by assuming \(\rho_0(x)\) vanishes on the boundary \(\Gamma(0)\).

The study on the physical vacuum free boundary problem for the compressible Euler equations was first given by Jang-Masmoudi \cite{MR2547977} and Coutand-Lindblad-Shkoller \cite{MR2608125} with different methods handling the degeneracy near the free boundary. For other important progress on the vacuum free boundary problems in compressible fluids, one may also refer to \cite{MR2177323, MR2779087,  MR2980528, MR3280249, MR3218831, MR3551252} and the references therein.

The physical vacuum free boundary problem of shallow waters was studied by both Duan \cite{MR2914266} and Ou-Zeng \cite{MR3397339}, with the external force "\(-\rho f\)" (imposed on the right hand side of the momentum equation \(\eqref{eq:intro-2}_2\)), for global theory. In \cite{MR2914266}, the author considered some kind initial density degenerated as \(d^{1/2}(x)\) near the vacuum boundary and showed the global well-posedness of weak solutions by establishing certain global space-time square estimates using Lagrangian mass coordinates. In \cite{MR3397339}, the authors considered some sort of initial density like \(d(x)\) near the vacuum boundary and showed the global well-posedness of strong solutions based on certain weighted energy estimates with both space and time weights using Hardy's inequality together with the particle path method.

We aim to present a detailed proof on the local well-posedness of classical solutions (see Definition \ref{Classical Solution} (b)) to the vacuum free boundary problem \eqref{eq:intro-2}-\eqref{eq:intro-3} in the present paper. Comparing with \cite{MR2914266, MR3397339}, our classical solution satisfies an additional Nuewmann boundary condition \(u_x=0\ \text{on}\ \Gamma(t)\), which is captured by the high regularity of the solution on the vacuum boundary (see Remark \ref{re:main-1}-\ref{re:main-3}). 

To handle the degeneracy near the vacuum boundary and to capture 
the feature \(u_x=0\ \text{on}\ \Gamma(t)\) of our classical solution, we first construct a higher-order energy functional associated to
the degeneracy near the vacuum boundary, and then develop some delicate weighted estimates to close the higher-order energy functional, in which the weighted Sobolev inequalities and some weighted interpolation inequality will play an important role.  Our higher-order energy functional consists of the following four type terms:
\[\int_I \rho_0(\partial_t^{k_1}v)^2\,\diff x,\
\int_I \rho_0(\partial_t^{k_2}v_x)^2\,\diff x,\
\int_I \rho_0^{k_3}(\partial_t\partial_x^{k_3}v)^2\,\diff x,\
\int_I \rho_0^{k_4}(\partial_x^{k_4}v)^2\,\diff x,\]
for some non-negative integers \(k_1, k_2, k_3, k_4\) to be chosen. The first two type terms come from the time-differentiated energy estimates, which are essentially the estimates of the derivatives
in the tangential direction of the moving boundary. 
While the last two type terms are from the elliptic estimates, which depend highly on the degenerate parabolic structure of the momentum equation in \eqref{eq:main-2} and make it possible for us to gain more regularities through the estimates of the derivatives in the normal direction of the moving boundary.

Constructing approximate solutions usually is not a trivial process in showing well-posedness of the physical vacuum free boundary problem of compressible fluids since the system degenerates on the boundary, see \cite{MR2547977, MR2779087, MR2980528, MR3280249}.
 In \cite{MR2779087}, in order to get the regular solution to the compressible Euler equations, Coutand-Shkoller considered a degenerate parabolic regularization well matched with the compressible Euler equations, more precisely where the viscosity has a structure \(\kappa(\rho_0^2v_x)_x\). To show the existence of weak solutions to this degenerate parabolic equation by the Galerkin's scheme, the authors introduced a new variable \(X=\rho_0v_x\) which satisfies a Dirichlet boundary condition
\(X=0\ \text{on}\ \partial I\times [0,T]\) since \(\rho_0\) vanishes on the boundary and \(v_x\) is bounded and then studied the equation for \(X\) instead of \(v\).
(Note that \(v\) itself does not satisfy any boundary condition.) On the other hand, to tackle the strong degeneracy of the viscosity, the authors had to divide a weight \(\rho_0\) on both sides of the degenerate parabolic equation to lower the degeneracy (but there is no singularity in the new equation), where a new higher-order Hardy-type inequality necessitates.

It seems difficult to apply the idea of \cite{MR2779087} straightforwardly to construct approximate solutions of the viscous Saint-Venant system for shallow waters \eqref{eq:intro-2} (see Remark \ref{re:main-5}). 
In this paper, we will construct a classical solution to the vacuum free boundary problem \eqref{eq:intro-2}-\eqref{eq:intro-3} satisfying the Nuewmann boundary condition \eqref{Nuewmann boundary condition-2} (see Remark \ref{re:main-1}-\ref{re:main-3}), so this boundary condition will play an important role in constructing approximate solutions in the Hilbert space \(\mathcal{H}(I)=\{h\in H^3(I): h_x=0 \ \text{on}\ \Gamma \}\). We will first use the Galerkin's scheme to
construct a unique weak solution to the linearized problem, and then improve the regularity of this weak solution based on some key higher order a priori estimates, and finally show  that the approximate solutions converge to a unique classical solution to the degenerate parabolic problem by a contraction mapping method. 

It should be pointed out that, on the one hand, in deducing a priori estimates on higher order derivatives here, one can not manipulate as \cite{MR2779087} to divide a \(\rho_0\) on both sides of the degenerate parabolic equation to lower the degeneracy since it will introduce some singularity in the new equation which  prevents the analysis to work. Hence we will keep the original structure of the degenerate parabolic equation, and use mainly the weighted Sobolev inequalities to handle the degeneracy which depends heavily on the degenerate parabolic structure of the momentum equation in \eqref{eq:main-2}. On the other hand, due to the degeneracy, the energy estimates on the
 approximate solutions are insufficient for us to pass limit in \(n\) on the iteration problem for time pointwisely. 
Therefore we need use some weighted interpolation inequality that can help us to obtain a pointwise convergence for time
on the approximate solutions to the iteration problem (see Section \ref{classical solution}).

In \cite{MR2864798}, Guo-Li-Xin studied the multi-dimensional
viscous Saint-Venat system for the shallow waters and showed the global existence of a spherically symmetric
weak solution to its free boundary value problem, in which detailed regularity and Lagrangian structure
of this solution was presented. It is interesting to extend our classical solutions' result to the multi-dimensional (spherically symmetric) viscous Saint-Venat system for the shallow waters, which is left for future.

The paper is organized as follows. In Section \ref{main results}, we will first formulate the vacuum free boundary problem into a fixed boundary problem and then state our main results.  Section \ref{Some Preliminaries} lists some preliminaries. 
In Section \ref{Energy Estimates} and \ref{Elliptic estimates}, respectively, 
we will focus on the a priori estimates that constitute the energy estimates and elliptic estimates. Section \ref{Existence Part} and \ref{Uniqueness Part} are devoted to showing the existence and uniqueness of a classical solution to our degenerate parabolic problem, respectively.

\section{Reformulation and main results}\label{main results}

\subsection{Fixing the domain}

The initial domain (the reference domain) in one-dimension is given by \(I(0)=(0,1)\). Afterwards, we will use the short notation \(I\) to replace \(I(0)\) for convenience, and also denote by \(\Gamma=\partial I\) the boundary of the reference domain.

Denote by \(\eta\) the position of the fluid particle \(x\) at time \(t\)
\begin{equation}\label{fluid particle}
\begin{cases}
\partial_t\eta(x,t)=u(\eta(x,t),t),\\
\eta(x,0)=x,
\end{cases}
\end{equation}
and also by \(f(x,t)\) and \(v(x,t)\) the Lagrangian \emph{height} and velocity
\begin{equation}\label{Lagrangian variable}
\begin{cases}
f(x,t)=\rho(\eta(x,t),t),\\
v(x,t)=u(\eta(x,t),t).
\end{cases}
\end{equation}
Then  \eqref{eq:intro-2} is transformed to the following problem on the fixed reference interval \(I\):
\begin{equation}\label{eq:main-1}
\begin{cases}
f_t+{\frac{fv_x}{\eta_x}}=0 &\quad \text{in}\ I\times (0,T],\\
\eta_xfv_t+(f^2)_x=(f{\frac{v_x}{\eta_x}})_x &\quad \text{in}\ I\times (0,T],\\
f>0 &\quad \text{in}\ I\times (0,T],\\
f=0 &\quad \text{on}\ \Gamma\times (0,T],\\
(f,v,\eta)=(\rho_0,u_0,e) &\quad \text{on}\ I\times \{t=0\},
\end{cases}
\end{equation}
where \(e(x)=x\) denotes the identity map on \(I\).

Solving \(f\) from Equation \(\eqref{eq:main-1}_1\) yields
\begin{align}\label{eq:Lag-initial}
f(x,t)=\rho_0(x)\eta_x^{-1}(x,t),
\end{align}
one inserts \eqref{eq:Lag-initial} back to Equation \(\eqref{eq:main-1}_2\) to transfer the problem \eqref{eq:main-1} into
\begin{equation}\label{eq:main-2}
\begin{cases}
\rho_0v_t+\big({\frac{\rho_0^2}{\eta_x^2}}\big)_x
=\big(\frac{\rho_0v_x}{\eta_x^2}\big)_x &\quad \mbox{in}\ I\times (0,T],\\
(v,\eta)=(u_0,e) &\quad \mbox{on}\ I\times \{t=0\}.
\end{cases}
\end{equation}

 The problem \eqref{eq:main-2} is a degenerate parabolic problem.

\begin{definition}[Classical Solution]\label{Classical Solution} {\rm{(a)}} We say a function \(v\)  
	is a classical solution to the problem \eqref{eq:main-2}  provided  \(v\) satisfies  \(\eqref{eq:main-2}_1\) in \(\bar{I}\times (0,T]\) pointwisely and is continuous to the initial data \(u_0\).	
	
{\rm{(b)}} We say the pair  \((\rho(x,t), u(x,t), \Gamma(t))\) for \(t\in[0,T]\) and \(x\in I(t)\)  
	is a classical solution to the problem \eqref{eq:intro-2} provided  \((\rho(x,t), u(x,t), \Gamma(t))\) satisfies  \(\eqref{eq:intro-2}_1-\eqref{eq:intro-2}_5\) pointwisely  and is continuous to the initial data \((\rho_0, u_0, \Gamma)\), additionally, \(\eqref{eq:intro-2}_1\) and \(\eqref{eq:intro-2}_2\) hold on the spatial boundary of \(I(t)\) pointwisely. 

\end{definition}

\subsection{The higher-order energy functional}

Our main purpose is to study the local well-posedness of the degenerate parabolic problem \eqref{eq:main-2} in certain weighted Sobolev space with high regularity. For this, we will consider the following higher-order energy functional:
\begin{equation}\label{higher-order energy function}
\begin{aligned}
E(t,v)&=\sum_{k=0}^3\|\sqrt{\rho_0}\partial_t^kv(\cdot,t)\|_{L^2(I)}^2
+\sum_{k=0}^2\|\sqrt{\rho_0}\partial_t^kv_x(\cdot,t)\|_{L^2(I)}^2\\
&\quad+\sum_{k=2}^4\big\|\sqrt{\rho_0^k}\partial_t\partial_x^kv(\cdot,t)\big\|_{L^2(I)}^2
+\sum_{k=2}^6\big\|\sqrt{\rho_0^k}\partial_x^kv(\cdot,t)\big\|_{L^2(I)}^2.
\end{aligned}
\end{equation}

We define the polynomial function \(M_0\) by
\begin{align*}
M_0=P(E(0,v_0)),
\end{align*}
where \(P\) denotes a generic polynomial function of its arguments.

\subsection{Main result on the problem \eqref{eq:main-2}}

The main result in the paper can be stated as follows:
\begin{theorem}\label{th:main-1} Assume the initial data \((\rho_0,v_0)\) satisfy \eqref{eq:intro-3} and \(M_0<\infty\), then
there exist a suitably small \(T>0\) and a unique classical solution 
\begin{align}\label{regularity}
v\in C([0,T]; H^3(I))\cap C^1([0,T]; H^1(I))
\end{align}
 to the problem \eqref{eq:main-2} on \([0,T]\) such that
\begin{align}\label{eq:inequality-1}
\sup_{0\leq t\leq T} E(t,v)\leq 2M_0.
\end{align}
\end{theorem}

Moreover, \(v\) satisfies the Nuewmann boundary condition
\begin{align}\label{Nuewmann boundary condition}
v_x=0 \quad \mbox{on}\ \Gamma\times (0,T].
\end{align}

\subsection{Main result on the vacuum free
	boundary problem \eqref{eq:intro-2}-\eqref{eq:intro-3}}

Due to \eqref{eta-bound}, the flow map \(\eta(\cdot,t)\colon I\rightarrow I(t)\) is inverse for any \(t\in[0,T]\) and we denote its inverse by \(\tilde{\eta}(\cdot,t)\colon I(t)\rightarrow I\), where \(T\) is determined in Theorem \ref{th:main-1}. Let \((\eta,v)\) be the unique classical solution in Theorem \ref{th:main-1}. For \(t\in[0,T]\) and \(y\in I(t)\), set
\begin{equation*}
\begin{aligned}
&\rho(y,t)=\rho_0(\tilde{\eta}(y,t))\eta_x^{-1}(\tilde{\eta}(y,t),t),\\
&u(y,t)=v(\tilde{\eta}(y,t),t).
\end{aligned}
\end{equation*}
Then the triple \((\rho(y,t), u(y,t), \Gamma(t))\)) (\(t\in[0,T]\)) defines a unique classical solution to the vacuum free
boundary problem \eqref{eq:intro-2}-\eqref{eq:intro-3}. More precisely, Theorem \ref{th:main-1} can be transferred into the following:

\begin{theorem}\label{th:main-2} Assume the initial data \((\rho_0,u_0)\) satisfy \eqref{eq:intro-3} and \(M_0<\infty\), then
there exist a \(T>0\) and a unique classical solution \((\rho(y,t), u(y,t), \Gamma(t))\) for \(t\in[0,T]\) and \(y\in I(t)\) to the vacuum free
boundary problem \eqref{eq:intro-2}-\eqref{eq:intro-3}. Moreover, \(\Gamma(t)\in C^2([0,T])\), and for \(t\in[0,T]\) and \(y\in I(t)\), we have
\begin{equation}\label{regularity-2}
\begin{aligned}
&\rho(y,t)\in C([0,T];H^3(I(t)))\cap C^1([0,T];H^2(I(t)));\\
&u(y,t)\in C([0,T];H^3(I(t)))\cap C^1([0,T];H^1(I(t))).
\end{aligned}
\end{equation}

Moreover, \(u\) satisfies the Nuewmann boundary condition
\begin{align}\label{Nuewmann boundary condition-2}
u_x=0 \quad \rm{on}\ \Gamma(t).
\end{align}

\end{theorem}

\subsection{Some remarks}
The following remarks are helpful for understanding our main results.

\begin{remark}\label{re:main-1}
By the trace theorem \(H^3(I)\hookrightarrow H^{5/2}(\Gamma)\) (see \cite{MR2597943} for instance) and \(v(\cdot,t)\in H^3(I)\) for each \(t\in(0,T]\), one may define the Nuewmann boundary condition \eqref{Nuewmann boundary condition} pointwisely due to \eqref{regularity}. Similarly, one can also define \eqref{Nuewmann boundary condition-2} by \(u(y,t)\in C([0,T];H^3(I(t)))\) for  \(t\in[0,T]\) and \(y\in I(t)\) pointwisely due to \eqref{regularity-2}.

\end{remark}

\begin{remark}\label{re:main-2} It follows  from Remark \ref{re:main-1} that \eqref{Nuewmann boundary condition} is well-defined if the solution to the problem \eqref{eq:main-2} possesses the regularity \eqref{regularity}. In fact, \eqref{Nuewmann boundary condition} holds naturally for the classical solution in the sense of Definition \ref{Classical Solution} (a), however, with a higher regularity \eqref{eq:inequality-1}. In the following, we show how to derive \eqref{Nuewmann boundary condition} from Definition \ref{Classical Solution} (a) together with \eqref{eq:inequality-1}.  

First note from Equation \(\eqref{eq:main-2}_1\) that 
\begin{align}\label{r1}
	\rho_0v_t+{\frac{2\rho_0(\rho_0)_x}{\eta_x^2}}-{\frac{2\rho_0^2\eta_{xx}}{\eta_x^3}}
	=\frac{(\rho_0)_xv_x}{\eta_x^2}+\rho_0\bigg(\frac{v_{xx}}{\eta_x^2}-\frac{2v_x\eta_{xx}}{\eta_x^3}\bigg),
\end{align}
for \((x,t)\in \ I\times (0,T]\). 
It follows from \eqref{eq:inequality-1}, Lemma \ref{le:Preliminary-1} and Lemma \ref{le:Preliminary-2} that
\begin{align*}
	\rho_0v_t(\cdot, t),\ v_x(\cdot, t),\ \eta_x(\cdot, t),\ \rho_0v_{xx}(\cdot, t),\ \rho_0\eta_{xx}(\cdot, t)\in H^2(I)\quad  \text{for}\ t\in (0,T],
\end{align*}
which combines the trace theorem \(H^2(I)\hookrightarrow H^{3/2}(\Gamma)\) yields 
\begin{align*}
	\rho_0v_t(\cdot, t),\ v_x(\cdot, t),\ \eta_x(\cdot, t),\ \rho_0v_{xx}(\cdot, t),\ \rho_0\eta_{xx}(\cdot, t)\in H^{3/2}(\Gamma)\quad  \text{for}\ t\in (0,T].
\end{align*}
This implies that each term in \eqref{r1} is well-defined pointwisely on \( \Gamma\times (0,T]\).
Using \eqref{eq:intro-3}, \eqref{eta-bound}, and  letting \(x\) go to the vacuum boundary \(\Gamma(t)\), then one obtains  
\begin{align}\label{r2}
	(\rho_0)_xv_x=0 \quad \rm{on}\ \Gamma\times (0,T]. 
\end{align}
By \eqref{eq:intro-3} again, one sees \((\rho_0)_x\neq 0\) on \(\Gamma\), hence \eqref{Nuewmann boundary condition} follows from \eqref{r2}.

On the other hand side, to construct a classical solution to the problem \eqref{eq:main-2}, we will use a Galerkin's scheme to study its linearized problem, in which the Nuewmann boundary condition \eqref{Nuewmann boundary condition} will play a crucial role. 

\end{remark}

\begin{remark}\label{re:main-3} 
For the problem \eqref{eq:intro-2}-\eqref{eq:intro-3}, since \(\rho\) vanishes on \(\Gamma(t)\),  
the usual stress free condition
\begin{align}\label{r3}
S=\rho^2-\rho u_x=0\quad \rm{on}\ \Gamma(t)
\end{align}
holds automatically.

\end{remark}

\begin{remark}\label{re:main-4}
 In \cite{MR2779087}, Coutand-Shkoller studied the well-posedness of the physical vacuum free boundary problem of the compressible Euler equations, which may be written in Lagrangian coordinates as
\begin{align}\label{r3.5}
	\rho_0v_t+\bigg({\frac{\rho_0^\gamma}{\eta_x^\gamma}}\bigg)_x
	=0.
\end{align}
For \(1<\gamma\leq 2\), the authors constructed the following energy functional (see Section 8 in \cite{MR2779087}):
\begin{equation}\label{CS'energy function}
\begin{aligned}
	E_\gamma(t,v)&=\sum_{s=0}^4\|\partial_t^sv(\cdot, t)\|_{H^{2-s/2}}^2+\sum_{s=0}^2\|\rho_0\partial_t^{2s}v(\cdot, t)\|_{H^{3-s}}^2+\|\sqrt{\rho_0}\partial_t\partial_x^2v(\cdot, t)\|_{L^2}^2\\
&\quad+\|\sqrt{\rho_0}\partial_t^3\partial_xv(\cdot, t)\|_{L^2}^2
+\sum_{a=0}^{a_0}\|\sqrt{\rho_0}^{1+\frac{1}{\gamma-1}-a}\partial_t^{4+a_0-a}\partial_xv(\cdot,t)\|_{L^2}^2,
\end{aligned}
\end{equation}
where \(a_0\) satisfies \(1<1+\frac{1}{\gamma-1}-a_0\leq 2\). Note that the last sum in \(E_\gamma\) appears whenever \(1<\gamma< 2\), and the order of the time-derivative increases to infinity as \(\gamma\to 1^+\). 

But the energy functional \eqref{CS'energy function} fails for \(\gamma=1\) whose equation corresponds to the isothermal Euler equation:
\begin{align}\label{r4}
	\rho_0v_t+\bigg({\frac{\rho_0}{\eta_x}}\bigg)_x
	=0.
\end{align}
Next, we will compare the isothermal Euler model with the shallow water model in the following two aspects. On the one hand, applying 
\(\partial_t\) to Equation \eqref{r4} yields
\begin{align}\label{r5}
	\rho_0\partial_t^2v=\bigg({\frac{\rho_0v_x}{\eta_x^2}}\bigg)_x.
\end{align}
The term \(\big({\frac{\rho_0v_x}{\eta_x^2}}\big)_x\) in Equation \eqref{r5} also appears in Equation \(\eqref{eq:main-2}_1\), which contributes the main difficulties in the elliptic estimates (see Section \ref{Elliptic estimates}).
One the other hand, it follows from \eqref{r4} that
\begin{align}\label{rIB}
	\rho_0v_t+{\frac{(\rho_0)_x}{\eta_x}}-{\frac{\rho_0\eta_{xx}}{\eta_x^2}}=0.
\end{align}
One can claim that there is no classical solution to \eqref{r4} living in some weighted Sobolev space with high regularity such that
\begin{align*}
	\rho_0v_t(\cdot, t),\ \rho_0\eta_{xx}(\cdot, t)\in H^2(I)\quad  \text{for}\ t\in (0,T].
\end{align*} 
Otherwise, one may argue as Remark \ref{re:main-2} for \eqref{rIB} to deduce
\begin{align*}
	(\rho_0)_x=0,
\end{align*} 
which contradicts \eqref{eq:intro-3}.

\end{remark}

\begin{remark}\label{re:main-5}
In \cite{MR2779087}, to construct the approximate solutions of \eqref{r3.5} with \(\gamma=2\), Coutand-Shkoller used the following parabolic \(\kappa\)-problem:
\begin{equation}\label{r6}
\begin{cases}
\rho_0v_t+\big({\frac{\rho_0^2}{\eta_x^2}}\big)_x
=\kappa(\rho_0^2v_x)_x &\quad \mbox{in}\ I\times (0,T],\\
(v,\eta)=(u_0,e) &\quad \mbox{on}\ I\times \{t=0\}
\end{cases}
\end{equation}
for small \( \kappa>0\). To show the existence of solutions to the problem \eqref{r6}, the authors considered its linearized problem  
\begin{equation}\label{r7}
\begin{cases}
\rho_0v_t+\big({\frac{\rho_0^2}{\bar{\eta}_x^2}}\big)_x
=\kappa(\rho_0^2v_x)_x &\quad \mbox{in}\ I\times (0,T],\\
(v,\eta)=(u_0,e) &\quad \mbox{on}\ I\times \{t=0\},
\end{cases}
\end{equation}
where 
\[\bar{\eta}(x,s)=x+\int_0^t\bar{v}(x,s)\,\diff s\]
for \(\bar{v}\) in some Hilbert space \(\mathcal{C}_T(M)\). The solution to the parabolic \(\kappa\)-problem \eqref{r6} will then be obtained as a fixed point of the map \(\bar{v}\mapsto v\) (\(v\) is a unique solution to the problem \eqref{r7}) in \(\mathcal{C}_T(M)\) for small \(T>0\)  via the Tychonoff fixed-point theorem (which requires that the solution space is a reflexive separable Banach space).

To show the existence of solutions to the problem \eqref{eq:main-2}, we also need to consider its linearized problem \eqref{existence-3}. However, the solution space (defined by \eqref{solution space}) for the problem \eqref{existence-3} (which is the same one with the problem \eqref{eq:main-2}) is a non-reflexive Banach space, which prevents us applying the Tychonoff fixed-point theorem straightforwardly to obtain the  existence of solutions to the problem \eqref{eq:main-2}. To get around the difficulty, we will design a contraction mapping for the approximate solutions to the iteration problem \eqref{existence-21} and show its approximate solutions converge uniformly to a classical solution to the problem \eqref{eq:main-2}, in which some weighted interpolation inequality is needed to overcome the difficulty of passing limit in \(n\) on the approximate solutions to the iteration problem \eqref{existence-21}
for time pointwisely, which is caused by the degeneracy in the energy estimates (see Section \ref{classical solution}).

\end{remark}

\section{Some Preliminaries}\label{Some Preliminaries}

\subsection{Weighted Sobolev inequalities}
To handle the degeneracy near the vacuum boundary, we will need the following weighted Sobolev inequalities, whose proof can be found for instance in \cite{MR0802206}. Let \(d(x)=\colon d(x,\Gamma)\) be the distant function to the boundary \(\Gamma\). Then the following weighted Sobolev inequalities hold:
\begin{equation}\label{ineq:weighted Sobolev-0}
\begin{aligned}
\|w\|_{H^{1/2}(I)}^2\lesssim \int_Id(x)(w^2+w_x^2)(x)\,\diff x,
\end{aligned}
\end{equation}
\begin{equation}\label{ineq:weighted Sobolev}
\begin{aligned}
\int_Id^k(x)w^2(x)\,\diff x\lesssim \int_Id^{k+2}(x)(w^2+w_x^2)(x)\,\diff x\quad \mbox{for}\ k=0,1,2,...,
\end{aligned}
\end{equation}
here and thereafter the convention \(\cdot\lesssim\cdot\) denotes \(\cdot\leq C\cdot \), and \(C\) always denotes a nonnegative universal constant which may be different from line to line.

Recall that the initial \emph{height} profile \(\rho_0(x)\) connects to vacuum as \eqref{eq:intro-3}, so the distance function \(d(x)\) can be replaced by \(\rho_0(x)\) in the weighted Sobolev inequalities \eqref{ineq:weighted Sobolev-0} and \eqref{ineq:weighted Sobolev}.

\subsection{Sobolev embedding}
The standard Sobolev embedding inequality
\begin{equation}\label{Sobolev inequaty}
\begin{aligned}
\|w\|_{L^{2/(1-2s)}(I)}\lesssim \|w\|_{H^s(I)}\quad \mbox{for}\ 0<s<1/2,
\end{aligned}
\end{equation}
will also be used.

\subsection{Consequences of \eqref{higher-order energy function}}\label{consequency of higher-order energy function}

As a prerequisite for later use, we will use the weighted Sobolev inequality \eqref{ineq:weighted Sobolev} to deduce some useful consequences of the boundness of the energy functional defined in  \eqref{higher-order energy function}. 
\begin{lemma}\label{le:Preliminary-1} It holds that
	\begin{align}\label{Preliminary-1}
		\|v(\cdot,t)\|_{H^3(I)}\lesssim E^{1/2}(t,v).
	\end{align}
	As a consequence, if \eqref{fluid particle} and \eqref{Lagrangian variable} hold, then 
	\begin{align}
		&\|\eta_{xx}(\cdot,t)\|_{L^2(I)}+\|\partial_x^3\eta(\cdot,t)\|
		_{L^2(I)}\lesssim t\sup_{0\leq s\leq t}E^{1/2}(t,v),\label{Preliminary-2}\\
		&\|v_x(\cdot,t)\|_{L^\infty(I)}
		+\|v_{xx}(\cdot,t)\|_{L^\infty(I)}\lesssim E^{1/2}(t,v),\label{Preliminary-3}\\
		&\|\eta_{xx}(\cdot,t)\|_{L^\infty(I)}\lesssim t\sup_{0\leq s\leq t}E^{1/2}(s,v)\label{Preliminary-5}.
	\end{align}
\end{lemma}
\begin{proof}
	Indeed, it follows from the weighted Sobolev inequality \eqref{ineq:weighted Sobolev} that
	\begin{equation*}
		\begin{aligned}
			\int_I v^2\,\diff x\lesssim \int_I\rho_0^2(v^2+v_x^2)\,\diff x\lesssim E(t,v),
		\end{aligned}
	\end{equation*}
	\begin{equation*}
		\begin{aligned}
			\int_I v_x^2\,\diff x\lesssim \int_I\rho_0^2(v_x^2+v_{xx}^2)\,\diff x\lesssim E(t,v),
		\end{aligned}
	\end{equation*}
	\begin{equation*}
		\begin{aligned}
			&\int_I v_{xx}^2\,\diff x\lesssim \int_I\rho_0^2[(v_{xx}^2+(\partial_x^3v)^2]\,\diff x\\
			&\lesssim E(t,v)+\int_I\rho_0^4[(\partial_x^3v)^2+(\partial_x^4v)^2]\,\diff x
			\lesssim E(t,v),
		\end{aligned}
	\end{equation*}
	and
	\begin{equation*}
		\begin{aligned}
			&\int_I (\partial_x^3v)^2\,\diff x\lesssim \int_I\rho_0^2[(\partial_x^3v)^2+(\partial_x^4v)^2]\,\diff x\\
			&\lesssim \int_I\rho_0^4[(\partial_x^3v)^2+(\partial_x^4v)^2]\,\diff x
			+\int_I\rho_0^4[(\partial_x^4v)^2+(\partial_x^5v)^2]\,\diff x\\
			&\lesssim E(t,v)+\int_I\rho_0^6[(\partial_x^5v)^2+(\partial_x^6v)^2]\,\diff x\lesssim E(t,v).
		\end{aligned}
	\end{equation*}
	Hence \eqref{Preliminary-1} follows.

	For \eqref{Preliminary-2}, it follows from \eqref{Preliminary-1} that
	\begin{equation*}
		\begin{aligned}
			\|\partial_x^k\eta(\cdot,t)\|_{L^2(I)}&\leq \int_0^t\left(\int_I (\partial_x^kv)^2\,\diff x\right)^{1/2}\diff s
			&\lesssim t \sup_{0\leq s\leq t}E^{1/2}(t,v)),\ k=2,3,
		\end{aligned}
	\end{equation*}
	where one has used Minkowski's inequality in the first inequality.
	
	The inequality \eqref{Preliminary-3} is a consequence of \eqref{Preliminary-1} and the Sobolev embedding \(H^1(I)\hookrightarrow L^\infty(I)\). Then the inequality  \eqref{Preliminary-5} may be shown as
	\begin{align*}
		\|\eta_{xx}(\cdot,t)\|_{L^\infty(I)}\leq \int_0^t\|v_{xx}(\cdot,s)\|_{L^\infty(I)}\diff s\leq t\sup_{0\leq s\leq t}E^{1/2}(s,v).
	\end{align*}
\end{proof}

Similarly, one also has
\begin{lemma}\label{le:Preliminary-2} It holds that
	 \begin{align}\label{Preliminary-6}
			\|\rho_0\partial_x^4v(\cdot,t)\|_{L^2(I)}+\|\rho_0^2\partial_x^5v(\cdot,t)\|_{L^2(I)}+\|\rho_0^3\partial_x^6v(\cdot,t)\|_{L^2(I)}\lesssim E^{1/2}(t,v).
	\end{align}
	As a consequence, if \eqref{fluid particle} and \eqref{Lagrangian variable} hold, then
	\begin{align}
			&\|\rho_0\partial_x^4\eta(\cdot,t)\|_{L^2(I)}+\|\rho_0^2\partial_x^5\eta(\cdot,t)\|_{L^2(I)}
			+\|\rho_0^3\partial_x^6\eta(\cdot,t)\|_{L^2(I)}
			\lesssim t\sup_{0\leq s\leq t}E^{1/2}(t,v),\label{Preliminary-7}\\
			&\|\rho_0\partial_x^3v(\cdot,t)\|_{L^\infty(I)}+\|\rho_0^2\partial_x^4v(\cdot,t)\|_{L^\infty(I)}+\|\rho_0^3\partial_x^5v(\cdot,t)\|_{L^\infty(I)}\lesssim E^{1/2}(t,v),\label{Preliminary-8}\\
			&\|\rho_0\partial_x^3\eta(t,\cdot)\|_{L^\infty(I)}
			+\|\rho_0^2\partial_x^4\eta(\cdot,t)\|_{L^\infty(I)}+\|\rho_0^3\partial_x^5\eta(\cdot,t)\|_{L^\infty(I)}\lesssim t\sup_{0\leq s\leq t}E^{1/2}(s,v)\label{Preliminary-9}.
	\end{align}
\end{lemma}
\begin{proof}
	The proof follows a similar procedure as in that of Lemma \ref{le:Preliminary-1} by repeating using the weighted Sobolev inequality \eqref{ineq:weighted Sobolev}.
\end{proof}

\subsection{The a priori assumption.}\label{The a priori assumption} Let \(c_1\) be the Sobolev embedding \(H^1(I)\hookrightarrow L^\infty(I)\) constant, and \(c_2\) be the constant in the inequality \eqref{Preliminary-1}. Set \(M_1=2M_0\).
Let \((v,\eta)\) satisfy \eqref{fluid particle} and \eqref{Lagrangian variable}. Assume that there exists some suitably small \(T\in(0,1/(2c_1c_2\sqrt{M_1})]\cap (0,1)\) such that 
\begin{align}\label{a priori assumption}
	\sup_{0\leq t\leq T}E(t,v)\leq M_1.
\end{align}
Then one has 
\begin{align}\label{eta-bound}
	1/2\leq \eta_x(x,t)\leq 3/2, \quad (x,t)\in I\times [0,T].
\end{align}
Indeed, it follows from \eqref{fluid particle} that
\begin{align*}
	\eta(x,t)=x+\int_0^tv(x,s)\,\diff s,\quad (x,t)\in I\times [0,T],
\end{align*}
which leads to 
\begin{equation*}
\begin{aligned}
	|\eta_x(x,t)-1|&\leq\int_0^t\|v_x(\cdot,s)\|_{L^\infty(I)}\,\diff s\leq T\sup_{0\leq t\leq T}\|v_x(\cdot,t)\|_{L^\infty(I)}\\
	&\leq c_1T\sup_{0\leq t\leq T}\|v_x(\cdot,t)\|_{H^1(I)}\leq c_1c_2T\sup_{0\leq t\leq T}E^{1/2}(t,v)\\
	&\leq c_1c_2\sqrt{M_1}T\leq 1/2, \quad (x,t)\in I\times [0,T].
\end{aligned}
\end{equation*}
Hence \eqref{eta-bound} follows.\\

\begin{remark}\label{determin M_1 and T} The a priori assumption \eqref{a priori assumption} will be closed by the a priori bound \eqref{I-Priori-ellip-33}.     
\end{remark}

\section{Energy Estimates}\label{Energy Estimates}

This section is devoted to deducing some basic energy estimates on time-derivatives. Let \((v,\eta)\) be a solution to the problem \eqref{eq:main-2} satisfying \eqref{a priori assumption}.\\

\noindent{\bf{Estimate of \(\sum_{k=0}^3\|\sqrt{\rho_0}\partial_t^kv\|_{L^2(I)}\).}}
We first estimate \(\|\sqrt{\rho_0}\partial_t^3v\|_{L^2(I)}\). To this end,
one can apply \(\partial_t^3\) to Equation \(\eqref{eq:main-2}_1\), multiplying it by \(\partial_t^3v\), after some elementary computations, to obtain that
\begin{equation}\label{II-Priori-time-1}
\begin{aligned}
&\frac{1}{2}\int_I \rho_0(\partial_t^3v)^2\,\diff x
+\int_0^t\int_I\frac{\rho_0(\partial_t^3v_x)^2}{\eta_x^2}\,\diff x\diff s\\
&\quad=\frac{1}{2}\int_I \rho_0(\partial_t^3v)^2(x,0)\,\diff x+\int_0^t\int_I\partial_t^3\bigg({\frac{\rho_0^2}{\eta_x^2}}\bigg)\partial_t^3v_x\,\diff x\diff s\\
&\qquad-\int_0^t\int_I\bigg[\partial_t^3\bigg(\frac{\rho_0v_x}{\eta_x^2}\bigg)\partial_t^3v_x
-\frac{\rho_0(\partial_t^3v_x)^2}{\eta_x^2}\bigg]\,\diff x\diff s.
\end{aligned}
\end{equation}

Using \eqref{eta-bound},  one finds that
\begin{equation}\label{Add-1}
\begin{aligned}
\bigg|\partial_t^3\bigg(\frac{1}{\eta_x^2}\bigg)\bigg|\lesssim |\partial_t^2v_x|+|v_x\partial_tv_x|+|v_x|^3,
\end{aligned}
\end{equation}
and
\begin{equation}\label{Add-2}
\begin{aligned}
\bigg|\partial_t^3\bigg(\frac{v_x}{\eta_x^2}\bigg)\partial_t^3v_x
-\frac{(\partial_t^3v_x)^2}{\eta_x^2}\bigg|\lesssim \big[|v_x\partial_t^2v_x|+|\partial_tv_x|(v_x^2+|\partial_tv_x|)+|v_x|^4\big]|\partial_t^3v_x|.
\end{aligned}
\end{equation}
Then one may use Cauchy's inequality to get
\begin{equation}\label{II-Priori-time-2}
\begin{aligned}
&\bigg|\int_0^t\int_I\partial_t^3\bigg({\frac{\rho_0^2}{\eta_x^2}}\bigg)\partial_t^3v_x\,\diff x\diff s\bigg|\\
&\quad\leq \frac{1}{100}\int_0^t\int_I\rho_0(\partial_t^3v_x)^2\,\diff x\diff s
+C\int_0^t\int_I\rho_0(\partial_t^2v_x)^2\,\diff x\diff s\\
&\qquad+C\int_0^t\|v_x\|_{L^\infty}^2\int_I\rho_0(\partial_tv_x)^2\,\diff x\diff s
+C\int_0^t\|v_x\|_{L^\infty}^4\int_I\rho_0v_x^2\,\diff x\diff s\\
&\quad\leq \frac{1}{100}\int_0^t\int_I\rho_0(\partial_t^3v_x)^2\,\diff x\diff s+Ct P(\sup_{0\leq s\leq t}E^{1/2}(s,v)),
\end{aligned}
\end{equation}
and
\begin{equation}\label{II-Priori-time-3}
\begin{aligned}
&\bigg|\int_0^t\int_I\bigg[\partial_t^3(\frac{\rho_0v_x}{\eta_x^2})\partial_t^3v_x
-{\frac{\rho_0(\partial_t^3v_x)^2}{\eta_x^2}}\bigg]\,\diff x\diff s\bigg|\\
&\quad\leq \frac{1}{100}\int_0^t\int_I\rho_0(\partial_t^3v_x)^2\,\diff x\diff s
+C\int_0^t\|v_x\|_{L^\infty}^2\int_I\rho_0(\partial_t^2v_x)^2\,\diff x\diff s\\
&\qquad+C\int_0^t\|v_x\|_{L^\infty}^4\int_I\rho_0(\partial_tv_x)^2\,\diff x\diff s
+C\int_0^t\|\partial_tv_x\|_{L^\infty}^2\int_I\rho_0(\partial_tv_x)^2\,\diff x\diff s\\
&\qquad+C\int_0^t\|v_x\|_{L^\infty}^6\int_I\rho_0v_x^2\,\diff x\diff s\\
&\quad\leq \frac{1}{100}\int_0^t\int_I\rho_0(\partial_t^3v_x)^2\,\diff x\diff s+Ct P(\sup_{0\leq s\leq t}E^{1/2}(s,v)),
\end{aligned}
\end{equation}
where \eqref{Preliminary-3} was used in \eqref{II-Priori-time-2} and \eqref{II-Priori-time-3}, while,
\begin{equation*}
\begin{aligned}
\|\partial_tv_x\|_{L^\infty}&\lesssim \|\partial_tv_x\|_{L^2}+\|\partial_tv_{xx}\|_{L^2}\\
&\lesssim \|\rho_0\partial_tv_x\|_{L^2}+\|\rho_0\partial_tv_{xx}\|_{L^2}+\|\rho_0\partial_t\partial_x^3v\|_{L^2}\\
&\lesssim E^{1/2}(s,v)+\|\rho_0^2\partial_t\partial_x^3v\|_{L^2}+\|\rho_0^2\partial_t\partial_x^4v\|_{L^2}
\lesssim E^{1/2}(s,v)
\end{aligned}
\end{equation*}
was used in \eqref{II-Priori-time-3}, here the weighted Sobolev inequality \eqref{ineq:weighted Sobolev} was utilized. Here and thereafter \(P(\cdot)\) denotes a generic polynomial function of its arguments.

Due to the bound \eqref{eta-bound}, and noting that the term \(\int_0^t\int_I\frac{\rho_0(\partial_t^3v_x)^2}{\eta_x^2}\,\diff x\diff s\) on the left hand side (which will be abbreviated as LHS from now on) of \eqref{II-Priori-time-1} is bounded from below by \(\frac{4}{9}\int_0^t\int_I\rho_0(\partial_t^3v_x)^2\,\diff x\diff s\),
hence one inserts \eqref{II-Priori-time-2} and \eqref{II-Priori-time-3} into \eqref{II-Priori-time-1} to obtain
\begin{equation}\label{II-Priori-time-4}
\begin{aligned}
\int_I \rho_0(\partial_t^3v)^2\,\diff x
+\int_0^t\int_I\rho_0(\partial_t^3v_x)^2\,\diff x\diff s
\leq M_0+Ct P(\sup_{0\leq s\leq t}E^{1/2}(s,v)).
\end{aligned}
\end{equation}

Next, we estimate \(\|\sqrt{\rho_0}\partial_t^2v\|_{L^2(I)}\). Since
\begin{equation*}
\begin{aligned}
\partial_t^2v(x,t)=\partial_t^2v(x,0)+\int_0^t\partial_t^3v(x,s)\,\diff s,
\end{aligned}
\end{equation*}
it then follows from Cauchy's inequality and Fubini's theorem that
\begin{equation}\label{I-Priori-time-12}
\begin{aligned}
\int_I \rho_0(\partial_t^2v)^2\,\diff x
&\lesssim \int_I \rho_0(\partial_t^2v)^2(x,0)\,\diff x
+t\int_0^t\int_I \rho_0(\partial_t^3v)^2\,\diff x\diff s\\
&\leq M_0+Ct P(\sup_{0\leq s\leq t}E^{1/2}(s,v)),
\end{aligned}
\end{equation}
where \eqref{II-Priori-time-4} has been used in the last line. Similarly, by \eqref{I-Priori-time-12}, one can get
\begin{equation}\label{I-Priori-time-8}
\begin{aligned}
\int_I \rho_0(\partial_tv)^2\,\diff x
\leq M_0+Ct P(\sup_{0\leq s\leq t}E^{1/2}(s,v)),
\end{aligned}
\end{equation}
and
\begin{equation}\label{I-Priori-time-3}
\begin{aligned}
\int_I \rho_0v^2\,\diff x
\leq M_0+Ct P(\sup_{0\leq s\leq t}E^{1/2}(s,v)).
\end{aligned}
\end{equation}

\bigskip
\noindent{\bf{Estimate of \(\sum_{k=0}^2\|\sqrt{\rho_0}\partial_t^kv_x\|_{L^2(I)}\).}}
We start with \(\|\sqrt{\rho_0}\partial_t^2v_x\|_{L^2(I)}\).
Applying \(\partial_t^2\) to Equation \(\eqref{eq:main-2}_1\),
and multiplying it by \(\partial_t^3v\), one gets by some direct calculations that
\begin{equation}\label{II-Priori-time-5}
\begin{aligned}
&\int_0^t\int_I \rho_0(\partial_t^3v)^2\,\diff x\diff s
+\frac{1}{2}\int_I\frac{\rho_0(\partial_t^2v_x)^2}{\eta_x^2}\,\diff x\\
&\quad=\frac{1}{2}\int_I\rho_0(\partial_t^2v_x)^2(x,0)\,\diff x-2\int_0^t\int_I\bigg(-3\frac{\rho_0^2v_x^2}{\eta_x^4}
+\frac{\rho_0^2\partial_tv_x}{\eta_x^3}\bigg)\partial_t^3v_x\,\diff x\diff s\\
&\qquad-\int_0^t\int_I\frac{\rho_0v_x(\partial_t^2v_x)^2}{\eta_x^3}\,\diff x\diff s
-6\int_0^t\int_I\frac{\rho_0v_x^3\partial_t^3v_x}{\eta_x^4}\,\diff x\diff s \\
&\qquad +6\int_0^t\int_I\frac{\rho_0v_x\partial_tv_x\partial_t^3v_x}{\eta_x^3}\,\diff x\diff s.
\end{aligned}
\end{equation}
The above three terms on the right hand side (which will be abbreviated as RHS from now on) of \eqref{II-Priori-time-5}  can be estimated as follows:
\begin{equation}\label{II-Priori-time-6}
\begin{aligned}
&\bigg|\int_0^t\int_I\bigg(-3\frac{\rho_0^2v_x^2}{\eta_x^4}
+\frac{\rho_0^2\partial_tv_x}{\eta_x^3}\bigg)\partial_t^3v_x\,\diff x\diff s\bigg|\\
&\quad\lesssim \int_0^t\int_I\rho_0(\partial_t^3v_x)^2\,\diff x\diff s+\int_0^t\|v_x\|_{L^\infty}^2\int_I\rho_0v_x^2\,\diff x\diff s\\
&\qquad+\int_0^t\int_I\rho_0(\partial_tv_x)^2\,\diff x\diff s\\
&\quad\leq M_0+Ct P(\sup_{0\leq s\leq t}E^{1/2}(s,v)),
\end{aligned}
\end{equation}
\begin{equation}\label{II-Priori-time-7}
\begin{aligned}
\bigg|\int_0^t\int_I\frac{\rho_0v_x(\partial_t^2v_x)^2}{\eta_x^3}\,\diff x\diff s\bigg|
&\lesssim \int_0^t\|v_x\|_{L^\infty}\int_I\rho_0(\partial_t^2v_x)^2\,\diff x\diff s\\
&\leq Ct P(\sup_{0\leq s\leq t}E^{1/2}(s,v)),
\end{aligned}
\end{equation}
\begin{equation}\label{II-Priori-time-8}
\begin{aligned}
\bigg|\int_0^t\int_I\frac{\rho_0v_x^3\partial_t^3v_x}{\eta_x^4}\,\diff x\diff s\bigg|
&\lesssim \int_0^t\int_I\rho_0(\partial_t^3v_x)^2\,\diff x\diff s\\
&\quad+\int_0^t\|v_x\|_{L^\infty}^4\int_I\rho_0v_x^2\,\diff x\diff s\\
&\leq M_0+Ct P(\sup_{0\leq s\leq t}E^{1/2}(s,v)),
\end{aligned}
\end{equation}
and
\begin{equation}\label{II-Priori-time-9}
\begin{aligned}
\bigg|\int_0^t\int_I\frac{\rho_0v_x\partial_tv_x\partial_t^3v_x}{\eta_x^3}\,\diff x\diff s\bigg|
&\lesssim \int_0^t\int_I\rho_0(\partial_t^3v_x)^2\,\diff x\diff s\\
&\quad+\int_0^t\|v_x\|_{L^\infty}^2\int_I\rho_0(\partial_tv_x)^2\,\diff x\diff s\\
&\leq M_0+Ct P(\sup_{0\leq s\leq t}E^{1/2}(s,v)).
\end{aligned}
\end{equation}
Here \eqref{Preliminary-3} has been used in the last line of \eqref{II-Priori-time-6}- \eqref{II-Priori-time-9}.

Hence substituting \eqref{II-Priori-time-6}-\eqref{II-Priori-time-9} into \eqref{II-Priori-time-5} yields
\begin{equation}\label{II-Priori-time-10}
\begin{aligned}
\int_0^t\int_I \rho_0(\partial_t^3v)^2\,\diff x\diff s
+\int_I\rho_0(\partial_t^2v_x)^2\,\diff x
\leq M_0+Ct P(\sup_{0\leq s\leq t}E^{1/2}(s,v)).
\end{aligned}
\end{equation}

We now consider \(\|\sqrt{\rho_0}\partial_tv_x\|_{L^2(I)}\). Since
\begin{equation*}
\begin{aligned}
\partial_tv_x(x,t)=\partial_tv_x(x,0)+\int_0^t\partial_t^2v_x(x,s)\,\diff s,
\end{aligned}
\end{equation*}
it then follows from Cauchy's inequality and Fubini's theorem that
\begin{equation}\label{I-Priori-time-21}
\begin{aligned}
\int_I \rho_0(\partial_tv_x)^2\,\diff x
&\lesssim \int_I \rho_0(\partial_tv_x)^2(x,0)\,\diff x
+t\int_0^t\int_I \rho_0(\partial_t^2v_x)^2\,\diff x\diff s\\
&\leq M_0+Ct P(\sup_{0\leq s\leq t}E^{1/2}(s,v))\quad \text{for\ small}\ t>0,
\end{aligned}
\end{equation}
where \eqref{II-Priori-time-10} has been used in the last line.
In view of \eqref{I-Priori-time-21}, one can derive similarly that
\begin{equation}\label{I-Priori-time-16}
\begin{aligned}
\int_I\rho_0v_x^2\,\diff x
\leq M_0+Ct P(\sup_{0\leq s\leq t}E^{1/2}(s,v)).
\end{aligned}
\end{equation}

\section{Elliptic estimates}\label{Elliptic estimates}

Having the estimates on time-derivatives in Section \ref{Energy Estimates}, we will use the elliptic theory to gain the spatial regularity of the solutions in this section.\\

\noindent{\bf{Estimate of \(\|\rho_0v_{xx}\|_{L^2(I)}\).}}
 It follows from Equation \(\eqref{eq:main-2}_1\) that
\begin{equation}\label{I-Priori-ellip-1}
\begin{aligned}
\bigg\|\bigg({\frac{\rho_0v_x}{\eta_x^2}}\bigg)_{x}\bigg\|_{L^2}^2
&\leq \|\rho_0\partial_tv\|_{L^2}^2+\bigg\|\bigg({\frac{\rho_0^2}{\eta_x^2}}\bigg)_x\bigg\|_{L^2}^2\\
&\leq M_0+Ct P(\sup_{0\leq s\leq t}E^{1/2}(s,v)),
\end{aligned}
\end{equation}
in which \(\|\rho_0\partial_tv\|_{L^2}\) is bounded by \eqref{I-Priori-time-8}, and the bound on \(\big\|\big({\frac{\rho_0^2}{\eta_x^2}}\big)_x\big\|_{L^2}\) relies on
\begin{equation*}
\begin{aligned}
\bigg|\bigg({\frac{\rho_0^2}{\eta_x^2}}\bigg)_{x}\bigg|\lesssim 1+\rho_0|\eta_{xx}|,
\end{aligned}
\end{equation*}
and hence
\begin{equation}\label{I-Priori-ellip-2}
\begin{aligned}
\bigg\|\bigg({\frac{\rho_0^2}{\eta_x^2}}\bigg)_x\bigg\|_{L^2}
&\lesssim 1+\|\eta_{xx}\|_{L^2}\lesssim 1+t P(\sup_{0\leq s\leq t}E^{1/2}(s,v)).
\end{aligned}
\end{equation}
where one has used \eqref{Preliminary-2}.

Next, we estimate \(\|\rho_0v_{xx}\|_{L^2(I)}\). Note that
\begin{equation}\label{I-Priori-ellip-2.2}
\begin{aligned}
\rho_0\eta_x^{-2}v_{xx}+(\rho_0)_x\eta_x^{-2}v_{x}
=\bigg(\frac{\rho_0v_x}{\eta_x^2}\bigg)_{x}-
\rho_0(\eta_x^{-2})_{x}v_x.
\end{aligned}
\end{equation}
The last term in \eqref{I-Priori-ellip-2.2} can be estimated as follows:
\begin{equation}\label{I-Priori-ellip-2.4}
\begin{aligned}
\|\rho_0(\eta_x^{-2})_{x}v_x\|_{L^2}
\lesssim \|v_x\|_{L^2}\|\eta_{xx}\|_{L^\infty}
\leq Ct P(\sup_{0\leq s\leq t}E^{1/2}(s,v)),
\end{aligned}
\end{equation}
where \eqref{Preliminary-1} and \eqref{Preliminary-5} were used.
We then insert \eqref{I-Priori-ellip-1} and \eqref{I-Priori-ellip-2.4} into \eqref{I-Priori-ellip-2.2} to get
\begin{equation}\label{I-Priori-ellip-3}
\begin{aligned}
\|\rho_0\eta_x^{-2}v_{xx}+(\rho_0)_x\eta_x^{-2}v_{x}\|_{L^2}^2
\leq M_0+CtP(\sup_{0\leq s\leq t}E^{1/2}(s,v)).
\end{aligned}
\end{equation}
Integration by parts yields
\begin{equation}\label{I-Priori-ellip-4}
\begin{aligned}
&\|\rho_0\eta_x^{-2}v_{xx}\|_{L^2}^2\\
&\quad= \|\rho_0\eta_x^{-2}v_{xx}+(\rho_0)_x\eta_x^{-2}v_{x}\|_{L^2}^2\\
&\qquad-\|(\rho_0)_x\eta_x^{-2}v_{x}\|_{L^2}^2
-\int_I\rho_0(\rho_0)_x\eta_x^{-4}(v_x^2)_x\,\diff x\\
&\quad=\|\rho_0\eta_x^{-2}v_{xx}+(\rho_0)_x\eta_x^{-2}v_{x}\|_{L^2}^2+
\int_I\rho_0[(\rho_0)_x\eta_x^{-4}]_{x}v_x^2\,\diff x\\
&\quad\leq M_0+Ct P(\sup_{0\leq s\leq t}E^{1/2}(s,v)),
\end{aligned}
\end{equation}
where one has used \eqref{I-Priori-ellip-3} and the estimate
\begin{equation}\label{additial estimate}
\begin{aligned}
&\bigg|\int_I\rho_0[(\rho_0)_x\eta_x^{-4}]_xv_x^2\,\diff x\bigg|\\
&\quad\lesssim \bigg|\int_I\rho_0(\rho_0)_{xx}\eta_x^{-4}v_x^2\,\diff x\bigg|
+\bigg|\int_I\rho_0(\rho_0)_x\eta_x^{-5}\eta_{xx}v_x^2\,\diff x\bigg|\\
&\quad\lesssim
(1+\|\eta_{xx}\|_{L^\infty})\int_I\rho_0v_x^2\,\diff x\\
&\quad\leq M_0+Ct P(\sup_{0\leq s\leq t}E^{1/2}(s,v)).
\end{aligned}
\end{equation}
Here \eqref{Preliminary-5} and \eqref{I-Priori-time-16} have been used.
It follows from \eqref{eta-bound} and \eqref{I-Priori-ellip-4} that
\begin{equation}\label{I-Priori-ellip-5}
\begin{aligned}
\|\rho_0v_{xx}\|_{L^2}^2\leq M_0+Ct P(\sup_{0\leq s\leq t}E^{1/2}(s,v)).
\end{aligned}
\end{equation}

\bigskip
\noindent{\bf{Estimate of \(\|\rho_0^{3/2}\partial_x^3v\|_{L^2(I)}\).}}
First, it follows from \eqref{ineq:weighted Sobolev}, \eqref{I-Priori-time-8} and \eqref{I-Priori-time-21} that
\begin{equation}\label{I-Priori-ellip-6}
\begin{aligned}
\|(\rho_0\partial_tv)_x\|_{L^2}^2
&\lesssim \|\partial_tv\|_{L^2}^2+\|\rho_0\partial_tv_x\|_{L^2}^2\\
&\lesssim \|\rho_0\partial_tv\|_{L^2}^2+\|\rho_0\partial_tv_x\|_{L^2}^2\\
&\leq M_0+Ct P(\sup_{0\leq s\leq t}E^{1/2}(s,v)).
\end{aligned}
\end{equation}
Since
\begin{equation*}
\begin{aligned}
\bigg|\bigg({\frac{\rho_0^2}{\eta_x^2}}\bigg)_{xx}\bigg|\lesssim 1+\rho_0|\eta_{xx}|+\rho_0^2(|\eta_{xx}|^2+|\partial_x^3\eta|),
\end{aligned}
\end{equation*}
one may estimate by Lemma \ref{le:Preliminary-1} that
\footnote{We can throw away the weight \(\rho_0\) in \(\|\rho_0\eta_{xx}\|_{L^2}\), \(\|\rho_0\eta_{xx}\|_{L^\infty}\), \(\|\rho_0^2\eta_{xxx}\|_{L^2}\) and similar terms later on since we work with the energy functional \(E(t,v)\), see the difference when one works with a lower-order energy functional in Subsection \ref{lower-order energy function}.}
\begin{equation}\label{I-Priori-ellip-7}
\begin{aligned}
\bigg\|\bigg({\frac{\rho_0^2}{\eta_x^2}}\bigg)_{xx}\bigg\|_{L^2}
&\lesssim 1+\|\eta_{xx}\|_{L^2}+\|\eta_{xx}\|_{L^\infty}\|\eta_{xx}\|_{L^2}+\|\partial_x^3\eta\|_{L^2}\\
&\leq 1+Ct P(\sup_{0\leq s\leq t}E^{1/2}(s,v)).
\end{aligned}
\end{equation}
It then follows from \eqref{I-Priori-ellip-6} and \eqref{I-Priori-ellip-7} that
\begin{equation}\label{I-Priori-ellip-10}
\begin{aligned}
\bigg\|\bigg({\frac{\rho_0v_x}{\eta_x^2}}\bigg)_{xx}\bigg\|_{L^2}^2
&\leq \|(\rho_0\partial_tv)_x\|_{L^2}^2+\bigg\|\bigg({\frac{\rho_0^2}{\eta_x^2}}\bigg)_{xx}\bigg\|_{L^2}^2\\
&\leq M_0+Ct P(\sup_{0\leq s\leq t}E^{1/2}(s,v)).
\end{aligned}
\end{equation}

To estimate \(\|\rho_0^{3/2}\partial_x^3v\|_{L^2(I)}\), we first write
\begin{equation}\label{I-Priori-ellip-11}
\begin{aligned}
\rho_0\eta_x^{-2}\partial_x^3v+2(\rho_0)_x\eta_x^{-2}v_{xx}=\bigg({\frac{\rho_0v_x}{\eta_x^2}}\bigg)_{xx}
-2\rho_0(\eta_x^{-2})_{x}v_{xx}-(\rho_0\eta_x^{-2})_{xx}v_x.
\end{aligned}
\end{equation}
Considering the second term on the RHS of \eqref{I-Priori-ellip-11}, one use \eqref{Preliminary-1} and \eqref{Preliminary-5} to estimate
\begin{equation}\label{I-Priori-ellip-12}
\begin{aligned}
\|\rho_0(\eta_x^{-2})_{x}v_{xx}\|_{L^\infty}
\leq \|v_{xx}\|_{L^2}\|\eta_{xx}\|_{L^\infty}
\leq Ct P(\sup_{0\leq s\leq t}E^{1/2}(s,v)).
\end{aligned}
\end{equation}
Since
\begin{equation}\label{I-Priori-ellip-13}
\begin{aligned}
|(\rho_0\eta_x^{-2})_{xx}|\lesssim 1+|\eta_{xx}|+\rho_0(|\eta_{xx}|^2+|\partial_x^3\eta|),
\end{aligned}
\end{equation}
the last term on the RHS of \eqref{I-Priori-ellip-11} may be estimated as follows:
\begin{equation}\label{I-Priori-ellip-14}
\begin{aligned}
\|(\rho_0\eta_x^{-2})_{xx}v_x\|_{L^2}
&\lesssim \|v_x\|_{L^2}+
\|v_x\|_{L^\infty}(\|\eta_{xx}\|_{L^2}\\
&\quad+\|\eta_{xx}\|_{L^2}\|\eta_{xx}\|_{L^\infty}
+\|\partial_x^3\eta\|_{L^2})\\
&\leq [M_0+Ct P(\sup_{0\leq s\leq t}E^{1/2}(s,v))]^{1/2}.
\end{aligned}
\end{equation}
In the last line of \eqref{I-Priori-ellip-14}, one has used \eqref{Preliminary-2}, \eqref{Preliminary-3}, \eqref{Preliminary-5} and the estimate
\begin{equation}\label{I-Priori-ellip-15}
\begin{aligned}
\|v_x\|_{L^2}
&\lesssim \|\rho_0v_x\|_{L^2}+\|\rho_0v_{xx}\|_{L^2}\\
&\leq [M_0+Ct P(\sup_{0\leq s\leq t}E^{1/2}(s,v))]^{1/2},
\end{aligned}
\end{equation}
which follows from \eqref{I-Priori-time-16} and \eqref{I-Priori-ellip-5}.
Inserting \eqref{I-Priori-ellip-10}, \eqref{I-Priori-ellip-12} and \eqref{I-Priori-ellip-14} into \eqref{I-Priori-ellip-11} yields
\begin{equation}\label{I-Priori-ellip-16}
\begin{aligned}
\|\rho_0\eta_x^{-2}\partial_x^3v+2(\rho_0)_x\eta_x^{-2}v_{xx}\|_{L^2}^2\leq M_0+Ct P(\sup_{0\leq s\leq t}E^{1/2}(s,v)).
\end{aligned}
\end{equation}
We then compensate a weight \(\rho_0^{1/2}\) and integrate by parts to deduce that
\begin{equation}\label{I-Priori-ellip-17}
\begin{aligned}
&\|\rho_0^{3/2}\eta_x^{-2}\partial_x^3v\|_L^2\\
&\quad= \|\rho_0^{3/2}\eta_x^{-2}\partial_x^3v+2\rho_0^{1/2}(\rho_0)_x\eta_x^{-2}v_{xx}\|_{L^2}^2
-4\|\rho_0^{1/2}(\rho_0)_x\eta_x^{-2}v_{xx}\|_{L^2}^2\\
&\qquad-2\int_I\rho_0^2(\rho_0)_x\eta_x^{-4}[(v_{xx})^2]_x\,\diff x\\
&\quad= \|\rho_0^{3/2}\eta_x^{-2}\partial_x^3v+2\rho_0^{1/2}(\rho_0)_x\eta_x^{-2}v_{xx}\|_{L^2}^2
+\int_I\rho_0^2[(\rho_0)_x\eta_x^{-4}]_{x}v_{xx}^2\,\diff x\\
&\quad\lesssim  \|\rho_0^{3/2}\eta_x^{-2}\partial_x^3v+2\rho_0^{1/2}(\rho_0)_x\eta_x^{-2}v_{xx}\|_{L^2}^2
+(1+\|\eta_{xx}\|_{L^\infty})\int_I\rho_0^2v_{xx}^2\,\diff x\\
&\quad\leq M_0+Ct P(\sup_{0\leq s\leq t}E^{1/2}(s,v)),
\end{aligned}
\end{equation}
where \eqref{Preliminary-5}, \eqref{I-Priori-ellip-5} and \eqref{I-Priori-ellip-16} have been used.
The inequality \eqref{I-Priori-ellip-17} and the bound \eqref{eta-bound} give
\begin{equation}\label{I-Priori-ellip-third}
\begin{aligned}
\|\rho_0^{3/2}\partial_x^3v\|_{L^2}^2\leq M_0+Ct P(\sup_{0\leq s\leq t}E^{1/2}(s,v)).
\end{aligned}
\end{equation}

\bigskip
\noindent{\bf{Estimate of \(\|\rho_0\partial_tv_{xx}\|_{L^2(I)}\).}} We first claim that
\begin{equation}\label{I-Priori-ellip-claim}
\begin{aligned}
\bigg\|\partial_t\bigg(\frac{\rho_0v_x}{\eta_x^2}\bigg)_x\bigg\|_{L^2}^2
\leq M_0+Ct P(\sup_{0\leq s\leq t}E^{1/2}(s,v)).
\end{aligned}
\end{equation}
To verify \eqref{I-Priori-ellip-claim}, we note that
\begin{equation}\label{I-Priori-ellip-20}
\begin{aligned}
\partial_t\bigg(\frac{\rho_0v_x}{\eta_x^2}\bigg)_x=\rho_0\partial_t^2v
+\partial_t\bigg(\frac{\rho_0^2}{\eta_x^2}\bigg)_x.
\end{aligned}
\end{equation}
Since
\begin{equation*}
\begin{aligned}
\bigg|\partial_t\bigg({\frac{\rho_0^2}{\eta_x^2}}\bigg)_x\bigg|
\lesssim \rho_0|v_x|+\rho_0^2(|v_x\eta_{xx}|+|v_{xx}|),
\end{aligned}
\end{equation*}
one obtains
\begin{equation}\label{I-Priori-ellip-21}
\begin{aligned}
\bigg\|\partial_t\bigg({\frac{\rho_0^2}{\eta_x^2}}\bigg)_x\bigg\|_{L^2}
&\lesssim \|\rho_0v_x\|_{L^2}+\|v_x\|_{L^\infty}\|\eta_{xx}\|_{L^2}+\|\rho_0v_{xx}\|_{L^2}\\
&\leq [M_0+Ct P(\sup_{0\leq s\leq t}E^{1/2}(s,v))]^{1/2},
\end{aligned}
\end{equation}
where one has used \eqref{Preliminary-2}, \eqref{Preliminary-3}, \eqref{I-Priori-time-16} and \eqref{I-Priori-ellip-5} in the last inequality. Then \eqref{I-Priori-ellip-claim} follows from \eqref{I-Priori-ellip-20}, \eqref{I-Priori-ellip-21} and \eqref{I-Priori-time-12}.

Direct calculations give
\begin{equation}\label{I-Priori-ellip-21.5}
\begin{aligned}
(\rho_0\partial_tv_x)_x
&=\partial_t\bigg({\frac{\rho_0v_x}{\eta_x^2}}\bigg)_x\eta_x^2+2\bigg({\frac{\rho_0v_x}{\eta_x^2}}\bigg)_x
\eta_xv_x\\
&\quad+2\partial_t\bigg(\frac{\rho_0v_x}{\eta_x^2}\bigg)\eta_{x}\eta_{xx}
+2\frac{\rho_0v_x}{\eta_x^2}(v_x\eta_{xx}+\eta_{x}v_{xx}).
\end{aligned}
\end{equation}
It follows from \eqref{I-Priori-ellip-claim} and \eqref{I-Priori-ellip-1} that the \(L^2-\) norm of the first two terms on the RHS of \eqref{I-Priori-ellip-21.5} has the desired bound. It suffices to handle the last two terms on the RHS of \eqref{I-Priori-ellip-21.5}. Considering the third term on the RHS of \eqref{I-Priori-ellip-21.5}, by \(H^1(I)\hookrightarrow L^\infty(I)\), one has
\begin{equation*}
\begin{aligned}
\bigg\|\partial_t\bigg({\frac{\rho_0v_x}{\eta_x^2}}\bigg)\eta_{x}\eta_{xx}\bigg\|_{L^2}&\lesssim
\|\eta_{xx}\|_{L^2}\left(\bigg\|\partial_t\bigg(\frac{\rho_0v_x}{\eta_x^2}\bigg)\bigg\|_{L^2}
+\bigg\|\partial_t\bigg(\frac{\rho_0v_x}{\eta_x^2}\bigg)_x\bigg\|_{L^2}\right)\\
&\leq Ct P(\sup_{0\leq s\leq t}E^{1/2}(s,v)),
\end{aligned}
\end{equation*}
where in the last line one has used \eqref{I-Priori-ellip-claim} and the estimate
\begin{equation}\label{I-Priori-ellip-21.6}
\begin{aligned}
\bigg\|\partial_t\bigg(\frac{\rho_0v_x}{\eta_x^2}\bigg)\bigg\|_{L^2}
&\lesssim \|\rho_0\partial_tv_x\|_{L^2}+\|v_x\|_{L^\infty}\|\rho_0v_x\|_{L^2}\\
&\leq M_0+C(t+1) P(\sup_{0\leq s\leq t}E^{1/2}(s,v)),
\end{aligned}
\end{equation}
which together with \(\|\eta_{xx}\|_{L^2}\) yields the bound \(Ct P(\sup_{0\leq s\leq t}E^{1/2}(s,v))\) since \(\|\eta_{xx}\|_{L^2}\) contributes a factor \(t\) due to \eqref{Preliminary-2}.
In the last term on the RHS of \eqref{I-Priori-ellip-21.5}, the \(L^2-\) norm of the first part is bounded by \(\|v_x\|_{L^\infty}^2\|\eta_{xx}\|_{L^2}\) which contributes the bound \(Ct P(\sup_{0\leq s\leq t}E^{1/2}(s,v))\),  and the second part can be estimated by \eqref{Sobolev inequaty} as follows:
\begin{equation*}
\begin{aligned}
\bigg\|\frac{\rho_0v_x}{\eta_x^2}\eta_{x}v_{xx}\bigg\|_{L^2}
&\lesssim \|v_x\|_{L^4}\|\rho_0v_{xx}\|_{L^4}\lesssim \|v_x\|_{H^{1/2}}\|\rho_0v_{xx}\|_{H^{1/2}}\\
&\leq M_0+Ct P(\sup_{0\leq s\leq t}E^{1/2}(s,v)),
\end{aligned}
\end{equation*}
since each factor in the second inequality enjoys the same bound \([M_0+Ct P(\sup_{0\leq s\leq t}E^{1/2}(s,v))]^{1/2}\). Indeed, one can apply \eqref{ineq:weighted Sobolev-0}, \eqref{ineq:weighted Sobolev}, \eqref{I-Priori-time-16}, \eqref{I-Priori-ellip-5} and \eqref{I-Priori-ellip-third} to deduce
\begin{equation*}
\begin{aligned}
\|v_x\|_{H^{1/2}}&\lesssim \|\rho_0^{1/2}v_x\|_{L^2}+\|\rho_0^{1/2}v_{xx}\|_{L^2}\\
&\lesssim \|\rho_0^{1/2}v_x\|_{L^2}+(\|\rho_0^{3/2}v_{xx}\|_{L^2}+\|\rho_0^{3/2}v_{xxx}\|_{L^2})\\
&\leq [M_0+Ct P(\sup_{0\leq s\leq t}E^{1/2}(s,v))]^{1/2},
\end{aligned}
\end{equation*}
and similarly,
\begin{equation*}
\begin{aligned}
\|\rho_0v_{xx}\|_{H^{1/2}}&\lesssim \|\rho_0^{1/2}(\rho_0v_{xx})\|_{L^2}+\|\rho_0^{1/2}(\rho_0v_{xx})_x\|_{L^2}\\
&\lesssim \|\rho_0^{1/2}v_{xx}\|_{L^2}+\|\rho_0^{3/2}v_{xxx}\|_{L^2}\\
&\lesssim \|\rho_0^{3/2}v_{xx}\|_{L^2}+\|\rho_0^{3/2}v_{xxx}\|_{L^2}\\
&\leq [M_0+Ct P(\sup_{0\leq s\leq t}E^{1/2}(s,v))]^{1/2}.
\end{aligned}
\end{equation*}

Taking all the cases into account and noticing 
\begin{equation*}
	\begin{aligned}
	(\rho_0\partial_tv_x)_x=\rho_0\partial_tv_{xx}+(\rho_0)_x\partial_tv_{x},	
	\end{aligned}
\end{equation*}
one obtains
\begin{equation}\label{I-Priori-ellip-21.7}
\begin{aligned}
\|\rho_0\partial_tv_{xx}+(\rho_0)_x\partial_tv_{x}\|_{L^2}^2
\leq M_0+Ct P(\sup_{0\leq s\leq t}E^{1/2}(s,v)).
\end{aligned}
\end{equation}
Then integration by parts yields
\begin{equation}\label{I-Priori-ellip-21.8}
	\begin{aligned}
		\|\rho_0\partial_tv_{xx}\|_{L^2}^2
		&= \|\rho_0\partial_tv_{xx}+(\rho_0)_x\partial_tv_{x}\|_{L^2}^2
		-\|(\rho_0)_x\partial_tv_{x}\|_{L^2}^2\\
		&\quad-\int_I\rho_0(\rho_0)_x[(\partial_tv_{x})^2]_x\,\diff x\\
		&= \|\rho_0\partial_tv_{xx}+(\rho_0)_x\partial_tv_{x}\|_{L^2}^2
		+\int_I\rho_0(\rho_0)_{xx}(\partial_tv_{x})^2\,\diff x\\
		&\leq M_0+Ct P(\sup_{0\leq s\leq t}E^{1/2}(s,v)),
	\end{aligned}
\end{equation} 
where \eqref{I-Priori-time-21} and \eqref{I-Priori-ellip-21.7} have been used. Therefore it follows from \eqref{I-Priori-ellip-21.8} that
\begin{equation}\label{I-Priori-ellip-22}
	\begin{aligned}
		\|\rho_0\partial_tv_{xx}\|_{L^2}^2
		\leq M_0+Ct P(\sup_{0\leq s\leq t}E^{1/2}(s,v)).
	\end{aligned}
\end{equation}

\bigskip
\noindent{\bf{Estimate of \(\|\rho_0^2\partial_x^4v\|_{L^2(I)}\).}} Applying \(\partial_x^2\) to Equation \(\eqref{eq:main-2}_1\) gives
\begin{equation}\label{I-Priori-ellip-23}
\begin{aligned}
\partial_x^3\bigg(\frac{\rho_0v_x}{\eta_x^2}\bigg)=(\rho_0\partial_tv)_{xx}
+\partial_x^3\bigg(\frac{\rho_0^2}{\eta_x^2}\bigg).
\end{aligned}
\end{equation}
A direct calculation shows that
\begin{equation*}
\begin{aligned}
\bigg|\partial_x^3\bigg({\frac{\rho_0^2}{\eta_x^2}}\bigg)\bigg|
&\lesssim 1+|\eta_{xx}|+\rho_0(\eta_{xx}^2+|\partial_x^3\eta|)\\
&\quad+\rho_0^2(|\eta_{xx}|^3+|\eta_{xx}\partial_x^3\eta|+|\partial_x^4\eta|).
\end{aligned}
\end{equation*}
We then may apply Lemma \ref{le:Preliminary-1} and Lemma \ref{le:Preliminary-2} to estimate
\begin{equation}\label{I-Priori-ellip-25}
\begin{aligned}
\bigg\|\partial_x^3\bigg({\frac{\rho_0^2}{\eta_x^2}}\bigg)\bigg\|_{L^2}
&\lesssim 1+\|\eta_{xx}\|_{L^2}+(\|\eta_{xx}\|_{L^\infty}\|\eta_{xx}\|_{L^2}+\|\partial_x^3\eta\|_{L^2})\\
&\quad+(\|\eta_{xx}\|_{L^\infty}^2\|\eta_{xx}\|_{L^2}
+\|\eta_{xx}\|_{L^\infty}\|\partial_x^3\eta\|_{L^2}
+\|\rho_0\partial_x^4\eta\|_{L^2})\\
&\leq 1+Ct P(\sup_{0\leq s\leq t}E^{1/2}(s,v)).
\end{aligned}
\end{equation}
On the other hand, it holds that
\begin{equation}\label{I-Priori-ellip-25.5}
\begin{aligned}
\|(\rho_0\partial_tv)_{xx}\|_{L^2}^2
&\leq \|(\rho_0)_{xx}\partial_tv\|_{L^2}^2+2\|(\rho_0)_x\partial_tv_{x}\|_{L^2}^2+\|\rho_0\partial_tv_{xx}\|_{L^2}^2\\
&\lesssim \|\rho_0\partial_tv\|_{L^2}^2+\|\rho_0\partial_tv_{x}\|_{L^2}^2+\|\rho_0\partial_tv_{xx}\|_{L^2}^2\\
&\leq M_0+Ct P(\sup_{0\leq s\leq t}E^{1/2}(s,v)),
\end{aligned}
\end{equation}
where one has used \eqref{ineq:weighted Sobolev} for \(\|\partial_tv\|_{L^2}\) and \(\|\partial_tv_{x}\|_{L^2}\) in the second inequality, and \eqref{I-Priori-time-8}, \eqref{I-Priori-time-21} and \eqref{I-Priori-ellip-22} in the last inequality.
In view of \eqref{I-Priori-ellip-23}, \eqref{I-Priori-ellip-25} and \eqref{I-Priori-ellip-25.5}, we deduce
\begin{equation}\label{I-Priori-ellip-26}
\begin{aligned}
\bigg\|\bigg({\frac{\rho_0v_x}{\eta_x^2}}\bigg)_{xxx}\bigg\|_{L^2}^2
&\leq \|(\rho_0\partial_tv)_{xx}\|_{L^2}^2
+\bigg\|\partial_x^3\bigg(\frac{\rho_0^2}{\eta_x^2}\bigg)\bigg\|_{L^2}^2\\
&\leq M_0+Ct P(\sup_{0\leq s\leq t}E^{1/2}(s,v)).
\end{aligned}
\end{equation}

To bound \(\|\rho_0^2\partial_x^4v\big\|_{L^2(I)}\), one notes that
\begin{equation}\label{I-Priori-ellip-27}
\begin{aligned}
\rho_0\eta_x^{-2}\partial_x^4v+3(\rho_0)_x\eta_x^{-2}\partial_x^3v
&=\partial_x^3\bigg(\frac{\rho_0v_x}{\eta_x^2}\bigg)-3\rho_0(\eta_x^{-2})_x\partial_x^3v\\
&\quad-3(\rho_0\eta_x^{-2})_{xx}v_{xx}-\partial_x^3(\rho_0\eta_x^{-2})v_x.
\end{aligned}
\end{equation}
Considering the second term on the RHS of \eqref{I-Priori-ellip-27}, one may estimate
\begin{equation}\label{I-Priori-ellip-27.2}
\begin{aligned}
\|\rho_0^2(\eta_x^{-2})_{x}\partial_x^3v\|_{L^2}
\lesssim \|\partial_x^3v\|_{L^2}\|\eta_{xx}\|_{L^\infty}
\leq Ct P(\sup_{0\leq s\leq t}E^{1/2}(s,v)),
\end{aligned}
\end{equation}
where \eqref{Preliminary-1} and \eqref{Preliminary-5} have been utilized.
Recalling \eqref{I-Priori-ellip-13}, we may estimate the third term on the RHS of \eqref{I-Priori-ellip-27} as follows:
\begin{equation}\label{I-Priori-ellip-27.5}
\begin{aligned}
\|\rho_0(\rho_0\eta_x^{-2})_{xx}v_{xx}\|_{L^2}
&\lesssim \|\rho_0v_{xx}\|_{L^2}+
\|v_{xx}\|_{L^\infty}\big(\|\eta_{xx}\|_{L^2}+\\
&\quad(\|\eta_{xx}\|_{L^2}\|\eta_{xx}\|_{L^\infty}
+\|\partial_x^3\eta\|_{L^2})\big)\\
&\leq [M_0+Ct P(\sup_{0\leq s\leq t}E^{1/2}(s,v))]^{1/2}.
\end{aligned}
\end{equation}
Since
\begin{equation}\label{I-Priori-ellip-28}
\begin{aligned}
|\partial_x^3(\rho_0\eta_x^{-2})|
&\lesssim 1+|\eta_{xx}|+(\eta_{xx}^2+|\partial_x^3\eta|)\\
&\quad+\rho_0(|\eta_{xx}|^3+|\eta_{xx}\partial_x^3\eta|+|\partial_x^4\eta|),
\end{aligned}
\end{equation}
the last term on the RHS of \eqref{I-Priori-ellip-27} can be estimated as follows:
\begin{equation}\label{I-Priori-ellip-28.5}
\begin{aligned}
&\|\rho_0\partial_x^3(\rho_0\eta_x^{-2})v_{x}\|_{L^2}\\
&\quad\lesssim \|\rho_0v_{x}\|_{L^2}+
\|v_{x}\|_{L^\infty}\big(\|\eta_{xx}\|_{L^2}
+(\|\eta_{xx}\|_{L^2}\|\eta_{xx}\|_{L^\infty}
+\|\partial_x^3\eta\|_{L^2})\\
&\qquad+(\|\eta_{xx}\|_{L^2}\|\eta_{xx}\|_{L^\infty}^2
+\|\eta_{xx}\|_{L^2}\|\rho_0\partial_x^3\eta\|_{L^\infty}
+\|\rho_0\partial_x^4\eta\|_{L^2})\big)\\
&\quad\leq [M_0+Ct P(\sup_{0\leq s\leq t}E^{1/2}(s,v))]^{1/2},
\end{aligned}
\end{equation}
where \eqref{I-Priori-time-16}, Lemma \ref{le:Preliminary-1} and Lemma \ref{le:Preliminary-2} have been used.
Thus inserting \eqref{I-Priori-ellip-26}, \eqref{I-Priori-ellip-27.2}, \eqref{I-Priori-ellip-27.5} and \eqref{I-Priori-ellip-28.5} into \eqref{I-Priori-ellip-27} yields
\begin{equation}\label{I-Priori-ellip-29}
\begin{aligned}
\|\rho_0^2\eta_x^{-2}\partial_x^4v+3\rho_0(\rho_0)_x\eta_x^{-2}\partial_x^3v\|_{L^2}^2\leq M_0+Ct P(\sup_{0\leq s\leq t}E^{1/2}(s,v)).
\end{aligned}
\end{equation}
We then use integration by parts and invoke \eqref{Preliminary-5}, \eqref{I-Priori-ellip-third} and \eqref{I-Priori-ellip-29} to find
\begin{equation}\label{I-Priori-ellip-30}
\begin{aligned}
&\|\rho_0^2\eta_x^{-2}\partial_x^4v\|_{L^2}^2\\
&\quad=\|\rho_0^2\eta_x^{-2}\partial_x^4v+3\rho_0(\rho_0)_x\eta_x^{-2}\partial_x^3v\|_{L^2}^2
-9\|\rho_0(\rho_0)_x\eta_x^{-2}\partial_x^3v\|_{L^2}^2\\
&\qquad-3\int_I\rho_0^3(\rho_0)_x\eta_x^{-4}[(\partial_x^3v)^2]_x\,\diff x\\
&\quad= \|\rho_0^2\eta_x^{-2}\partial_x^4v+3\rho_0(\rho_0)_x\eta_x^{-2}\partial_x^3v\|_{L^2}^2+\int_I\rho_0^3[(\rho_0)_x\eta_x^{-4}]_x(\partial_x^3v)^2\,\diff x\\
&\quad\lesssim \|\rho_0^2\eta_x^{-2}\partial_x^4v+3\rho_0(\rho_0)_x\eta_x^{-2}\partial_x^3v\|_{L^2}^2+(1+\|\eta_{xx}\|_{L^\infty})
\int_I\rho_0^3(\partial_x^3v)^2\,\diff x\\
&\quad\leq M_0+Ct P(\sup_{0\leq s\leq t}E^{1/2}(s,v)).
\end{aligned}
\end{equation}
Hence \eqref{I-Priori-ellip-30} and \eqref{eta-bound} imply
\begin{equation}\label{I-Priori-ellip-31}
\begin{aligned}
\|\rho_0^2\partial_x^4v\|_{L^2}^2\leq M_0+Ct P(\sup_{0\leq s\leq t}E^{1/2}(s,v)).
\end{aligned}
\end{equation}

\bigskip
\noindent{\bf{Estimate of \(\|\rho_0^{3/2}\partial_t\partial_x^3v\|_{L^2(I)}\).}}
Since
\begin{equation*}
\begin{aligned}
\bigg|\partial_t\bigg({\frac{\rho_0^2}{\eta_x^2}}\bigg)_{xx}\bigg|
&\lesssim |v_x|+\rho_0(|v_x\eta_{xx}|+|v_{xx}|)\\
&\quad+\rho_0^2(|v_x\eta_{xx}^2|+|v_x\partial_x^3\eta|+|v_{xx}\eta_{xx}|+|\partial_x^3v|),
\end{aligned}
\end{equation*}
it holds that
\begin{equation}\label{II-Priori-ellip-1}
\begin{aligned}
\bigg\|\partial_t\bigg({\frac{\rho_0^2}{\eta_x^2}}\bigg)_x\bigg\|_{L^2}
&\lesssim \|v_x\|_{L^2}+\|v_x\|_{L^\infty}\|\eta_{xx}\|_{L^2}+\|\rho_0v_{xx}\|_{L^2}\\
&\quad+\|v_x\|_{L^\infty}\|\eta_{xx}\|_{L^\infty}\|\eta_{xx}\|_{L^2}+\|v_x\|_{L^\infty}\|\partial_x^3\eta\|_{L^2}\\
&\quad+\|v_{xx}\|_{L^\infty}\|\eta_{xx}\|_{L^2}+\|\rho_0^2\partial_x^3v\|_{L^2}\\
&\leq [M_0+Ct P(\sup_{0\leq s\leq t}E^{1/2}(s,v))]^{1/2},
\end{aligned}
\end{equation}
where one has used \eqref{I-Priori-ellip-5}, \eqref{I-Priori-ellip-15}, \eqref{I-Priori-ellip-third}, Lemma \ref{le:Preliminary-1} and Lemma \ref{le:Preliminary-2} in the last inequality.
Applying \(\partial_{tx}^2\) to Equation \(\eqref{eq:main-2}_1\) gives
\begin{equation}\label{II-Priori-ellip-2}
\begin{aligned}
\partial_t\bigg(\frac{\rho_0v_x}{\eta_x^2}\bigg)_{xx}=(\rho_0\partial_t^2v)_x
+\partial_t\bigg(\frac{\rho_0^2}{\eta_x^2}\bigg)_{xx},
\end{aligned}
\end{equation}
we then utilize \eqref{I-Priori-time-12}, \eqref{II-Priori-time-10}, \eqref{II-Priori-ellip-1} and \eqref{II-Priori-ellip-2} to estimate
\begin{equation}\label{II-Priori-ellip-3}
\begin{aligned}
\bigg\|\partial_t\bigg(\frac{\rho_0v_x}{\eta_x^2}\bigg)_{xx}\bigg\|_{L^2}^2
&\leq \|(\rho_0)_x\partial_t^2v\|_{L^2}^2+\|\rho_0\partial_t^2v_x\|_{L^2}^2
+\bigg\|\partial_t\bigg(\frac{\rho_0^2}{\eta_x^2}\bigg)_{xx}\bigg\|_{L^2}^2\\
&\lesssim \|\rho_0\partial_t^2v\|_{L^2}^2+\|\rho_0\partial_t^2v_x\|_{L^2}^2
+\bigg\|\partial_t\bigg(\frac{\rho_0^2}{\eta_x^2}\bigg)_{xx}\bigg\|_{L^2}^2\\
&\leq M_0+Ct P(\sup_{0\leq s\leq t}E^{1/2}(s,v)),
\end{aligned}
\end{equation}
where  \eqref{ineq:weighted Sobolev} has been used for \(\|\partial_t^2v\|_{L^2}\) in the second inequality.

Write
\begin{equation*}
\begin{aligned}
&(\rho_0\partial_tv_x)_{xx}\\
&=\partial_t\bigg({\frac{\rho_0v_x}{\eta_x^2}}\bigg)_{xx}\eta_x^2
+4\partial_t\bigg(\frac{\rho_0v_x}{\eta_x^2}\bigg)_{x}\eta_{x}\eta_{xx}
+2\partial_t\bigg(\frac{\rho_0v_x}{\eta_x^2}\bigg)(\eta_{xx}^2+\eta_{x}\partial_x^3\eta)\\
&\quad+2\bigg({\frac{\rho_0v_x}{\eta_x^2}}\bigg)_{xx}\eta_xv_x
+4\bigg(\frac{\rho_0v_x}{\eta_x^2}\bigg)_{x}(v_x\eta_{xx}+\eta_{x}v_{xx})\\
&\quad+2\frac{\rho_0v_x}{\eta_x^2}(2v_{xx}\eta_{xx}+v_x\partial_x^3\eta+\eta_{x}\partial_x^3v)=\colon \sum_{k=1}^6I_k.
\end{aligned}
\end{equation*}
It is clear that \(\|I_1\|_{L^2}\) satisfies the desired bound due to \eqref{II-Priori-ellip-3}.
In view of \eqref{Preliminary-5} and \eqref{I-Priori-ellip-21}, one may estimate
\begin{equation*}
\begin{aligned}
\|I_2\|_{L^2}
\lesssim
\bigg\|\partial_t\bigg(\frac{\rho_0v_x}{\eta_x^2}\bigg)_x\bigg\|_{L^2}
\|\eta_{xx}\|_{L^\infty}
\leq Ct P(\sup_{0\leq s\leq t}E^{1/2}(s,v)).
\end{aligned}
\end{equation*}
The estimates  \eqref{I-Priori-ellip-21} and \eqref{I-Priori-ellip-21.6} together with \eqref{Preliminary-2} and
\eqref{Preliminary-5} yield
\begin{equation*}
\begin{aligned}
\|I_3\|_{L^2}
&\lesssim
\bigg\|\partial_t\bigg(\frac{\rho_0v_x}{\eta_x^2}\bigg)\bigg\|_{L^\infty}(\|\eta_{xx}\|_{L^\infty}\|\eta_{xx}\|_{L^2}
+\|\partial_x^3\eta\|_{L^2})\\
&\leq Ct P(\sup_{0\leq s\leq t}E^{1/2}(s,v)).
\end{aligned}
\end{equation*}
It follows from \eqref{I-Priori-ellip-1} and \eqref{I-Priori-ellip-10} that
\begin{equation*}
\begin{aligned}
\|I_4\|_{L^2}\lesssim
\|v_{x}\|_{L^\infty}\bigg\|\bigg(\frac{\rho_0v_x}{\eta_x^2}\bigg)_{xx}\bigg\|_{L^2}
\leq M_0+Ct P(\sup_{0\leq s\leq t}E^{1/2}(s,v)),
\end{aligned}
\end{equation*}
and
\begin{equation*}
\begin{aligned}
\|I_5\|_{L^2}&\lesssim
\bigg\|\bigg(\frac{\rho_0v_x}{\eta_x^2}\bigg)_{x}\bigg\|_{L^\infty}(\|v_{x}\|_{L^\infty}\|\eta_{xx}\|_{L^2}
+\|v_{xx}\|_{L^2})\\
&\leq M_0+Ct P(\sup_{0\leq s\leq t}E^{1/2}(s,v)),
\end{aligned}
\end{equation*}
where one has used \eqref{ineq:weighted Sobolev} to find
\begin{equation}\label{II-Priori-ellip-3.2}
\begin{aligned}
\|v_{xx}\|_{L^2}&\lesssim
\|\rho_0v_{xx}\|_{L^2}+\|\rho_0\partial_x^3v\|_{L^2}\\
&\lesssim
\|\rho_0v_{xx}\|_{L^2}+\|\rho_0^2\partial_x^3v\|_{L^2}+\|\rho_0^2\partial_x^4v\|_{L^2}\\
&\leq [M_0+Ct P(\sup_{0\leq s\leq t}E^{1/2}(s,v))]^{1/2},
\end{aligned}
\end{equation}
due to \eqref{I-Priori-ellip-5}, \eqref{I-Priori-ellip-third} and \eqref{I-Priori-ellip-31}, and thus
\begin{equation}\label{II-Priori-ellip-3.3}
	\begin{aligned}
		\|v_x\|_{L^\infty}
		&\lesssim
		\|v_x\|_{L^2}+\|v_{xx}\|_{L^2}\\
		&\leq [M_0+Ct P(\sup_{0\leq s\leq t}E^{1/2}(s,v))]^{1/2},
	\end{aligned}
\end{equation}
which follows from \eqref{I-Priori-ellip-15} and \eqref{II-Priori-ellip-3.2}.
\(I_6\) can be estimated as
\begin{equation*}
\begin{aligned}
\|I_6\|_{L^2}
&\lesssim
\|v_x\|_{L^\infty}(\|v_{xx}\|_{L^\infty}\|\eta_{xx}\|_{L^2}
+\|v_{x}\|_{L^\infty}\|\partial_x^3\eta\|_{L^2}+\|\rho_0\partial_x^3v\|_{L^2})\\
&\leq M_0+Ct P(\sup_{0\leq s\leq t}E^{1/2}(s,v)),
\end{aligned}
\end{equation*}
where, to estimate the term \(\|v_x\|_{L^\infty}\|\rho_0\partial_x^3v\|_{L^2}\), one has used the fact that each term enjoys the bound \([M_0+Ct P(\sup_{0\leq s\leq t}E^{1/2}(s,v))]^{1/2}\), which follows from \eqref{II-Priori-ellip-3.3} and
\begin{equation}\label{II-Priori-ellip-3.4}
\begin{aligned}
\|\rho_0\partial_x^3v\|_{L^2}
&\lesssim
\|\rho_0^2\partial_x^3v\|_{L^2}+\|\rho_0^2\partial_x^4v\|_{L^2}\\
&\leq [M_0+Ct P(\sup_{0\leq s\leq t}E^{1/2}(s,v))]^{1/2},
\end{aligned}
\end{equation}
due to \eqref{I-Priori-ellip-third} and \eqref{I-Priori-ellip-31}.
Collecting all the cases, we finally get
\begin{equation}\label{II-Priori-ellip-3.5}
\begin{aligned}
\|(\rho_0\partial_tv_x)_{xx}\|_{L^2}^2
\leq M_0+Ct P(\sup_{0\leq s\leq t}E^{1/2}(s,v)).
\end{aligned}
\end{equation}

Since
\begin{equation*}
\begin{aligned}
\rho_0\partial_t\partial_x^3v+2(\rho_0)_{x}\partial_tv_{xx}=
(\rho_0\partial_tv_x)_{xx}-(\rho_0)_{xx}\partial_tv_x,
\end{aligned}
\end{equation*}
it follows that
\begin{equation}\label{II-Priori-ellip-3.6}
\begin{aligned}
&\|\rho_0^{3/2}\partial_t\partial_x^3v+2\rho_0^{1/2}(\rho_0)_{x}\partial_tv_{xx}\|_{L^2}^2\\
&\quad\lesssim \|\rho_0^{1/2}(\rho_0\partial_tv_x)_{xx}\|_{L^2}^2+\|\rho_0^{1/2}\partial_tv_x\|_{L^2}^2\\
&\quad\lesssim \|(\rho_0\partial_tv_x)_{xx}\|_{L^2}^2+\|\rho_0^{1/2}\partial_tv_x\|_{L^2}^2\\
&\quad\leq M_0+Ct P(\sup_{0\leq s\leq t}E^{1/2}(s,v)),
\end{aligned}
\end{equation}
where one has used \eqref{I-Priori-time-21} and \eqref{II-Priori-ellip-3.5}.
Integration by parts gives
\begin{equation}\label{II-Priori-ellip-3.7}
\begin{aligned}
&\|\rho_0^{3/2}\partial_t\partial_x^3v\|_{L^2}^2\\
&\quad=\|\rho_0^{3/2}\partial_t\partial_x^3v+2\rho_0^{1/2}(\rho_0)_{x}\partial_tv_{xx}\|_{L^2}^2
-4\|\rho_0^{1/2}(\rho_0)_{x}\partial_tv_{xx}\|_{L^2}^2\\
&\qquad-2\int_I\rho_0^2(\rho_0)_x[(\partial_tv_{xx})^2]_x\,\diff x\\
&\quad=\|\rho_0^{3/2}\partial_t\partial_x^3v+2\rho_0^{1/2}(\rho_0)_{x}\partial_tv_{xx}\|_{L^2}^2
+2\int_I\rho_0^2(\rho_0)_{xx}(\partial_tv_{xx})^2\,\diff x\\
&\quad\lesssim \|\rho_0^{3/2}\partial_t\partial_x^3v+2\rho_0^{1/2}(\rho_0)_{x}\partial_tv_{xx}\|_{L^2}^2
+
\int_I\rho_0^2(\partial_tv_{xx})^2\,\diff x\\
&\quad\leq M_0+Ct P(\sup_{0\leq s\leq t}E^{1/2}(s,v)),
\end{aligned}
\end{equation}
where one has used  \eqref{I-Priori-ellip-22} and \eqref{II-Priori-ellip-3.6}.
Hence it follows from \eqref{II-Priori-ellip-3.7} that
\begin{equation}\label{II-Priori-ellip-4}
\begin{aligned}
\|\rho_0^{3/2}\partial_t\partial_x^3v\|_{L^2}^2
\leq M_0+Ct P(\sup_{0\leq s\leq t}E^{1/2}(s,v)).
\end{aligned}
\end{equation}

\bigskip
\noindent{\bf{Estimate of \(\|\rho_0^{5/2}\partial_x^5v\|_{L^2(I)}\).}} Applying \(\partial_x^3\) to Equation \(\eqref{eq:main-2}_1\) gives
\begin{equation}\label{II-Priori-ellip-4.5}
\begin{aligned}
\partial_x^4\bigg(\frac{\rho_0v_x}{\eta_x^2}\bigg)=\partial_x^3(\rho_0\partial_tv)
+\partial_x^4\bigg(\frac{\rho_0^2}{\eta_x^2}\bigg).
\end{aligned}
\end{equation}
We will estimate the \(L^2-\) norm of \(\partial_x^4\big(\frac{\rho_0v_x}{\eta_x^2}\big)\) with {suitable weight using \eqref{II-Priori-ellip-4.5}. We start with the term \(\partial_x^3(\rho_0\partial_tv)\). For this, due to \eqref{II-Priori-ellip-4}, one shall compensate a weight \(\rho_0^{1/2}\) to estimate
\begin{equation}\label{II-Priori-ellip-5}
\begin{aligned}
\|\rho_0^{1/2}\partial_x^3(\rho_0\partial_tv)\|_{L^2}^2
&\lesssim \|\rho_0^{1/2}\partial_tv\|_{L^2}^2+\|\rho_0^{1/2}\partial_tv_x\|_{L^2}^2\\
&\quad+\|\rho_0^{1/2}\partial_tv_{xx}\|_{L^2}^2
+\|\rho_0^{3/2}\partial_t\partial_x^3v\|_{L^2}^2\\
&\lesssim \|\rho_0^{1/2}\partial_tv\|_{L^2}^2+\|\rho_0^{1/2}\partial_tv_x\|_{L^2}^2\\
&\quad+\|\rho_0^{3/2}\partial_tv_{xx}\|_{L^2}^2
+\|\rho_0^{3/2}\partial_t\partial_x^3v\|_{L^2}^2\\
&\leq M_0+Ct P(\sup_{0\leq s\leq t}E^{1/2}(s,v)).
\end{aligned}
\end{equation}
Here one has used \eqref{ineq:weighted Sobolev} for \(\|\rho_0^{1/2}\partial_tv_{xx}\|_{L^2}\) in the second inequality, and \eqref{I-Priori-ellip-22} and \eqref{II-Priori-ellip-4} in the last equality.
Next, we deal with the term \(\partial_x^4\big({\frac{\rho_0^2}{\eta_x^2}}\big)\).
Direct calculations give
\begin{equation}\label{II-Priori-ellip-6}
\begin{aligned}
\bigg|\partial_x^4\bigg({\frac{\rho_0^2}{\eta_x^2}}\bigg)\bigg|
&\lesssim 1+|\eta_{xx}|+(\eta_{xx}^2+|\partial_x^3\eta|)
+\rho_0(|\eta_{xx}|^3+|\eta_{xx}\partial_x^3\eta|+|\partial_x^4\eta|)\\
&\quad+\rho_0^2(|\eta_{xx}|^4+\eta_{xx}^2|\partial_x^3\eta|+(\partial_x^3\eta)^2+|\eta_{xx}\partial_x^4\eta|+|\partial_x^5\eta|).
\end{aligned}
\end{equation}
Thus, one can get
\begin{equation}\label{II-Priori-ellip-7}
\begin{aligned}
\bigg\|\partial_x^4\bigg({\frac{\rho_0^2}{\eta_x^2}}\bigg)\bigg\|_{L^2}
&\lesssim 1+\|\eta_{xx}\|_{L^2}+(\|\eta_{xx}\|_{L^\infty}\|\eta_{xx}\|_{L^2}+\|\partial_x^3\eta\|_{L^2})\\
&\quad+(\|\eta_{xx}\|_{L^\infty}^2\|\eta_{xx}\|_{L^2}
+\|\eta_{xx}\|_{L^\infty}\|\partial_x^3\eta\|_{L^2}
+\|\rho_0\partial_x^4\eta\|_{L^2})\\
&\quad+(\|\eta_{xx}\|_{L^\infty}^3\|\eta_{xx}\|_{L^2}
+\|\eta_{xx}\|_{L^\infty}^2\|\partial_x^3\eta\|_{L^2}\\
&\quad+\|\rho_0\partial_x^3\eta\|_{L^\infty}\|\partial_x^3\eta\|_{L^2}
+\|\eta_{xx}\|_{L^2}\|\rho_0\partial_x^4\eta\|_{L^2}
+\|\rho_0^2\partial_x^5\eta\|_{L^2})\\
&\leq 1+Ct P(\sup_{0\leq s\leq t}E^{1/2}(s,v)),
\end{aligned}
\end{equation}
where in the last inequality Lemma \ref{le:Preliminary-1} and Lemma \ref{le:Preliminary-2} have been utilized.
By compensating a weight \(\rho_0^{1/2}\), we deduce from \eqref{II-Priori-ellip-4.5}, \eqref{II-Priori-ellip-5} and \eqref{II-Priori-ellip-7} that
\begin{equation}\label{II-Priori-ellip-8}
\begin{aligned}
\bigg\|\rho_0^{1/2}\partial_x^4\bigg({\frac{\rho_0v_x}{\eta_x^2}}\bigg)\bigg\|_{L^2}^2
&\leq \|\rho_0^{1/2}\partial_x^3(\rho_0\partial_tv)\|_{L^2}^2
+\bigg\|\rho_0^{1/2}\partial_x^4\bigg(\frac{\rho_0^2}{\eta_x^2}\bigg)\bigg\|_{L^2}^2\\
&\lesssim \|\rho_0^{1/2}\partial_x^3(\rho_0\partial_tv)\|_{L^2}^2
+\bigg\|\partial_x^4\bigg(\frac{\rho_0^2}{\eta_x^2}\bigg)\bigg\|_{L^2}^2\\
&\leq M_0+Ct P(\sup_{0\leq s\leq t}E^{1/2}(s,v)).
\end{aligned}
\end{equation}

We next control \(\|\rho_0^{5/2}\partial_x^5v\|_{L^2}\).
Note that
\begin{equation*}
\begin{aligned}
&\rho_0\eta_x^{-2}\partial_x^5v+4(\rho_0)_x\eta_x^{-2}\partial_x^4v\\
&=\partial_x^4\bigg(\frac{\rho_0v_x}{\eta_x^2}\bigg)-4\rho_0(\eta_x^{-2})_x\partial_x^4v-6(\rho_0\eta_x^{-2})_{xx}\partial_x^3v\\
&\quad-4\partial_x^3(\rho_0\eta_x^{-2})v_{xx}-\partial_x^4(\rho_0\eta_x^{-2})v_x=\colon \sum_{k=1}^5I_k.
\end{aligned}
\end{equation*}
The term \(I_1\) has been handled by compensating a weight \(\rho_0^{1/2}\) due to \eqref{II-Priori-ellip-8}.
It follows from Lemma \ref{le:Preliminary-1} and Lemma \ref{le:Preliminary-2} that \(I_2\) and \(I_4\) may be estimated as follows:
\begin{equation*}
\begin{aligned}
\|I_2\|_{L^2}
\leq \|\eta_{xx}\|_{L^\infty}\|\rho_0\partial_x^4v\|_{L^2}
\leq Ct P(\sup_{0\leq s\leq t}E^{1/2}(s,v)),
\end{aligned}
\end{equation*}
and
\begin{equation*}
\begin{aligned}
\|I_4\|_{L^2}
&\lesssim \|v_{xx}\|_{L^2}+
\|v_{xx}\|_{L^\infty}\big(\|\eta_{xx}\|_{L^2}
+(\|\eta_{xx}\|_{L^2}\|\eta_{xx}\|_{L^\infty}
+\|\partial_x^3\eta\|_{L^2})\\
&\quad+(\|\eta_{xx}\|_{L^2}\|\eta_{xx}\|_{L^\infty}^2
+\|\eta_{xx}\|_{L^2}\|\rho_0\partial_x^3\eta\|_{L^\infty}
+\|\rho_0\partial_x^4\eta\|_{L^2})\big)\\
&\leq [M_0+Ct P(\sup_{0\leq s\leq t}E^{1/2}(s,v))]^{1/2},
\end{aligned}
\end{equation*}
where \eqref{I-Priori-ellip-28} and \eqref{II-Priori-ellip-3.2} have been used in estimating  \(I_4\).
For \(I_3\) and \(I_5\), one can use a weight \(\rho_0\) and apply Lemma \ref{le:Preliminary-1} and Lemma \ref{le:Preliminary-2} to get
{\small \begin{equation*}
\begin{aligned}
\|\rho_0I_3\|_{L^2}
&\lesssim \|\rho_0\partial_x^3v\|_{L^2}+
\|\rho_0\partial_x^3v\|_{L^\infty}\big(\|\eta_{xx}\|_{L^2}+
(\|\eta_{xx}\|_{L^2}\|\eta_{xx}\|_{L^\infty}
+\|\partial_x^3\eta\|_{L^2})\big)\\
&\leq [M_0+Ct P(\sup_{0\leq s\leq t}E^{1/2}(s,v))]^{1/2},
\end{aligned}
\end{equation*}}
and
{\small \begin{equation*}
\begin{aligned}
\|\rho_0I_5\|_{L^2}
&\lesssim \|\rho_0v_{x}\|_{L^2}+
\|v_{x}\|_{L^\infty}\big(\|\eta_{xx}\|_{L^2}
+(\|\eta_{xx}\|_{L^2}\|\rho_0\eta_{xx}\|_{L^\infty}
+\|\rho_0\partial_x^3\eta\|_{L^2})\\
&\quad+(\|\eta_{xx}\|_{L^2}\|\eta_{xx}\|_{L^\infty}^2
+\|\eta_{xx}\|_{L^2}\|\rho_0\partial_x^3\eta\|_{L^\infty}
+\|\rho_0\partial_x^4\eta\|_{L^2})\\
&\quad+(\|\eta_{xx}\|_{L^2}\|\eta_{xx}\|_{L^\infty}^3
+\|\partial_x^3\eta\|_{L^2}\|\eta_{xx}\|_{L^\infty}^2
+\|\partial_x^3\eta\|_{L^2}\|\rho_0\partial_x^3\eta\|_{L^\infty}\\
&\quad+\|\eta_{xx}\|_{L^2}\|\rho_0^2\partial_x^4\eta\|_{L^\infty}
+\|\rho_0^2\partial_x^5\eta\|_{L^2})\big)\\
&\leq [M_0+Ct P(\sup_{0\leq s\leq t}E^{1/2}(s,v))]^{1/2},
\end{aligned}
\end{equation*}}
here  one has invoked \eqref{I-Priori-ellip-13} and \eqref{II-Priori-ellip-3.4} in estimating  \(I_3\), and  used
\begin{equation*}
\begin{aligned}
|\partial_x^4(\rho_0\eta_x^{-2})|
&\lesssim 1+|\eta_{xx}|+(\eta_{xx}^2+|\partial_x^3\eta|)
+(|\eta_{xx}|^3+|\eta_{xx}\partial_x^3\eta|+|\partial_x^4\eta|)\\
&\quad+\rho_0(|\eta_{xx}|^4+\eta_{xx}^2|\partial_x^3\eta|+\eta_{xxx}^2+|\eta_{xx}\partial_x^4\eta|+|\partial_x^5\eta|),
\end{aligned}
\end{equation*}
in estimating \(I_5\).
It follows from theses estimates and using a weight \(\rho_0\) that
\begin{equation}\label{II-Priori-ellip-9}
\begin{aligned}
\|\rho_0^2\eta_x^{-2}\partial_x^5v+4\rho_0(\rho_0)_x\eta_x^{-2}\partial_x^4v\|_{L^2}^2\leq M_0+Ct P(\sup_{0\leq s\leq t}E^{1/2}(s,v)).
\end{aligned}
\end{equation}
Then integration by parts leads to
\begin{equation}\label{II-Priori-ellip-10}
\begin{aligned}
&\|\rho_0^{5/2}\eta_x^{-2}\partial_x^5v\|_{L^2}^2\\
&\quad=\|\rho_0^{5/2}\eta_x^{-2}\partial_x^5v
+4\rho_0^{3/2}(\rho_0)_x\eta_x^{-2}\partial_x^4v\|_{L^2}^2\\
&\qquad-16\|\rho_0^{3/2}(\rho_0)_x\eta_x^{-2}\partial_x^4v\|_{L^2}^2-4\int_I\rho_0^4(\rho_0)_x\eta_x^{-4}[(\partial_x^4v)^2]_x\,\diff x\\
&\quad= \|\rho_0^{5/2}\eta_x^{-2}\partial_x^5v
+4\rho_0^{3/2}(\rho_0)_x\eta_x^{-2}\partial_x^4v\|_{L^2}^2
+4\int_I\rho_0^4[(\rho_0)_x\eta_x^{-4}]_x(\partial_x^4v)^2\,\diff x\\
&\quad\lesssim \|\rho_0^{5/2}\eta_x^{-2}\partial_x^5v
+4\rho_0^{3/2}(\rho_0)_x\eta_x^{-2}\partial_x^4v\|_{L^2}^2
+(1+\|\eta_{xx}\|_{L^\infty})\int_I\rho_0^4(\partial_x^4v)^2\,\diff x\\
&\quad\leq M_0+Ct P(\sup_{0\leq s\leq t}E^{1/2}(s,v)),
\end{aligned}
\end{equation}
where \eqref{Preliminary-5}, \eqref{I-Priori-ellip-31} and \eqref{II-Priori-ellip-9} have been used.
The inequality \eqref{II-Priori-ellip-10} and \eqref{eta-bound} yield
\begin{equation}\label{II-Priori-ellip-11}
\begin{aligned}
\|\rho_0^{5/2}\partial_x^5v\|_{L^2}^2\leq M_0+Ct P(\sup_{0\leq s\leq t}E^{1/2}(s,v)).
\end{aligned}
\end{equation}

\bigskip
\noindent{\bf{Estimate of \(\|\rho_0^2\partial_t\partial_x^4v\|_{L^2(I)}\).}} Applying \(\partial_t\partial_x^2\) to Equation \(\eqref{eq:main-2}_1\) gives
\begin{equation}\label{II-Priori-ellip-11.5}
\begin{aligned}
\partial_t\partial_x^3\bigg(\frac{\rho_0v_x}{\eta_x^2}\bigg)=(\rho_0\partial_t^2v)_{xx}
+\partial_t\partial_x^3\bigg(\frac{\rho_0^2}{\eta_x^2}\bigg).
\end{aligned}
\end{equation}
Thus, to estimate the \(L^2-\) norm of \(\partial_t\partial_x^3\big(\frac{\rho_0v_x}{\eta_x^2}\big)\), it suffices to estimate \(L^2-\) norm of \(\partial_t\partial_x^3\big({\frac{\rho_0^2}{\eta_x^2}}\big)\) and \((\rho_0\partial_t^2v)_{xx}\).
We start with \(\partial_t\partial_x^3\big({\frac{\rho_0^2}{\eta_x^2}}\big)\).
Since
\begin{equation*}
\begin{aligned}
\bigg|\partial_t\partial_x^3\bigg({\frac{\rho_0^2}{\eta_x^2}}\bigg)\bigg|
&\lesssim |v_x|+(|v_x\eta_{xx}|+|v_{xx}|)
+\rho_0(|v_x\eta_{xx}^2|+|v_x\partial_x^3\eta|\\
&\quad+|v_{xx}\eta_{xx}|+|\partial_x^3v|)
+\rho_0^2(|v_x\eta_{xx}^3|+|v_x\eta_{xx}\partial_x^3\eta|\\
&\quad+|v_x\partial_x^4\eta|+|v_{xx}\eta_{xx}^2|
+|v_{xx}\partial_x^3\eta|+|\partial_x^3v\eta_{xx}|+|\partial_x^4v|),
\end{aligned}
\end{equation*}
one gets
\begin{equation}\label{II-Priori-ellip-12}
\begin{aligned}
\bigg\|\partial_t\partial_x^3\bigg({\frac{\rho_0^2}{\eta_x^2}}\bigg)\bigg\|_{L^2}
&\lesssim \|v_x\|_{L^2}+(\|v_x\|_{L^\infty}\|\eta_{xx}\|_{L^2}+\|v_{xx}\|_{L^2})\\
&\quad+(\|v_x\|_{L^\infty}\|\eta_{xx}\|_{L^\infty}\|\eta_{xx}\|_{L^2}+\|v_x\|_{L^\infty}\|\partial_x^3\eta\|_{L^2}\\
&\quad+\|v_{xx}\|_{L^\infty}\|\eta_{xx}\|_{L^2}+
\|\rho_0\partial_x^3v\|_{L^2})\\
&\quad+(\|v_x\|_{L^\infty}\|\eta_{xx}\|_{L^\infty}^2\|\eta_{xx}\|_{L^2}+\|v_x\|_{L^\infty}\|\eta_{xx}\|_{L^\infty}\|\partial_x^3\eta\|_{L^2}\\
&\quad+\|v_x\|_{L^\infty}\|\rho_0\partial_x^4\eta\|_{L^2}
+\|v_{xx}\|_{L^\infty}\|\eta_{xx}\|_{L^\infty}\|\eta_{xx}\|_{L^2}\\
&\quad+\|v_{xx}\|_{L^\infty}\|\partial_x^3\eta\|_{L^2}+\|\eta_{xx}\|_{L^\infty}\|\partial_x^3v\|_{L^2}+\|\rho_0^2\partial_x^4v\|_{L^2})\\
&\leq [M_0+Ct P(\sup_{0\leq s\leq t}E^{1/2}(s,v))]^{1/2},
\end{aligned}
\end{equation}
where one has used \eqref{I-Priori-ellip-15}, \eqref{II-Priori-ellip-3.2}, Lemma \ref{le:Preliminary-1} and Lemma \ref{le:Preliminary-2}.

Next, we deal with \((\rho_0\partial_t^2v)_{xx}\).
Applying \(\partial_t^2\) to Equation \(\eqref{eq:main-2}_1\) yields
\begin{equation*}
\begin{aligned}
\partial_t^2\bigg(\frac{\rho_0v_x}{\eta_x^2}\bigg)_{x}=\rho_0\partial_t^3v
+\partial_t^2\bigg(\frac{\rho_0^2}{\eta_x^2}\bigg)_{x}.
\end{aligned}
\end{equation*}
Since
\begin{equation*}
\begin{aligned}
\bigg|\partial_t^2\bigg({\frac{\rho_0^2}{\eta_x^2}}\bigg)_{x}\bigg|
\lesssim \rho_0(v_x^2+|\partial_tv_x|)+\rho_0^2(v_x^2|\eta_{xx}|
+|\partial_tv_x\eta_{xx}|+|v_{x}v_{xx}|+|\partial_tv_{xx}|),
\end{aligned}
\end{equation*}
one gets
\begin{equation}\label{II-Priori-ellip-13}
\begin{aligned}
\bigg\|\partial_t^2\bigg({\frac{\rho_0^2}{\eta_x^2}}\bigg)_{x}\bigg\|_{L^2}
&\lesssim (\|\rho_0v_x\|_{L^\infty}\|v_x\|_{L^2}+\|\rho_0\partial_tv_x\|_{L^2})\\
&\quad+(\|v_x\|_{L^\infty}^2\|\eta_{xx}\|_{L^2}+\|\rho_0\partial_tv_x\|_{L^2}\|\eta_{xx}\|_{L^\infty}\\
&\quad+\|v_x\|_{L^\infty}\|\rho_0v_{xx}\|_{L^2}+\|\rho_0\partial_tv_{xx}\|_{L^2})\\
&\leq [M_0+Ct P(\sup_{0\leq s\leq t}E^{1/2}(s,v))]^{1/2},
\end{aligned}
\end{equation}
where one has used the fact that in \(\|\rho_0v_x\|_{L^\infty}\|v_x\|_{L^2}\) and \(\|v_x\|_{L^\infty}\|\rho_0v_{xx}\|_{L^2}\), each factor enjoys the same bound \([M_0+Ct P(\sup_{0\leq s\leq t}E^{1/2}(s,v))]^{1/2}\), due to \eqref{I-Priori-ellip-15}, \eqref{II-Priori-ellip-3.3} and \eqref{I-Priori-ellip-5}.
In view of \eqref{II-Priori-time-4} and \eqref{II-Priori-ellip-13}, we obtain
\begin{equation}\label{II-Priori-ellip-14}
\begin{aligned}
\bigg\|\partial_t^2\bigg(\frac{\rho_0v_x}{\eta_x^2}\bigg)_{x}\bigg\|_{L^2}^2
&\lesssim \|\rho_0\partial_t^3v\|_{L^2}^2+\bigg\|\partial_t^2\bigg({\frac{\rho_0^2}{\eta_x^2}}\bigg)_{x}\bigg\|_{L^2}^2\\
&\leq M_0+Ct P(\sup_{0\leq s\leq t}E^{1/2}(s,v)).
\end{aligned}
\end{equation}
Note that
\begin{equation*}
\begin{aligned}
\rho_0\partial_t^2v_{xx}&=\partial_t^2\bigg(\frac{\rho_0v_x}{\eta_x^2}\bigg)_{x}\eta_x^2
-\partial_t^2\bigg[\bigg(\frac{\rho_0}{\eta_x^2}\bigg)_xv_{x}\bigg]\eta_x^2\\
&\quad-\partial_t^2\bigg(\frac{\rho_0}{\eta_x^2}\bigg)v_{xx}\eta_x^2
-\partial_t\bigg(\frac{\rho_0}{\eta_x^2}\bigg)\partial_tv_{xx}\eta_x^2
=\colon\sum_{k=1}^4I_k.
\end{aligned}
\end{equation*}
The estimate on the \(L^2-\) norm of \(I_1\) follows from \eqref{II-Priori-ellip-14}.
The terms \(I_3\) and \(I_4\) can be estimated straightforwardly as follows:
\begin{equation*}
\begin{aligned}
\|I_3\|_{L^2}
&\lesssim \|\rho_0(v_x^2+\partial_tv_x)v_{xx}\|_{L^2}
\lesssim \|v_{x}\|_{L^\infty}^2\|\rho_0v_{xx}\|_{L^2}
+\|\rho_0v_{xx}\|_{L^\infty}\|\partial_tv_x\|_{L^2}\\
&\lesssim \|v_{x}\|_{L^\infty}^2\|\rho_0v_{xx}\|_{L^2}
+\|\rho_0v_{xx}\|_{L^\infty}(\|\rho_0\partial_tv_x\|_{L^2}
+\|\rho_0\partial_tv_{xx}\|_{L^2})\\
&\leq M_0+Ct P(\sup_{0\leq s\leq t}E^{1/2}(s,v)),
\end{aligned}
\end{equation*}
and
\begin{equation*}
\begin{aligned}
\|I_4\|_{L^2}
&\lesssim \|\rho_0v_x\partial_tv_{xx}\|_{L^2}
\lesssim \|v_{x}\|_{L^\infty}\|\rho_0\partial_tv_{xx}\|_{L^2}\\
&\leq M_0+Ct P(\sup_{0\leq s\leq t}E^{1/2}(s,v)),
\end{aligned}
\end{equation*}
since each factor on the RHS of \(I_3\) and \(I_4\) enjoys the same bound \([M_0+Ct P(\sup_{0\leq s\leq t}E^{1/2}(s,v))]^{1/2}\).
Here one has used \eqref{ineq:weighted Sobolev} for \(\|\partial_tv_x\|_{L^2}\) in the third inequality of \(I_3\), and bounded \(\|\rho_0v_{xx}\|_{L^\infty}\) by \eqref{II-Priori-ellip-3.2} and \eqref{II-Priori-ellip-3.4} in the forth inequality of \(I_3\).
For \(I_2\), we first calculate
{\small \begin{equation*}
\begin{aligned}
\bigg|\partial_t^2\bigg[\bigg(\frac{\rho_0}{\eta_x^2}\bigg)_xv_{x}\bigg]\bigg|
&\lesssim |v_x|\big(v_x^2+|\partial_tv_x|
+\rho_0(|\eta_{xx}|v_x^2
+|v_xv_{xx}|+|\eta_{xx}\partial_tv_x|+|\partial_tv_{xx}|)\big)\\
&\quad+|\partial_tv_x|\big(|v_x|+\rho_0(|\eta_{xx}v_x|+|v_{xx}|)\big)+|\partial_t^2v_x|(1+\rho_0|\eta_{xx}|),
\end{aligned}
\end{equation*}}
and then compensate a weight \(\rho_0\) to estimate
\begin{equation*}
\begin{aligned}
\|\rho_0I_2\|_{L^2}
&\lesssim (\|v_{x}\|_{L^\infty}^2\|v_{x}\|_{L^2}+\|v_{x}\|_{L^\infty}\|\rho_0\partial_tv_x\|_{L^2}
+\|v_{x}\|_{L^\infty}^3\|\eta_{xx}\|_{L^2}\\
&\quad+\|v_{x}\|_{L^\infty}^2\|v_{xx}\|_{L^2}
+\|v_{x}\|_{L^\infty}\|\eta_{xx}\|_{L^\infty}\|\rho_0\partial_tv_x\|_{L^2}\\
&\quad+\|v_{x}\|_{L^\infty}\|\rho_0\partial_tv_{xx}\|_{L^2})\\
&\quad+(\|v_{x}\|_{L^\infty}\|\rho_0\partial_tv_x\|_{L^2}
+\|v_{x}\|_{L^\infty}\|\eta_{xx}\|_{L^\infty}\|\rho_0\partial_tv_x\|_{L^2}\\
&\quad+\|\rho_0v_{xx}\|_{L^\infty}\|\rho_0\partial_tv_x\|_{L^2})\\
&\quad+(\|\rho_0\partial_t^2v_x\|_{L^2}
+\|\eta_{xx}\|_{L^\infty}\|\rho_0\partial_t^2v_x\|_{L^2})\\
&\leq  [M_0+Ct P(\sup_{0\leq s\leq t}E^{1/2}(s,v))]^{1/2}.
\end{aligned}
\end{equation*}
It follows from the estimates in \(I_i,\ i=1,2,3,4\) that
\begin{equation}\label{II-Priori-ellip-15}
\begin{aligned}
\|\rho_0^2\partial_t^2v_{xx}\|_{L^2}^2
\leq M_0+Ct P(\sup_{0\leq s\leq t}E^{1/2}(s,v)).
\end{aligned}
\end{equation}

Consequently, \eqref{II-Priori-ellip-11.5}, \eqref{II-Priori-ellip-12} and \eqref{II-Priori-ellip-15} yield
\begin{equation}\label{II-Priori-ellip-16}
\begin{aligned}
\bigg\|\rho_0\partial_t\partial_x^3\bigg(\frac{\rho_0v_x}{\eta_x^2}\bigg)\bigg\|_{L^2}^2
&\leq \|\rho_0(\rho_0\partial_t^2v)_{xx}\|_{L^2}^2
+\bigg\|\rho_0\partial_t\partial_x^3\bigg(\frac{\rho_0^2}{\eta_x^2}\bigg)\bigg\|_{L^2}^2\\
&\lesssim \|\rho_0\partial_t^2v\|_{L^2}^2+\|\rho_0\partial_t^2v_{x}\|_{L^2}^2
+\|\rho_0^2\partial_t^2v_{xx}\|_{L^2}^2\\
&\quad+\bigg\|\partial_t\partial_x^3\bigg(\frac{\rho_0^2}{\eta_x^2}\bigg)\bigg\|_{L^2}^2\\
&\leq M_0+Ct P(\sup_{0\leq s\leq t}E^{1/2}(s,v)).
\end{aligned}
\end{equation}

Next, we derive the weighted \(L^2\) estimate of \(\partial_x^3(\rho_0\partial_tv_x)\).
Note that
\begin{equation*}
\begin{aligned}
&\partial_x^3(\rho_0\partial_tv_x)\\
&=\partial_t\partial_x^3\bigg({\frac{\rho_0v_x}{\eta_x^2}}\bigg)\eta_x^2
+6\partial_t\bigg(\frac{\rho_0v_x}{\eta_x^2}\bigg)_{xx}\eta_{x}\eta_{xx}\\
&\quad+6\partial_t\bigg(\frac{\rho_0v_x}{\eta_x^2}\bigg)_x(\eta_{xx}^2+\eta_{x}\partial_x^3\eta)\\
&\quad+2\partial_t\bigg(\frac{\rho_0v_x}{\eta_x^2}\bigg)(3\eta_{xx}\partial_x^3\eta+\eta_{x}\partial_x^4\eta)\\
&\quad+2\partial_x^3\bigg({\frac{\rho_0v_x}{\eta_x^2}}\bigg)\eta_xv_x
+6\bigg(\frac{\rho_0v_x}{\eta_x^2}\bigg)_{xx}(v_x\eta_{xx}+\eta_{x}v_{xx})\\
&\quad+6\bigg(\frac{\rho_0v_x}{\eta_x^2}\bigg)_{x}(2v_{xx}\eta_{xx}+v_x\partial_x^3\eta+\eta_{x}\partial_x^3v)\\
&\quad+2\frac{\rho_0v_x}{\eta_x^2}(3v_{xx}\partial_x^3\eta+3v_{xxx}\eta_{xx}+v_x\partial_x^4\eta+\eta_{x}\partial_x^4v)
=\colon\sum_{k=1}^8I_k.
\end{aligned}
\end{equation*}
For the terms \(I_k\) when \(k=2,3,4,5,6\), one may get directly
\begin{equation*}
\begin{aligned}
\|I_2\|_{L^2}
\lesssim
\bigg\|\partial_t\bigg(\frac{\rho_0v_x}{\eta_x^2}\bigg)_{xx}\bigg\|_{L^2}\|\eta_{xx}\|_{L^\infty}
\leq Ct P(\sup_{0\leq s\leq t}E^{1/2}(s,v)),
\end{aligned}
\end{equation*}
\begin{equation*}
\begin{aligned}
\|I_3\|_{L^2}&\lesssim
\bigg\|\partial_t\bigg(\frac{\rho_0v_x}{\eta_x^2}\bigg)_x\bigg\|_{L^\infty}(\|\eta_{xx}\|_{L^\infty}\|\eta_{xx}\|_{L^2}
+\|\partial_x^3\eta\|_{L^2})\\
&\leq Ct P(\sup_{0\leq s\leq t}E^{1/2}(s,v)),
\end{aligned}
\end{equation*}
\begin{equation*}
\begin{aligned}
\|I_4\|_{L^2}&\lesssim
\bigg\|\partial_t\bigg(\frac{\rho_0v_x}{\eta_x^2}\bigg)\bigg\|_{L^\infty}(\|\eta_{xx}\|_{L^\infty}\|\partial_x^3\eta\|_{L^2}
+\|\partial_x^4\eta\|_{L^2})\\
&\leq Ct P(\sup_{0\leq s\leq t}E^{1/2}(s,v)),
\end{aligned}
\end{equation*}
\begin{equation*}
\begin{aligned}
\|I_5\|_{L^2}
\lesssim
\|v_{x}\|_{L^\infty}\bigg\|\partial_x^3\bigg(\frac{\rho_0v_x}{\eta_x^2}\bigg)\bigg\|_{L^2}
\leq M_0+Ct P(\sup_{0\leq s\leq t}E^{1/2}(s,v)),
\end{aligned}
\end{equation*}
\begin{equation*}
\begin{aligned}
\|I_6\|_{L^2}&\lesssim
\bigg\|\bigg(\frac{\rho_0v_x}{\eta_x^2}\bigg)_{xx}\bigg\|_{L^\infty}(\|v_{x}\|_{L^\infty}\|\eta_{xx}\|_{L^2}
+\|v_{xx}\|_{L^2})\\
&\leq M_0+Ct P(\sup_{0\leq s\leq t}E^{1/2}(s,v)).
\end{aligned}
\end{equation*}
To estimate \(I_1\), \(I_7\) and \(I_8\), we need a weight \(\rho_0\). The estimate on \(\rho_0I_1\) has been done due to \eqref{II-Priori-ellip-16}.
For \(I_7\) and \(I_8\), one can get
{\small \begin{equation*}
\begin{aligned}
\|\rho_0I_7\|_{L^2}
&\lesssim
\bigg\|\bigg(\frac{\rho_0v_x}{\eta_x^2}\bigg)_x\bigg\|_{L^\infty}(\|v_{xx}\|_{L^\infty}\|\eta_{xx}\|_{L^2}
+\|v_{x}\|_{L^\infty}\|\partial_x^3\eta\|_{L^2}+\|\rho_0\partial_x^3v\|_{L^2})\\
&\leq M_0+Ct P(\sup_{0\leq s\leq t}E^{1/2}(s,v)),
\end{aligned}
\end{equation*}}
and
\begin{equation*}
\begin{aligned}
\|\rho_0I_8\|_{L^2}
&\lesssim
\|v_x\|_{L^\infty}(\|v_{xx}\|_{L^\infty}\|\partial_x^3\eta\|_{L^2}
+\|v_{xx}\|_{L^\infty}\|\eta_{xx}\|_{L^2}\\
&\quad+\|v_{x}\|_{L^\infty}\|\rho_0\partial_x^4\eta\|_{L^2}+\|\rho_0^2\partial_x^4v\|_{L^2})\\
&\leq M_0+Ct P(\sup_{0\leq s\leq t}E^{1/2}(s,v)),
\end{aligned}
\end{equation*}
where one has used \eqref{I-Priori-ellip-31}, \eqref{II-Priori-ellip-3.4}, Lemma \ref{le:Preliminary-1} and Lemma \ref{le:Preliminary-2}.
Collecting all the cases leads to
\begin{equation}\label{II-Priori-ellip-17}
\begin{aligned}
\|\rho_0\partial_x^3(\rho_0\partial_tv_x)\|_{L^2}^2
\leq M_0+Ct P(\sup_{0\leq s\leq t}E^{1/2}(s,v)).
\end{aligned}
\end{equation}

Since
\begin{equation*}
\begin{aligned}
\rho_0\partial_t\partial_x^4v+3(\rho_0)_{x}\partial_t\partial_x^3v=
\partial_x^3(\rho_0\partial_tv_x)-\partial_x^3\rho_0\partial_tv_x-3(\rho_0)_{xx}\partial_tv_{xx},
\end{aligned}
\end{equation*}
one can get
\begin{equation}\label{II-Priori-ellip-17.2}
\begin{aligned}
&\|\rho_0^2\partial_t\partial_x^4v+3\rho_0(\rho_0)_{x}\partial_t\partial_x^3v\|_{L^2}^2\\
&\quad\lesssim \|\rho_0\partial_x^3(\rho_0\partial_tv_x)\|_{L^2}^2+\|\rho_0\partial_tv_x\|_{L^2}^2+\|\rho_0\partial_tv_{xx}\|_{L^2}^2\\
&\quad\leq M_0+Ct P(\sup_{0\leq s\leq t}E^{1/2}(s,v)),
\end{aligned}
\end{equation}
where \eqref{I-Priori-time-21}, \eqref{I-Priori-ellip-22} and \eqref{II-Priori-ellip-17} have been utilized.
Integration by parts gives
\begin{equation}\label{II-Priori-ellip-17.5}
\begin{aligned}
&\|\rho_0^2\partial_t\partial_x^4v\|_{L^2}^2\\
&\quad=\|\rho_0^2\partial_t\partial_x^4v+3\rho_0(\rho_0)_{x}\partial_t\partial_x^3v\|_{L^2}^2
-9\|\rho_0(\rho_0)_{x}\partial_t\partial_x^3v\|_{L^2}^2\\
&\qquad-3\int_I\rho_0^3(\rho_0)_x[(\partial_t\partial_x^3v)^2]_x\,\diff x\\
&\quad=\|\rho_0^2\partial_t\partial_x^4v+3\rho_0(\rho_0)_{x}\partial_t\partial_x^3v\|_{L^2}^2
+3\int_I\rho_0^3(\rho_0)_{xx}(\partial_t\partial_x^3v)^2\,\diff x\\
&\quad\leq M_0+Ct P(\sup_{0\leq s\leq t}E^{1/2}(s,v)),
\end{aligned}
\end{equation}
where one has used \eqref{II-Priori-ellip-4} and \eqref{II-Priori-ellip-17.2}.
Hence we obtain from \eqref{II-Priori-ellip-17.5} that
\begin{equation}\label{II-Priori-ellip-18}
\begin{aligned}
\|\rho_0^2\partial_t\partial_x^4v\|_{L^2}^2
\leq M_0+Ct P(\sup_{0\leq s\leq t}E^{1/2}(s,v)).
\end{aligned}
\end{equation}

\bigskip
\noindent{\bf{Estimate of \(\|\rho_0^3\partial_x^6v\|_{L^2(I)}\).}}
We first claim that
\begin{equation}\label{II-Priori-ellip-19}
\begin{aligned}
\bigg\|\rho_0\partial_x^5\bigg({\frac{\rho_0v_x}{\eta_x^2}}\bigg)\bigg\|_{L^2}^2\leq M_0+Ct P(\sup_{0\leq s\leq t}E^{1/2}(s,v)).
\end{aligned}
\end{equation}
Applying \(\partial_x^4\) to Equation \(\eqref{eq:main-2}_1\) gives
\begin{equation}\label{II-Priori-ellip-19.5}
\begin{aligned}
\partial_x^5\bigg(\frac{\rho_0v_x}{\eta_x^2}\bigg)=\partial_x^4(\rho_0\partial_tv)
+\partial_x^5\bigg(\frac{\rho_0^2}{\eta_x^2}\bigg).
\end{aligned}
\end{equation}
A direct calculation shows that
\begin{equation*}
\begin{aligned}
\bigg|\partial_x^5\bigg({\frac{\rho_0^2}{\eta_x^2}}\bigg)\bigg|
&\lesssim 1+|\eta_{xx}|+(\eta_{xx}^2+|\partial_x^3\eta|)
+(|\eta_{xx}|^3+|\eta_{xx}\partial_x^3\eta|+|\partial_x^4\eta|)\\
&\quad+\rho_0(|\eta_{xx}|^4+\eta_{xx}^2|\partial_x^3\eta|+(\partial_x^3\eta)^2+|\eta_{xx}\partial_x^4\eta|+|\partial_x^5\eta|)\\
&\quad+\rho_0^2(|\eta_{xx}|^5+|\eta_{xx}^3\partial_x^3\eta|+|\eta_{xx}|(\partial_x^3\eta)^2+\eta_{xx}^2|\partial_x^4\eta|\\
&\quad+|\partial_x^3\eta\partial_x^4\eta|+|\eta_{xx}\partial_x^5\eta|+|\partial_x^6\eta|).
\end{aligned}
\end{equation*}
Due to \eqref{Preliminary-7}, the estimate of the last term \(\rho_0^2\partial_x^6\eta\) requires a weight \(\rho_0\) ,
hence one may estimate as follows:
\begin{equation}\label{II-Priori-ellip-20}
\begin{aligned}
\bigg\|\rho_0\partial_x^5\bigg({\frac{\rho_0^2}{\eta_x^2}}\bigg)\bigg\|_{L^2}
&\lesssim 1+\|\eta_{xx}\|_{L^2}+(\|\eta_{xx}\|_{L^\infty}\|\eta_{xx}\|_{L^2}+\|\partial_x^3\eta\|_{L^2})\\
&\quad+(\|\eta_{xx}\|_{L^\infty}^2\|\eta_{xx}\|_{L^2}
+\|\eta_{xx}\|_{L^\infty}\|\partial_x^3\eta\|_{L^2}\\
&\quad+\|\rho_0\partial_x^4\eta\|_{L^2})
+(\|\eta_{xx}\|_{L^\infty}^3\|\eta_{xx}\|_{L^2}\\
&\quad+\|\eta_{xx}\|_{L^\infty}^2\|\partial_x^3\eta\|_{L^2}+\|\rho_0\partial_x^3\eta\|_{L^\infty}\|\partial_x^3\eta\|_{L^2}\\
&\quad+\|\eta_{xx}\|_{L^\infty}\|\rho_0\partial_x^4\eta\|_{L^2}+\|\rho_0^2\partial_x^5\eta\|_{L^2})\\
&\quad+(\|\eta_{xx}\|_{L^\infty}^4\|\eta_{xx}\|_{L^2}+\|\eta_{xx}\|_{L^\infty}^3\|\partial_x^3\eta\|_{L^2}\\
&\quad+\|\eta_{xx}\|_{L^\infty}\|\rho_0\partial_x^3\eta\|_{L^\infty}\|\partial_x^3\eta\|_{L^2}\\
&\quad+\|\eta_{xx}\|_{L^\infty}^2\|\rho_0\partial_x^4\eta\|_{L^2}+\|\rho_0\partial_x^3\eta\|_{L^\infty}\|\rho_0\partial_x^4\eta\|_{L^2}\\
&\quad+\|\eta_{xx}\|_{L^\infty}\|\rho_0^2\partial_x^5\eta\|_{L^2}
+\|\rho_0^3\partial_x^6\eta\|_{L^2})\\
&\leq 1+Ct P(\sup_{0\leq s\leq t}E^{1/2}(s,v)),
\end{aligned}
\end{equation}
where Lemma \ref{le:Preliminary-1} and Lemma \ref{le:Preliminary-2} have been used.
One the other hand, it holds that
\begin{equation}\label{II-Priori-ellip-21}
\begin{aligned}
\|\rho_0\partial_x^4(\rho_0\partial_tv)\|_{L^2}^2
&\lesssim \|\rho_0\partial_tv\|_{L^2}^2+\|\rho_0\partial_tv_x\|_{L^2}^2
+\|\rho_0\partial_tv_{xx}\|_{L^2}^2\\
&\quad+\|\rho_0\partial_t\partial_x^3v\|_{L^2}^2+\|\rho_0^2\partial_t\partial_x^4v\|_{L^2}^2\\
&\lesssim \|\rho_0\partial_tv\|_{L^2}^2+\|\rho_0\partial_tv_x\|_{L^2}^2
+\|\rho_0\partial_tv_{xx}\|_{L^2}^2\\
&\quad+\|\rho_0^2\partial_t\partial_x^3v\|_{L^2}^2+\|\rho_0^2\partial_t\partial_x^4v\|_{L^2}^2\\
&\leq M_0+Ct P(\sup_{0\leq s\leq t}E^{1/2}(s,v)).
\end{aligned}
\end{equation}
Here one has used \eqref{ineq:weighted Sobolev} for \(\|\rho_0\partial_t\partial_x^3v\|_{L^2}\) in the second inequality, and \eqref{II-Priori-ellip-18} in the last inequality.
Then \eqref{II-Priori-ellip-19} follows from \eqref{II-Priori-ellip-19.5}, \eqref{II-Priori-ellip-20} and \eqref{II-Priori-ellip-21}.

Next, we estimate \(\|\rho_0^3\partial_x^6v\|_{L^2}\).
To this end, one notes that
\begin{equation*}
\begin{aligned}
&\rho_0\eta_x^{-2}\partial_x^6v+5(\rho_0)_x\eta_x^{-2}\partial_x^5v\\
&=\partial_x^5\bigg(\frac{\rho_0v_x}{\eta_x^2}\bigg)-5\rho_0(\eta_x^{-2})_x\partial_x^5v-10(\rho_0\eta_x^{-2})_{xx}\partial_x^4v\\
&\quad-10\partial_x^3(\rho_0\eta_x^{-2})\partial_x^3v-5\partial_x^4(\rho_0\eta_x^{-2})v_{xx}-\partial_x^5(\rho_0\eta_x^{-2})v_{x}=\colon\sum_{k=1}^6I_k.
\end{aligned}
\end{equation*}
First, it follows from \eqref{II-Priori-ellip-19} that \(\|\rho_0I_1\|_{L^2}\) has the desired bound.
For the terms \(I_k\) when \(k=2,4,5\), in view of \eqref{I-Priori-time-16}, \eqref{I-Priori-ellip-5}, \eqref{II-Priori-ellip-3.4}, Lemma \ref{le:Preliminary-1} and Lemma \ref{le:Preliminary-2},
we may choose a weight \(\rho_0\) to estimate each term as follows:
\begin{equation*}
\begin{aligned}
\|\rho_0I_2\|_{L^2}
\lesssim \|\eta_{xx}\|_{L^\infty}\|\rho_0^2\partial_x^5v\|_{L^2}
\leq Ct P(\sup_{0\leq s\leq t}E^{1/2}(s,v)),
\end{aligned}
\end{equation*}
{\small \begin{equation*}
\begin{aligned}
\|\rho_0I_4\|_{L^2}
&\lesssim \|\rho_0\partial_x^3v\|_{L^2}+
\|\rho_0\partial_x^3v\|_{L^\infty}\big(\|\eta_{xx}\|_{L^2}
+(\|\eta_{xx}\|_{L^\infty}\|\eta_{xx}\|_{L^2}
+\|\partial_x^3\eta\|_{L^2})\\
&\quad+(\|\eta_{xx}\|_{L^\infty}^2\|\eta_{xx}\|_{L^2}
+\|\rho_0\partial_x^3\eta\|_{L^\infty}\|\eta_{xx}\|_{L^2}
+\|\rho_0\partial_x^4\eta\|_{L^2})\big)\\
&\leq [M_0+Ct P(\sup_{0\leq s\leq t}E^{1/2}(s,v))]^{1/2},
\end{aligned}
\end{equation*}}
{\small \begin{equation*}
\begin{aligned}
\|\rho_0I_5\|_{L^2}
&\lesssim \|\rho_0v_{xx}\|_{L^2}+
\|v_{xx}\|_{L^\infty}\big(\|\eta_{xx}\|_{L^2}
+(\|\rho_0\eta_{xx}\|_{L^\infty}\|\eta_{xx}\|_{L^2}
+\|\rho_0\partial_x^3\eta\|_{L^2})\\
&\quad+(\|\eta_{xx}\|_{L^\infty}^2\|\eta_{xx}\|_{L^2}
+\|\rho_0\partial_x^3\eta\|_{L^\infty}\|\eta_{xx}\|_{L^2}
+\|\rho_0\partial_x^4\eta\|_{L^2})\\
&\quad+(\|\eta_{xx}\|_{L^\infty}^3\|\eta_{xx}\|_{L^2}
+\|\eta_{xx}\|_{L^\infty}^2\|\partial_x^3\eta\|_{L^2}
+\|\rho_0\partial_x^3\eta\|_{L^\infty}\|\partial_x^3\eta\|_{L^2}\\
&\quad+\|\rho_0^2\partial_x^4\eta\|_{L^\infty}\|\eta_{xx}\|_{L^2}
+\|\rho_0^2\partial_x^5\eta\|_{L^2})\big)\\
&\leq [M_0+Ct P(\sup_{0\leq s\leq t}E^{1/2}(s,v))]^{1/2}.
\end{aligned}
\end{equation*}}

For \(I_3\) and \(I_6\), we choose a weight \(\rho_0^2\) to estimate them as follows:
\begin{equation*}
\begin{aligned}
\|\rho_0^2I_3\|_{L^2}
&\lesssim \|\rho_0^2\partial_x^4v\|_{L^2}(1+\|\eta_{xx}\|_{L^\infty}
+\|\eta_{xx}\|_{L^\infty}^2
+\|\rho_0\partial_x^3\eta\|_{L^\infty})\\
&\leq [M_0+Ct P(\sup_{0\leq s\leq t}E^{1/2}(s,v))]^{1/2},
\end{aligned}
\end{equation*}
and 
{\small \begin{equation*}
\begin{aligned}
\|\rho_0^2I_6\|_{L^2}
&\lesssim \|\rho_0v_{x}\|_{L^2}+
\|v_{x}\|_{L^\infty}\big(\|\eta_{xx}\|_{L^2}
+(\|\eta_{xx}\|_{L^\infty}\|\eta_{xx}\|_{L^2}
+\|\partial_x^3\eta\|_{L^2})\\
&\quad+(\|\eta_{xx}\|_{L^\infty}^2\|\eta_{xx}\|_{L^2}
+\|\eta_{xx}\|_{L^\infty}\|\partial_x^3\eta\|_{L^2}
+\|\rho_0\partial_x^4\eta\|_{L^2})\\
&\quad+(\|\eta_{xx}\|_{L^\infty}^3\|\eta_{xx}\|_{L^2}
+\|\eta_{xx}\|_{L^\infty}^2\|\partial_x^3\eta\|_{L^2}
+\|\rho_0\partial_x^3\eta\|_{L^\infty}\|\partial_x^3\eta\|_{L^2}\\
&\quad+\|\eta_{xx}\|_{L^\infty}\|\rho_0\partial_x^4\eta\|_{L^2}
+\|\rho_0^2\partial_x^5\eta\|_{L^2})
+(\|\eta_{xx}\|_{L^\infty}^4\|\eta_{xx}\|_{L^2}\\
&\quad+\|\eta_{xx}\|_{L^\infty}^3\|\partial_x^3\eta\|_{L^2}
+\|\eta_{xx}\|_{L^\infty}\|\rho_0\partial_x^3\eta\|_{L^\infty}\|\partial_x^3\eta\|_{L^2}\\
&\quad+\|\eta_{xx}\|_{L^\infty}^2\|\rho_0\partial_x^4\eta\|_{L^2}+\|\rho_0\partial_x^3\eta\|_{L^\infty}\|\rho_0\partial_x^4\eta\|_{L^2}\\
&\quad+\|\eta_{xx}\|_{L^\infty}\|\rho_0^2\partial_x^5\eta\|_{L^2}
+\|\rho_0^3\partial_x^6\eta\|_{L^2})\big)\\
&\leq [M_0+Ct P(\sup_{0\leq s\leq t}E^{1/2}(s,v))]^{1/2},
\end{aligned}
\end{equation*}}
where in estimating \(I_6\) one has used
{\small \begin{equation*}
\begin{aligned}
|\partial_x^5(\rho_0\eta_x^{-2})|
&\lesssim 1+|\eta_{xx}|+(\eta_{xx}^2+|\partial_x^3\eta|)
+(|\eta_{xx}|^3+|\eta_{xx}\partial_x^3\eta|+|\partial_x^4\eta|)\\
&\quad+(|\eta_{xx}|^4+\eta_{xx}^2|\partial_x^3\eta|+(\partial_x^3\eta)^2+|\eta_{xx}\partial_x^4\eta|+|\partial_x^5\eta|)\\
&\quad+\rho_0(|\eta_{xx}|^5+|\eta_{xx}^3\partial_x^3\eta|+|\eta_{xx}|(\partial_x^3\eta)^2+\eta_{xx}^2|\partial_x^4\eta|\\
&\quad+|\partial_x^3\eta\partial_x^4\eta|+|\eta_{xx}\partial_x^5\eta|+|\partial_x^6\eta|).
\end{aligned}
\end{equation*}}
Consequently,
{\small \begin{equation}\label{II-Priori-ellip-23}
\begin{aligned}
&\|\rho_0^3\eta_x^{-2}\partial_x^5v+5\rho_0^2(\rho_0)_x\eta_x^{-2}\partial_x^4v\|_{L^2}^2\\
&\quad\lesssim \|\rho_0^2I_1\|_{L^2}^2+\|\rho_0^2I_2\|_{L^2}^2+\|\rho_0^2I_3\|_{L^2}^2+\|\rho_0^2I_4\|_{L^2}^2+\|\rho_0^2I_5\|_{L^2}^2+\|\rho_0^2I_6\|_{L^2}^2\\
&\quad\lesssim \|\rho_0I_1\|_{L^2}^2+\|\rho_0I_2\|_{L^2}+\|\rho_0^2I_3\|_{L^2}^2+\|\rho_0I_4\|_{L^2}^2+\|\rho_0I_5\|_{L^2}^2+\|\rho_0^2I_6\|_{L^2}^2\\
&\quad\leq M_0+Ct P(\sup_{0\leq s\leq t}E^{1/2}(s,v)).
\end{aligned}
\end{equation}}
Then integration by parts yields
{\small \begin{equation}\label{II-Priori-ellip-24}
\begin{aligned}
&\|\rho_0^3\eta_x^{-2}\partial_x^6v\|_{L^2}^2\\
&\quad=\|\rho_0^3\eta_x^{-2}\partial_x^6v
+5\rho_0^2(\rho_0)_x\eta_x^{-2}\partial_x^5v\|_{L^2}^2
-25\|\rho_0^2(\rho_0)_x\eta_x^{-2}\partial_x^5v\|_{L^2}^2\\
&\qquad-5\int_I\rho_0^5(\rho_0)_x\eta_x^{-4}[(\partial_x^5v)^2]_x\,\diff x\\
&\quad=\|\rho_0^3\eta_x^{-2}\partial_x^6v
+5\rho_0^2(\rho_0)_x\eta_x^{-2}\partial_x^5v\|_{L^2}^2
+5\int_I\rho_0^5[(\rho_0)_x\eta_x^{-4}]_{x}(\partial_x^5v)^2\,\diff x\\
&\quad\lesssim\|\rho_0^3\eta_x^{-2}\partial_x^6v
+5\rho_0^2(\rho_0)_x\eta_x^{-2}\partial_x^5v\|_{L^2}^2
+(1+\|\eta_{xx}\|_{L^\infty})
\int_I\rho_0^5(\partial_x^5v)^2\,\diff x\\
&\quad\leq M_0+Ct P(\sup_{0\leq s\leq t}E^{1/2}(s,v)),
\end{aligned}
\end{equation}}
where \eqref{Preliminary-5}, \eqref{II-Priori-ellip-11} and \eqref{II-Priori-ellip-23} have been used.
Hence we get from \eqref{II-Priori-ellip-24} and \eqref{eta-bound} that
\begin{equation}\label{II-Priori-ellip-25}
\begin{aligned}
\|\rho_0^3\partial_x^6v\|_{L^2}^2\leq M_0+Ct P(\sup_{0\leq s\leq t}E^{1/2}(s,v)).
\end{aligned}
\end{equation}

\section{An a priori Bound}\label{A priori Bound}
Collecting all inequalities \eqref{II-Priori-time-4}-\eqref{I-Priori-time-3} and \eqref{II-Priori-time-10}-\eqref{I-Priori-time-16}  in Section \ref{Energy Estimates}, \eqref{I-Priori-ellip-5}, \eqref{I-Priori-ellip-third}, \eqref{I-Priori-ellip-22}, \eqref{I-Priori-ellip-31}, \eqref{II-Priori-ellip-4}, \eqref{II-Priori-ellip-11}, \eqref{II-Priori-ellip-18} and \eqref{II-Priori-ellip-25} in Section \ref{Elliptic estimates}, we obtain
\begin{equation}\label{I-Priori-ellip-32}
\begin{aligned}
E(t,v)
\leq M_0+CtP(\sup_{0\leq s\leq t}E^{1/2}(s,v))\quad \mathrm{for\ all}\ t\in [0,T],
\end{aligned}
\end{equation}
where \(P\) denotes a generic polynomial function of its arguments, and \(C\) is an absolutely constant only depending on \(\|\partial_x^l\rho_0\|_{L^\infty(I)}\ (l=0,1,...,5)\). The inequality
\eqref{I-Priori-ellip-32} implies for  sufficiently small \(T>0\),
\begin{equation}\label{I-Priori-ellip-33}
	\begin{aligned}
		\sup_{0\leq t\leq T}E(t,v)
		\leq 2M_0.
	\end{aligned}
\end{equation}

\section{Proof of Theorem \ref{th:main-1}: Existence}\label{Existence Part}

In this section, we will show the existence of a classical solution to the problem \eqref{eq:main-2}. For given \(T>0\), let \(\mathcal{X}_T\) be a Banach space defined by
\begin{equation*}
	\begin{aligned}
		\mathcal{X}_T=\{ v\in L^\infty([0,T]; H^3(I)):\ \sup_{0\leq t\leq T}E(t,v)<\infty\},
	\end{aligned}
\end{equation*}
endowed with its natural norm
\begin{equation*}
	\begin{aligned}
		\|v\|_{\mathcal{X}_T}^2=\sup_{0\leq t\leq T}E(t,v).
	\end{aligned}
\end{equation*}
For given \(M_1\), we define \(\mathcal{C}_T(M_1)\) to be a closed, bounded, and
convex subset of \(\mathcal{X}_T\) given by
\begin{equation}\label{solution space}
	\begin{aligned}
		\mathcal{C}_T(M_1)&=\{ v\in \mathcal{X}_T:\|v\|_{\mathcal{X}_T}^2\leq M_1
		,\ \partial_t^kv|_{t=0}=g_k
		\ \mbox{for}\  k=0,1,2,3,\\ 
          &\quad\quad\quad\quad\quad\quad\quad\quad\quad\quad  \mbox{and}\  \partial_t^kv_x|_{t=0}=h_k\ \mbox{for}\  k=0,1,2\},
	\end{aligned}
\end{equation}
where \(g_k\) and \(h_k\) are defined as follows:
\begin{equation*}
	\begin{aligned}
		&g_0=v|_{t=0}=u_0,\\
		&g_1=\partial_tv|_{t=0}=\rho_0^{-1}\bigg[\bigg(\frac{\rho_0v_x}{\eta_x^2}\bigg)_x
		-\bigg({\frac{\rho_0^2}{\eta_x^2}}\bigg)_x\bigg]\bigg|_{t=0}
		=\rho_0^{-1}[(\rho_0(u_0)_x)_x-(\rho_0^2)_x],\\
	&g_k:=\partial_t^kv|_{t=0}=\rho_0^{-1}\partial_t^{k-1}\bigg[\bigg(\frac{\rho_0v_x}{\eta_x^2}\bigg)_x
	-\bigg({\frac{\rho_0^2}{\eta_x^2}}\bigg)_x\bigg]\bigg|_{t=0}\quad \text{for}\ k=2,3,\\
&h_0:=v_x|_{t=0}=(u_0)_x,\\
& h_k:=\partial_t^kv_x|_{t=0}=(g_k)_x\quad \text{for}\  k=1,2.
	\end{aligned}
\end{equation*}
Note that each \(g_k\) (\(k=0,1,2,3\)) and \(h_k\) (\(k=1,2\)) is a function of spatial derivatives of \(\rho_0\) and \(u_0\).

For any given \(\bar{v}\in\mathcal{C}_T(M_1)\), define
\begin{align}\label{existence-1}
	\bar{\eta}(x,t)=x+\int_0^t\bar{v}(x,s)\,\diff s.
\end{align}
Arguing as for \eqref{eta-bound}, by choosing \(T>0\) suitably small, one also has
\begin{align}\label{eta-bound-2}
	1/2\leq \bar{\eta}_x(x,t)\leq 3/2,\quad (x,t)\in I\times [0,T].
\end{align}
The choice of  \(M_1\) and \(T\) is given in Subsection \ref{The a priori assumption}.
We then consider the following
linearized problem for \(v\):
\begin{equation}\label{existence-3}
	\begin{cases}
		\rho_0v_t+\big({\frac{\rho_0^2}{\bar{\eta}_x^2}}\big)_x
		=\big(\frac{\rho_0v_x}{\bar{\eta}_x^2}\big)_x &\quad \mbox{in}\ I\times (0,T],\\
		v=u_0 &\quad \mbox{on}\ I\times \{t=0\}.
	\end{cases}
\end{equation}

In order to construct classical solutions to the problem \eqref{existence-3}, we first study its weak solutions.
\subsection{Existence and uniqueness of a weak solution to the problem \eqref{existence-3}.}\label{weak solution}
 Let \(\langle\cdot, \cdot\rangle\) be the pairing of \(H^{-1}(I)\) and \(H^1(I)\), and \((\cdot,\cdot)\) stand for the inner product of \(L^2(I)\). Then we give the following definition:
\begin{definition}[Weak Solution]\label{Weak Solution} A function \(v\), satisfying
	\begin{equation*}
		\begin{aligned}
		\rho_0^{1/2}v_x\in L^2([0,T];L^2(I))\quad {\rm{and}}\quad \rho_0 v_t\in L^2([0,T];H^{-1}(I)),
		\end{aligned}
	\end{equation*}
	is said to be a weak solution to the problem \eqref{existence-3} provided\\	
	{\rm{(a)}}
	\begin{equation*}
		\begin{aligned}
			\langle\rho_0v_t, \phi\rangle+\bigg(\frac{\rho_0v_x}{\bar{\eta}_x^2}, \phi_x\bigg)=\bigg({\frac{\rho_0^2}{\bar{\eta}_x^2}}, \phi_x\bigg)
		\end{aligned}
	\end{equation*}	
	for each \(\phi\in H^1(I)\) and a.e. \(0< t\leq T\), and\\
	{\rm{(b)}} \(\|\rho_0v(t,\cdot)-\rho_0v(0, \cdot)\|_{L^2(I)}\to 0\ \text{as}\ t\to 0^+\), and $v(0,\cdot)=u_0(\cdot) \ \mbox{a.e.\ on}\ I$. 
\end{definition}

We will use the Galerkin's scheme (see \cite{MR2597943}) to construct weak solutions to  the problem \eqref{existence-3}. Set
\[\mathcal{H}(I)=\{h\in H^3(I): h_x=0 \ \text{on}\ \Gamma\}.\]
Let \(\{e_n\}_{n=1}^\infty\) be a Hilbert basis of \(\mathcal{H}(I)\), with each \(e_n\) being of class \(H^k(I)\) for any \(k\geq 1\).  Such a choice of basis indeed exists since one can take for instance
the eigenfunctions of the Laplace operator on \(I\) with the Nuewmann boundary
condition \(h_x=0\) for \(x\in \Gamma\).
Given a positive integer \(n\), we set
\begin{equation}\label{existence-4}
	\begin{aligned}
		X^n(t,x)=\sum_{i=1}^n\lambda_i^n(t)e_i(x),
	\end{aligned}
\end{equation}
in which the coefficients \(\lambda_i^n(t)\) are chosen such that
\begin{eqnarray}\label{existence-5}
	\begin{cases}
		\big(\rho_0\partial_tX^n,e_j\big)
		+\bigg(\frac{\rho_0X_x^n}{\bar{\eta}_x^2},(e_j)_x\bigg)
		=\bigg(\frac{\rho_0^2}{\bar{\eta}_x^2},(e_j)_x\bigg)
		&\quad \mbox{in}\ (0,T],\\
		\lambda_j^n=(u_0,e_j) &\quad \mbox{on}\ \{t=0\},
	\end{cases}
\end{eqnarray}
where \(j=1,2,...,n\).
Inserting \eqref{existence-4} into \eqref{existence-5} leads to
\begin{equation}\label{existence-6}
	\begin{cases}
		\sum_{i=1}^n\int_I \rho_0 e_ie_j\,\diff x\cdot  [\lambda_j^n(t)\big]_t\\
		\quad+\sum_{i=1}^n\int_I\frac{\rho_0(e_i)_x(e_j)_x}{\bar{\eta}_x^2}\,\diff x\cdot  \lambda_j^n(t)
		=\int_I \frac{\rho_0^2 (e_j)_x}{\bar{\eta}_x^2}\,\diff x	&\quad \mbox{in}\ (0,T],\\
		\lambda_j^n=(u_0,e_j) &\quad \mbox{on}\ \{t=0\},
	\end{cases}
\end{equation}
where \(j=1,2,...,n\).

It is clear that each integral in \eqref{existence-6} is well-defined since each \(e_i\) lives in \(H^k(I)\cap \mathcal{H}(I)\) for all \(k\geq 1\). On the one hand, the \(\{e_n\}_{n=1}^\infty\) are linearly independent, so are the \(\{\sqrt{\rho_0}e_n\}_{n=1}^\infty\). Hence the determinant of the matrix
\begin{align*}
	[\sqrt{\rho_0}e_i,\sqrt{\rho_0}e_j]_{i,j\in\{1,\dots,n\}}
\end{align*}
is nonzero. On the other hand, it follows from \(\bar{v}\in \mathcal{C}_T(M_1)\) and \eqref{eta-bound-2} that \(1/\bar{\eta}_x\) is continuous for \(t\in[0,T]\), which implies 
\begin{align*} 
	\int_I\frac{\rho_0(e_i)_x(e_j)_x}{\bar{\eta}_x^2}\,\diff x
\end{align*}
is continuous, and
\begin{align*} 
     \int_I \frac{\rho_0^2 (e_j)_x}{\bar{\eta}_x^2}\,\diff x
\end{align*}
is Lipschitz continuous for \(t\in[0,T]\). By the standard ODEs' theory, one can find solutions \(\lambda_i^n(t)\in C^1([0,T_n])\) \ \((i=1,...,n)\) to \eqref{existence-6}, which
means there exist approximate solutions \(X^n(t,x)\in C^1([0,T_n],\mathcal{H}(I))\)\ \((n=1,2,...)\) to \eqref{existence-5}.\\

We next show that \(\{X^n\}_{n=1}^\infty\)  satisfy some uniform estimates in \(n\geq 1\).

\begin{lemma}\label{uniform estimates} The approximate solutions \(\{X^n\}_{n=1}^\infty\) satisfy the following uniform estimates in \(n\geq 1\):
\begin{equation}\label{first uniform estimates}
	\begin{aligned}
		&\underset{t\in[0,T]}{\sup}\|\rho_0^{1/2}X^n\|_{L^2(I)}^2
		+\|\rho_0^{1/2}X_x^n\|_{L^2([0,T];L^2(I))}^2+\|\rho_0 X_t^n\|_{L^2([0,T];H^{-1}(I))}^2\\
		&\quad\leq C\|\rho_0^{1/2}u_0\|_{L^2(I)}^2+CT.
	\end{aligned}
\end{equation}
	
\end{lemma}

\begin{proof}
It follows from  \eqref{existence-4} and \(\eqref{existence-5}_1\) that
\begin{align*}
	\big(\rho_0\partial_tX^n,X^n\big)
	+\bigg(\frac{\rho_0\partial_xX^n}{\bar{\eta}_x^2},\partial_xX^n\bigg)
	=\bigg(\frac{\rho_0^2}{\bar{\eta}_x^2},\partial_xX^n\bigg).
\end{align*}
Integrating it over \(I\times[0,T_n]\) and integration by parts yield
\begin{equation}\label{existence-7}
	\begin{aligned}
		&\frac{1}{2}\int_I \rho_0(X^n)^2\,\diff x
		+\int_0^{T_n}\int_I\frac{\rho_0(\partial_xX^n)^2}{\bar{\eta}_x^2}\,\diff x\diff s\\
		&=\frac{1}{2}\int_I \rho_0(X^n)^2(x,0)\,\diff x
		+\int_0^{T_n}\int_I\frac{\rho_0^2\partial_xX^n}{\bar{\eta}_x^2}\,\diff x\diff s.
	\end{aligned}
\end{equation}
\eqref{eta-bound-2} and Cauchy's inequality imply
\begin{equation}\label{existence-8}
	\begin{aligned}
		\int_0^{T_n}\int_I\frac{\rho_0(\partial_xX^n)^2}{\bar{\eta}_x^2}\,\diff x\diff s
		\geq \frac{4}{9} \int_0^{T_n}\int_I\rho_0(\partial_xX^n)^2\,\diff x\diff s,
	\end{aligned}
\end{equation}
and
\begin{equation}\label{existence-9}
	\begin{aligned}
		\bigg|\int_0^{T_n}\int_I\frac{\rho_0^2\partial_xX^n}{\bar{\eta}_x^2}\,\diff x\diff s\bigg|
		\leq CT_n+\frac{1}{100}\int_0^{T_n}\int_I\rho_0(\partial_xX^n)^2\,\diff x\diff s.
	\end{aligned}
\end{equation}
Hence it follows from \eqref{existence-7}-\eqref{existence-9} that
\begin{equation}\label{existence-10}
	\begin{aligned}
		\int_I \rho_0(X^n)^2\,\diff x
		+\int_0^{T_n}\int_I\rho_0(\partial_xX^n)^2\,\diff x\diff s
		\leq C\|\rho_0^{1/2}u_0\|_{L^2(I)}^2+CT_n.
	\end{aligned}
\end{equation}

Fix any \(\phi\in H^1(I)\) with \(\|\phi\|_{H^1(I)}\leq 1\), and write \(\phi=\phi_1+\phi_2\), where 
\begin{equation*}
	\begin{aligned}
\phi_1\in \text{span}\{e_i\}_{i=1}^n\quad {\rm{and}}\quad (\phi_2,e_i)=0\ (i=1,...,n).
	\end{aligned}
\end{equation*}
Recalling that the functions \(\{e_i\}_{i=1}^n\) are orthogonal in \(H^1(I)\), one has
\begin{equation*}
	\begin{aligned}
\|\phi_1\|_{H^1(I)}\leq \|\phi\|_{H^1(I)}\leq 1.
	\end{aligned}
\end{equation*}
It follows from \(\eqref{existence-5}_1\)  that 
\begin{equation}\label{existence-11}
	\begin{aligned}
(\rho_0X_t^n, \phi_1)+\bigg(\frac{\rho_0X_x^n}{\bar{\eta}_x^2}, (\phi_1)_x\bigg)=\bigg({\frac{\rho_0^2}{\bar{\eta}_x^2}}, (\phi_1)_x\bigg),
	\end{aligned}
\end{equation}
for a.e. \(0\leq t\leq T\).
Hence \eqref{existence-11} yields 
\begin{equation*}
	\begin{aligned}
\langle\rho_0X_t^n, \phi\rangle
&=(\rho_0X_t^n, \phi)=(\rho_0X_t^n, \phi_1)\\
&=\bigg({\frac{\rho_0^2}{\bar{\eta}_x^2}}, (\phi_1)_x\bigg)-\bigg(\frac{\rho_0X_x^n}{\bar{\eta}_x^2}, (\phi_1)_x\bigg),
	\end{aligned}
\end{equation*}
which furthermore implies 
\begin{equation*}
	\begin{aligned}
		|\langle\rho_0X_t^n, \phi\rangle|&\leq C(1+\|\rho_0^{1/2}X_x^n\|_{L^2})\|\phi_1\|_{H^1(I)}\\
		&\leq C(1+\|\rho_0^{1/2}X_x^n\|_{L^2}).
	\end{aligned}
\end{equation*}
This results in 
\begin{equation*}
	\begin{aligned}
	  \|\rho_0X_t^n\|_{H^{-1}(I)}\leq C(1+\|\rho_0^{1/2}X_x^n\|_{L^2}),
	\end{aligned}
\end{equation*}
and therefore 
\begin{equation}\label{existence-12}
	\begin{aligned}
		\int_0^{T_n}\|\rho_0X_t^n\|_{H^{-1}(I)}^2\,\diff t&\leq C\int_0^{T_n}\int_I\rho_0(X_x^n)^2\,\diff x\,\diff s+CT_n\\
		&\leq C\|\rho_0^{1/2}u_0\|_{L^2(I)}^2+CT_n,
	\end{aligned}
\end{equation}
due to \eqref{existence-10}.

It follows from \eqref{existence-10} and \eqref{existence-12} that
\begin{equation}\label{existence-13}
	\begin{aligned}
		&\underset{t\in[0,T_n]}{\sup}\|\rho_0^{1/2}X^n\|_{L^2(I)}^2
		+\|\rho_0^{1/2}X_x^n\|_{L^2([0,T_n];L^2(I))}^2+\|\rho_0 X_t^n\|_{L^2([0,T_n];H^{-1}(I))}^2\\
		&\quad\leq C\|\rho_0^{1/2}u_0\|_{L^2(I)}^2+CT_n.
	\end{aligned}
\end{equation}
Note that \eqref{eta-bound-2} holds on \(I\times[0,T]\), hence \(T_n\) can reach \(T\).  Consequently \eqref{first uniform estimates} follows from \eqref{existence-13}.

\end{proof}

Finally, we show the existence of a weak solution to the problem \eqref{existence-3}. 

\begin{lemma} There exists a unique weak solution \(v\) to the problem \eqref{existence-3} with
	\begin{equation*}
		\begin{aligned}
			&\rho_0^{1/2}v\in L^\infty([0,T],L^2(I)),\ \rho_0^{1/2}v_x\in L^2([0,T],L^2(I)),\\
			&\rho_0\partial_tv\in L^2([0,T];H^{-1}(I)).
		\end{aligned}
	\end{equation*}	
	Moreover, the solution \(v\) satisfies the following estimate	
	\begin{equation}\label{solution bound}
		\begin{aligned}
			&\underset{t\in[0,T]}{\sup}\|\rho_0^{1/2}v\|_{L^2(I)}^2
			+\|\rho_0^{1/2}v_x\|_{L^2([0,T];L^2(I))}^2
			+\|\rho_0v_t\|_{L^2([0,T];H^{-1}(I))}^2\\
			&\quad\leq C|\rho_0^{1/2}u_0\|_{L^2(I)}^2+C(1+T).
		\end{aligned}
	\end{equation}	
	
\end{lemma}

\begin{proof}
It follows from  Lemma \ref{uniform estimates} that
\begin{equation*}
	\begin{aligned}
	\|\rho_0^{1/2}X_x^n\|_{L^2([0,T];L^2(I))}\quad {\rm{and}}\quad \|\rho_0 X_t^n\|_{L^2([0,T];H^{-1}(I))}
	\end{aligned}
\end{equation*}	
are uniformly bounded in \(n\geq 1\).  So there exist a subsequence of \(\{X^n\}_{n=1}^\infty\) (which is still denoted by \(\{X^n\}_{n=1}^\infty\) for convenience) and a function \(v\) satisfying \(\rho_0^{1/2}v_x \in\ L^2([0,T];L^2(I))\) and \(\rho_0v_t \in\ L^2([0,T];H^{-1}(I))\) such that as \(n\to\infty\)
\begin{equation*}
	\begin{cases}
		\rho_0^{1/2}X_x^n\rightharpoonup \rho_0^{1/2}v_x &\quad \text{in}\ L^2([0,T];L^2(I)),\\
		\rho_0X_t^n\rightharpoonup \rho_0v_t &\quad \text{in}\ L^2([0,T];H^{-1}(I)).	
	\end{cases}
\end{equation*}	
Then, the estimate \eqref{solution bound} follows easily from the energy estimates \eqref{first uniform estimates} by the lower semi-continuity of the norms.

We claim that \(v\) is a weak solution to the problem \eqref{existence-3}. Fix any positive integer \(m\geq 1\), and choose a function  \(\Phi\in C^1([0,T]; H^1(I))\)
of the form 
	\begin{equation}\label{existence-14}
	\begin{aligned}
		\Phi=\sum_{i=1}^m\mu_i(t)e_i(x),
	\end{aligned}
\end{equation}	
where \(\{\mu_i(t)\}_{i=1}^m\) are any given smooth functions. Choosing \(n\geq m\), multiplying \(\eqref{existence-5}_1\) by \(\mu_i(t)\), summing up for \(i=1,...,m\), and integrating with respect to \(t\) over \([0,T]\), we get
\begin{equation}\label{existence-15}
	\begin{aligned}
		\int_0^T\langle\rho_0X_t^n, \Phi\rangle+\bigg(\frac{\rho_0X_x^n}{\bar{\eta}_x^2}, \Phi_x\bigg)\,\diff t=\int_0^T\bigg({\frac{\rho_0^2}{\bar{\eta}_x^2}}, \Phi_x\bigg)\,\diff t.
	\end{aligned}
\end{equation}
Taking the limit \(n\rightarrow\infty\) yields
\begin{equation}\label{existence-16}
	\begin{aligned}
		\int_0^T\langle\rho_0v_t, \Phi\rangle+\bigg(\frac{\rho_0v_x}{\bar{\eta}_x^2}, \Phi_x\bigg)\,\diff t=\int_0^T\bigg({\frac{\rho_0^2}{\bar{\eta}_x^2}}, \Phi_x\bigg)\,\diff t.
	\end{aligned}
\end{equation}

Since functions of the form \eqref{existence-14} are dense in \(C([0,T]; H^1(I))\), 
\eqref{existence-16} holds for all  \(\Phi\in C^1([0,T]; H^1(I))\).	
In particular, it holds that
\begin{equation}\label{existence-weak solution}
	\begin{aligned}
\langle\rho_0v_t, \phi\rangle+\bigg(\frac{\rho_0v_x}{\bar{\eta}_x^2}, \phi_x\bigg)=\bigg({\frac{\rho_0^2}{\bar{\eta}_x^2}}, \phi_x\bigg)
	\end{aligned}
\end{equation}
for each \(\phi\in H^1(I)\) and a.e. \(0< t\leq T\).  

By Definition \ref{Weak Solution}, it remains to check that 
\begin{equation}\label{initial-add-1}
	\begin{aligned}
\|\rho_0v(t,\cdot)-\rho_0v(0, \cdot)\|_{L^2(I)}\to 0\quad \text{as}\ t\to 0^+,
	\end{aligned}
\end{equation}
and
\begin{equation}\label{initial-add-2}
	\begin{aligned}
v(0):=v(0, \cdot)=u_0(0, \cdot) \quad \text{a.e.\ in}\ I.
	\end{aligned}
\end{equation}
First note that
\begin{equation*}
		\begin{aligned}
			&\|\rho_0v\|_{L^2([0,T],H^1(I))}^2\\
&\quad\lesssim \|\rho_0^{1/2}v\|_{L^2([0,T],L^2(I))}^2+\|\rho_0^{1/2}v_x\|_{L^2([0,T],L^2(I))}^2
                                                +\|v\|_{L^2([0,T],L^2(I))}^2\\
                                 &\quad\lesssim \|\rho_0^{1/2}v\|_{L^2([0,T],L^2(I))}^2+\|\rho_0^{1/2}v_x\|_{L^2([0,T],L^2(I))}^2
+\|v\|_{L^2([0,T],H^{1/2}(I))}^2\\
                                 &\quad\lesssim \|\rho_0^{1/2}v\|_{L^2([0,T],L^2(I))}^2+\|\rho_0^{1/2}v_x\|_{L^2([0,T],L^2(I))}^2\\
                                 &\quad\lesssim \|\rho_0^{1/2}v\|_{L^\infty([0,T],L^2(I))}^2+\|\rho_0^{1/2}v_x\|_{L^2([0,T],L^2(I))}^2,
		\end{aligned}
	\end{equation*}	
where \eqref{ineq:weighted Sobolev-0} has been used in the third inequality. Hence
\begin{equation*}
	\begin{aligned} 
\rho_0v\in L^2([0,T],H^1(I)),
	\end{aligned}
\end{equation*}
which together with \(\rho_0\partial_tv\in L^2([0,T];H^{-1}(I))\) yields 
\begin{equation}\label{initial-add-3}
	\begin{aligned} 
\rho_0v\in C([0,T],L^2(I)).
	\end{aligned}
\end{equation}
Thus \eqref{initial-add-1} follows. 
Then one may deduce from \eqref{existence-16} and \eqref{initial-add-3} that  
\begin{equation}\label{existence-17}
	\begin{aligned}
		\int_0^T-\langle\Phi_t, \rho_0v\rangle+\bigg(\frac{\rho_0v_x}{\bar{\eta}_x^2}, \Phi_x\bigg)\,\diff t=\int_0^T\bigg({\frac{\rho_0^2}{\bar{\eta}_x^2}}, \Phi_x\bigg)\,\diff t+(\rho_0v(0), \Phi(0))
	\end{aligned}
\end{equation}
for each \(\Phi\in C^1([0,T]; H^1(I))\) with \(\Phi(T)=0\). For this \(\Phi\), it follows from \eqref{existence-15} that 
\begin{equation}\label{existence-18}
	\begin{aligned}
		\int_0^T-\langle\Phi_t, \rho_0X^n\rangle+\bigg(\frac{\rho_0X_x^n}{\bar{\eta}_x^2}, \Phi_x\bigg)\,\diff t=\int_0^T\bigg({\frac{\rho_0^2}{\bar{\eta}_x^2}}, \Phi_x\bigg)\,\diff t+(\rho_0X^n(0), \Phi(0)).
	\end{aligned}
\end{equation}
Passing limits in \(n\to \infty\) in \eqref{existence-18} gives 
\begin{equation}\label{existence-19}
	\begin{aligned}
		\int_0^T-\langle\Phi_t, \rho_0v\rangle+\bigg(\frac{\rho_0v_x}{\bar{\eta}_x^2}, \Phi_x\bigg)\,\diff t=\int_0^T\bigg({\frac{\rho_0^2}{\bar{\eta}_x^2}}, \Phi_x\bigg)\,\diff t+(\rho_0u_0, \Phi(0)),
	\end{aligned}
\end{equation}
where one has used the fact \(\|\rho_0^{1/2}X^n(0)-\rho_0^{1/2}u_0\|_{L^2(I)}\to 0\) as \(n\to\infty\). As \(\Phi(0)\) is arbitrary, comparing \eqref{existence-17} and \eqref{existence-19}, one gets 
\begin{equation*}
	\begin{aligned}
		\|\rho_0v(0)-\rho_0u_0\|_{L^2(I)}=0,
	\end{aligned}
\end{equation*}
which yields
\begin{equation*}
	\begin{aligned}
		\rho_0v(0)=\rho_0u_0\quad a.e.\ \text{in}\ I.
	\end{aligned}
\end{equation*}
Hence \eqref{initial-add-2} follows due to \eqref{eq:intro-3}.

The uniqueness of weak solutions of the problem \eqref{existence-3} is easy to check since \eqref{existence-3} is a linear problem.

\end{proof}

\subsection{Regularity.}\label{Regularity}
We have the following regularity result:
\begin{lemma}\label{Regularity lemma} The weak solution \(v\) to the problem \eqref{existence-3} has the following regularity: 
	\begin{equation}\label{existence-20}
		\begin{aligned}
			\sup_{0\leq t\leq T}E(t,v)
			\leq M_1.
		\end{aligned}
	\end{equation}
Consequently the solution map \(\bar{v}\mapsto v:\mathcal{C}_T(M_1)\rightarrow \mathcal{C}_T(M_1)\) is well-defined.

\end{lemma}

\begin{proof} To prove \eqref{existence-20}, it suffices to show
\begin{equation}\label{existence-add-0}
\begin{aligned}
E(t,v)
\leq M_0+CtP(\sup_{0\leq s\leq t}E^{1/2}(s,\bar{v}))\quad \mathrm{for\ all}\ t\in [0,T]
\end{aligned}
\end{equation}
whose proof is similar to that of \eqref{I-Priori-ellip-32} in Section \ref{Energy Estimates} and Section \ref{Elliptic estimates}. So we only sketch the proof of \eqref{existence-add-0} and point out the main modifications. 

\noindent{\bf{Estimate of \(\|\sqrt{\rho_0}\partial_tv\|_{L^2(I)}\).}} We start with estimating \(\|\sqrt{\rho_0}\partial_tv\|_{L^2(I)}\) based on \eqref{first uniform estimates} by some basic energy estimates.
To this end,
one can apply \(\partial_t\) to \(\eqref{existence-5}_1\), multiply it by \(\partial_t\lambda_i^n(t)\), and sum \(j=1,2,...,n\), to obtain that
\begin{equation*}
		\big(\rho_0\partial_t^2X^n,\partial_tX^n\big)
		+\bigg(\partial_t\bigg(\frac{\rho_0X_x^n}{\bar{\eta}_x^2}\bigg),\partial_tX_x^n\bigg)
		=\bigg(\partial_t\bigg(\frac{\rho_0^2}{\bar{\eta}_x^2}\bigg),\partial_tX_x^n\bigg),
\end{equation*}
which gives
\begin{equation}\label{existence-add-1}
\begin{aligned}
&\frac{1}{2}\int_I \rho_0(\partial_tX^n)^2\,\diff x
+\int_0^t\int_I\frac{\rho_0(\partial_tX^n_x)^2}{\bar{\eta}_x^2}\,\diff x\diff s\\
&\quad=\frac{1}{2}\int_I \rho_0(\partial_tX^n)^2(x,0)\,\diff x+\int_0^t\int_I\partial_t\bigg({\frac{\rho_0^2}{\bar{\eta}_x^2}}\bigg)\partial_tX^n_x\,\diff x\diff s\\
&\qquad-\int_0^t\int_I\bigg[\partial_t\bigg(\frac{\rho_0X^n_x}{\bar{\eta}_x^2}\bigg)\partial_tX^n_x
-\frac{\rho_0(\partial_tX^n_x)^2}{\bar{\eta}_x^2}\bigg]\,\diff x\diff s.
\end{aligned}
\end{equation}
Then one uses Cauchy's inequality to obtain
\begin{equation}\label{existence-add-2}
\begin{aligned}
&\bigg|\int_0^t\int_I\partial_t\bigg({\frac{\rho_0^2}{\bar{\eta}_x^2}}\bigg)\partial_tX^n_x\,\diff x\diff s\bigg|\lesssim \int_0^t\int_I\rho_0^2|\bar{v}_x\partial_tX^n_x|\,\diff x\diff s\\
&\quad\leq \frac{1}{100}\int_0^t\int_I\rho_0(\partial_tX^n_x)^2\,\diff x\diff s
+C\int_0^t\int_I\rho_0\bar{v}_x^2\,\diff x\diff s\\
&\quad\leq \frac{1}{100}\int_0^t\int_I\rho_0(\partial_tX^n_x)^2\,\diff x\diff s+Ct P(\sup_{0\leq s\leq t}E^{1/2}(s,\bar{v})),
\end{aligned}
\end{equation}
and
\begin{equation}\label{existence-add-3}
\begin{aligned}
&\bigg|\int_0^t\int_I\bigg[\partial_t\bigg(\frac{\rho_0X^n_x}{\bar{\eta}_x^2}\bigg)\partial_tX^n_x
-\frac{\rho_0(\partial_tX^n_x)^2}{\bar{\eta}_x^2}\bigg]\,\diff x\diff s\bigg|\\
&\quad\lesssim \int_0^t\int_I\rho_0|\bar{v}_xX^n_x\partial_tX^n_x|\,\diff x\diff s\\
&\quad\leq \frac{1}{100}\int_0^t\int_I\rho_0(\partial_tX^n_x)^2\,\diff x\diff s
+C\sup_{0\leq s\leq t}\|\bar{v}_x\|_{L^\infty}^2\int_0^t\int_I\rho_0(X^n_x)^2\,\diff x\diff s\\
&\quad\leq \frac{1}{100}\int_0^t\int_I\rho_0(\partial_tX^n_x)^2\,\diff x\diff s+Ct P(\sup_{0\leq s\leq t}E^{1/2}(s,\bar{v})),
\end{aligned}
\end{equation}
where \eqref{first uniform estimates} has been used in the last inequality.

It follows from \eqref{existence-add-1}-\eqref{existence-add-3} that
\begin{equation}\label{existence-add-4}
\begin{aligned}
\int_I \rho_0(\partial_tX^n)^2\,\diff x
+\int_0^t\int_I\rho_0(\partial_tX^n_x)^2\,\diff x\diff s
\leq M_0+Ct P(\sup_{0\leq s\leq t}E^{1/2}(s,\bar{v})).
\end{aligned}
\end{equation}
By the lower semi-continuity of the norms, it follows from \eqref{existence-add-4} by taking limit \(n\to\infty\) that 
\begin{equation}\label{existence-add-5}
\begin{aligned}
\int_I \rho_0(\partial_tv)^2\,\diff x
+\int_0^t\int_I\rho_0(\partial_tv_x)^2\,\diff x\diff s
\leq M_0+Ct P(\sup_{0\leq s\leq t}E^{1/2}(s,\bar{v})).
\end{aligned}
\end{equation}

\bigskip
\noindent{\bf{Estimate of \(\|\sqrt{\rho_0}v_x\|_{L^2(I)}\).}} Next, we estimate \(\|\sqrt{\rho_0}v_x\|_{L^2(I)}\) using \eqref{existence-add-4}.
Multiplying Equation \(\eqref{existence-5}_1\) by \(\partial_t\lambda_i^n(t)\), and summing \(j=1,2,...,n\), one obtains
\begin{equation*}
		\big(\rho_0\partial_tX^n,\partial_tX^n\big)
		+\bigg(\frac{\rho_0X_x^n}{\bar{\eta}_x^2},\partial_tX_x^n\bigg)
		=\bigg(\frac{\rho_0^2}{\bar{\eta}_x^2},\partial_tX_x^n\bigg),
\end{equation*}
which yields 
\begin{equation}\label{existence-add-6}
\begin{aligned}
&\int_0^t\int_I \rho_0(\partial_tX^n)^2\,\diff x\diff s
+\frac{1}{2}\int_I\frac{\rho_0(X_x^n)^2}{\bar{\eta}_x^2}\,\diff x\\
&\quad=\frac{1}{2}\int_I\rho_0(X_x^n)^2(x,0)\,\diff x+\int_0^t\int_I{\frac{\rho_0^2\partial_tX_x^n}{\bar{\eta}_x^2}}\,\diff x\diff s
+\int_0^t\int_I\frac{\rho_0(X_x^n)^2\bar{v}_x}{\bar{\eta}_x^3}\,\diff x\diff s. 
\end{aligned}
\end{equation}
\eqref{first uniform estimates}, \eqref{existence-add-4} and Cauchy's inequality imply
\begin{equation}\label{existence-add-6.1}
	\begin{aligned}
		\bigg|\int_0^t\int_I\frac{\rho_0^2\partial_tX_x^n}{\bar{\eta}_x^2}\,\diff x\diff s\bigg|
		&\lesssim t+\int_0^t\int_I\rho_0(\partial_tX_x^n)^2\,\diff x\diff s\\
&\leq M_0+Ct P(\sup_{0\leq s\leq t}E^{1/2}(s,\bar{v})),
	\end{aligned}
\end{equation}
and
\begin{equation}\label{existence-add-6.2}
	\begin{aligned}
		\bigg|\int_0^t\int_I\frac{\rho_0(X_x^n)^2\bar{v}_x}{\bar{\eta}_x^3}\,\diff x\diff s\bigg|
		&\lesssim \int_0^t\int_I\rho_0\bar{v}_x^2\,\diff x\diff s+\int_0^t\int_I\rho_0(X_x^n)^2\,\diff x\diff s\\
&\leq M_0+Ct P(\sup_{0\leq s\leq t}E^{1/2}(s,\bar{v})).
	\end{aligned}
\end{equation}

It follows from  \eqref{existence-add-6}-\eqref{existence-add-6.2} that
\begin{equation}\label{existence-add-7}
\begin{aligned}
\int_0^t\int_I \rho_0(\partial_tX^n)^2\,\diff x\diff s
+\int_I\rho_0(X_x^n)^2\,\diff x
\leq M_0+Ct P(\sup_{0\leq s\leq t}E^{1/2}(s,\bar{v})).
\end{aligned}
\end{equation}
By the lower semi-continuity of the norms again, one gets from \eqref{existence-add-7} by taking limit \(n\to\infty\) that 
\begin{equation}\label{existence-add-8}
\begin{aligned}
\int_0^t\int_I \rho_0(\partial_tv)^2\,\diff x\diff s
+\int_I\rho_0v_x^2\,\diff x
\leq M_0+Ct P(\sup_{0\leq s\leq t}E^{1/2}(s,\bar{v})).
\end{aligned}
\end{equation}

\bigskip
\noindent{\bf{Estimate of \(\|\rho_0v_{xx}\|_{L^2(I)}\).}} Now, we estimate \(\|\rho_0v_{xx}\|_{L^2(I)}\) based on \eqref{existence-add-5} and \eqref{existence-add-8} by carrying out some elliptic estimates.
We start with the following equality:  
\begin{equation}\label{existence-add-9}
	\begin{aligned}
(\rho_0v_t, \phi)+\bigg(\frac{\rho_0v_x}{\bar{\eta}_x^2}, \phi_x\bigg)=\bigg({\frac{\rho_0^2}{\bar{\eta}_x^2}}, \phi_x\bigg)
	\end{aligned}
\end{equation}
for each \(\phi\in H^1(I)\) and a.e. \(0< t\leq T\), which follows from \eqref{existence-weak solution} and \eqref{existence-add-5}. Indeed, \eqref{existence-add-5} implies \(\rho_0^{1/2}v_t\in L^\infty([0,T],L^2(I))\), and thus \(\rho_0v_t\in L^\infty([0,T],L^2(I))\), which leads to 
\[\langle\rho_0v_t, \phi\rangle=(\rho_0v_t, \phi).\]
Since \(\rho_0\) satisfies the assumption \eqref{eq:intro-3},  one can obtain the interior \(H^2(I)\)-regularity \(v\in H^2_{\text{loc}}(I)\) from \eqref{existence-add-9} by a standard argument (see \cite{MR2597943}). Hence 
\begin{equation}\label{existence-add-10}
	\begin{aligned}
		\rho_0v_t+\bigg({\frac{\rho_0^2}{\bar{\eta}_x^2}}\bigg)_x
		=\bigg(\frac{\rho_0v_x}{\bar{\eta}_x^2}\bigg)_x &\quad \mbox{a.e.\ in}\ I\times (0,T].
	\end{aligned}
\end{equation}
Now one can repeat the argument in estimating \eqref{I-Priori-ellip-5} from Equation \eqref{existence-add-10} to obtain the boundary regularity  
\begin{equation}\label{existence-add-12}
\begin{aligned}
\|\rho_0v_{xx}\|_{L^2(I)}^2\leq M_0+Ct P(\sup_{0\leq s\leq t}E^{1/2}(s,\bar{v})). 
\end{aligned}
\end{equation}
Indeed, it is easy to check that the only key estimate in this assignment is \eqref{I-Priori-ellip-2.4}{\color{red}{,}} which should be replaced by
\begin{equation*}
\begin{aligned}
\|\rho_0(\bar{\eta}_x^{-2})_{x}v_x\|_{L^2}
\lesssim \|\rho_0^{1/2}v_x\|_{L^2}\|\bar{\eta}_{xx}\|_{L^\infty}
\leq Ct P(\sup_{0\leq s\leq t}E^{1/2}(s,\bar{v})),
\end{aligned}
\end{equation*}
where \eqref{existence-add-8} has been used.

\bigskip
In the following, making using \eqref{existence-add-10}, we can show that the remaining terms in \(E(t,v)\) have the desired bound \(M_0+Ct P(\sup_{0\leq s\leq t}E^{1/2}(s,\bar{v}))\).\\
\noindent{\bf{Estimate of \(\|\sqrt{\rho_0}\partial_t^2v\|_{L^2(I)}\).}} Applying \(\partial_t^2\) to \eqref{existence-add-10} and multiplying it by \(\partial_t^2v\) yield 
\begin{equation}\label{existence-add-13}
\begin{aligned}
&\frac{1}{2}\int_I \rho_0(\partial_t^2v)^2\,\diff x
+\int_0^t\int_I\frac{\rho_0(\partial_t^2v_x)^2}{\bar{\eta}_x^2}\,\diff x\diff s\\
&\quad=\frac{1}{2}\int_I \rho_0(\partial_t^2v)^2(x,0)\,\diff x+\int_0^t\int_I\partial_t^2\bigg({\frac{\rho_0^2}{\bar{\eta}_x^2}}\bigg)\partial_t^2v_x\,\diff x\diff s\\
&\qquad-\int_0^t\int_I\bigg[\partial_t^2\bigg(\frac{\rho_0v_x}{\bar{\eta}_x^2}\bigg)\partial_t^2v_x
-\frac{\rho_0(\partial_t^2v_x)^2}{\bar{\eta}_x^2}\bigg]\,\diff x\diff s.
\end{aligned}
\end{equation}
Note that
\begin{equation*}
\begin{aligned}
\bigg|\partial_t^2\bigg(\frac{1}{\bar{\eta}_x^2}\bigg)\bigg|\lesssim |\partial_t\bar{v}_x|+\bar{v}_x^2,
\end{aligned}
\end{equation*}
and
\begin{equation*}
\begin{aligned}
\bigg|\partial_t^2\bigg(\frac{v_x}{\bar{\eta}_x^2}\bigg)\partial_t^2v_x
-\frac{(\partial_t^2v_x)^2}{\bar{\eta}_x^2}\bigg|\lesssim (|v_x\partial_t\bar{v}_x|+|v_x\bar{v}_x^2|+|\bar{v}_x\partial_tv_x|)|\partial_t^2v_x|.
\end{aligned}
\end{equation*}
Then one may use Cauchy's inequality to obtain
\begin{equation}\label{existence-add-14}
\begin{aligned}
&\bigg|\int_0^t\int_I\partial_t^2\bigg({\frac{\rho_0^2}{\bar{\eta}_x^2}}\bigg)\partial_t^2v_x\,\diff x\diff s\bigg|\\
&\quad\leq \frac{1}{100}\int_0^t\int_I\rho_0(\partial_t^2v_x)^2\,\diff x\diff s+C\int_0^t\int_I\big[\rho_0(\partial_t\bar{v}_x)^2+\rho_0\bar{v}_x^2\big]\,\diff x\diff s\\
&\quad\leq \frac{1}{100}\int_0^t\int_I\rho_0(\partial_t^2v_x)^2\,\diff x\diff s+Ct P(\sup_{0\leq s\leq t}E^{1/2}(s,\bar{v})),
\end{aligned}
\end{equation}
and
\begin{equation}\label{existence-add-15}
\begin{aligned}
&\bigg|\int_0^t\int_I\bigg[\partial_t^2(\frac{\rho_0v_x}{\bar{\eta}_x^2})\partial_t^2v_x
-{\frac{\rho_0(\partial_t^2v_x)^2}{\bar{\eta}_x^2}}\bigg]\,\diff x\diff s\bigg|\\
&\quad\leq \frac{1}{100}\int_0^t\int_I\rho_0(\partial_t^2v_x)^2\,\diff x\diff s
+C\int_0^t\|\partial_t\bar{v}_x\|_{L^\infty}^2\int_I\rho_0v_x^2\,\diff x\diff s\\
&\quad\quad+C\int_0^t\|\bar{v}_x\|_{L^\infty}^4\int_I\rho_0v_x^2\,\diff x\diff s
+C\sup_{0\leq s\leq t}\|\bar{v}_x\|_{L^\infty}^2\int_0^t\int_I\rho_0(\partial_tv_x)^2\,\diff x\diff s\\
&\quad\leq \frac{1}{100}\int_0^t\int_I\rho_0(\partial_t^2v_x)^2\,\diff x\diff s+Ct P(\sup_{0\leq s\leq t}E^{1/2}(s,\bar{v})),
\end{aligned}
\end{equation}
where \eqref{existence-add-5} and \eqref{existence-add-8} have been used. 

It follows from \eqref{existence-add-13}-\eqref{existence-add-15} that
\begin{equation}\label{existence-add-16}
\begin{aligned}
\int_I \rho_0(\partial_t^2v)^2\,\diff x
+\int_0^t\int_I\rho_0(\partial_t^2v_x)^2\,\diff x\diff s
\leq M_0+Ct P(\sup_{0\leq s\leq t}E^{1/2}(s,\bar{v})).
\end{aligned}
\end{equation}

\bigskip
\noindent{\bf{Estimate of \(\|\sqrt{\rho_0}\partial_tv_x\|_{L^2(I)}\).}} Applying \(\partial_t\) to \eqref{existence-add-10}, 
and multiplying it by \(\partial_t^2v\), one obtains by some direct calculations that
\begin{equation}\label{existence-add-17}
\begin{aligned}
&\int_0^t\int_I \rho_0(\partial_t^2v)^2\,\diff x\diff s
+\frac{1}{2}\int_I\frac{\rho_0(\partial_tv_x)^2}{\bar{\eta}_x^2}\,\diff x\\
&=\frac{1}{2}\int_I\rho_0(\partial_tv_x)^2(x,0)\,\diff x+\int_0^t\int_I\partial_t\bigg({\frac{\rho_0^2}{\bar{\eta}_x^2}}\bigg)\partial_t^2v_x\,\diff x\diff s\\
&\quad+\frac{1}{2}\int_0^t\int_I\partial_t\bigg({\frac{\rho_0}{\bar{\eta}_x^2}}\bigg)(\partial_tv_x)^2\,\diff x\diff s
-\int_0^t\int_I\partial_t\bigg({\frac{\rho_0}{\bar{\eta}_x^2}}\bigg)v_x\partial_t^2v_x\,\diff x\diff s.
\end{aligned}
\end{equation}
The last three terms on the RHS of \eqref{existence-add-17} can be estimated as follows:
\begin{equation}\label{existence-add-18}
\begin{aligned}
\bigg|\int_0^t\int_I\bigg({\frac{\rho_0^2}{\bar{\eta}_x^2}}\bigg)\partial_t^2v_x\,\diff x\diff s\bigg|
&\lesssim \int_0^t\int_I\rho_0(\partial_t^2v_x)^2\,\diff x\diff s+\int_0^t\int_I\rho_0\bar{v}_x^2\,\diff x\diff s\\
&\leq M_0+Ct P(\sup_{0\leq s\leq t}E^{1/2}(s,\bar{v})),
\end{aligned}
\end{equation}
\begin{equation}\label{existence-add-19}
\begin{aligned}
\bigg|\int_0^t\int_I\partial_t\bigg({\frac{\rho_0}{\bar{\eta}_x^2}}\bigg)(\partial_tv_x)^2\,\diff x\diff s\bigg|
&\lesssim \sup_{0\leq s\leq t}\|\bar{v}_x\|_{L^\infty}\int_0^t\int_I\rho_0(\partial_tv_x)^2\,\diff x\diff s\\
&\leq M_0+Ct P(\sup_{0\leq s\leq t}E^{1/2}(s,\bar{v})),
\end{aligned}
\end{equation}
and
\begin{equation}\label{existence-add-20}
\begin{aligned}
\bigg|\int_0^t\int_I\partial_t\bigg({\frac{\rho_0}{\bar{\eta}_x^2}}\bigg)v_x\partial_t^2v_x\,\diff x\diff s\bigg|&\lesssim \int_0^t\int_I\rho_0(\partial_t^2v_x)^2\,\diff x\diff s\\
&\quad+\int_0^t\|\bar{v}_x\|_{L^\infty}^2\int_I\rho_0v_x^2\,\diff x\diff s\\
&\leq M_0+Ct P(\sup_{0\leq s\leq t}E^{1/2}(s,\bar{v})),
\end{aligned}
\end{equation}
where one has used \eqref{existence-add-16} in \eqref{existence-add-18} and \eqref{existence-add-20}. 

It then follows from \eqref{existence-add-17}-\eqref{existence-add-20} that
\begin{equation}\label{existence-add-21}
\begin{aligned}
\int_0^t\int_I \rho_0(\partial_t^2v)^2\,\diff x\diff s
+\int_I\rho_0(\partial_tv_x)^2\,\diff x
\leq M_0+Ct P(\sup_{0\leq s\leq t}E^{1/2}(s,\bar{v})).
\end{aligned}
\end{equation}

\bigskip
\noindent{\bf{Estimate of \(\|\rho_0^{3/2}\partial_x^3v\|_{L^2(I)}, \|\rho_0\partial_tv_{xx}\|_{L^2},\|\rho_0^2\partial_x^4v\|_{L^2}\).}} Now, one can estimate \(\|\rho_0^{3/2}\partial_x^3v\|_{L^2(I)}\) by using \eqref{existence-add-21}.  Indeed, one just needs to replace \eqref{I-Priori-ellip-12} by
\begin{equation*}
\begin{aligned}
\|\rho_0(\bar{\eta}_x^{-2})_{x}v_{xx}\|_{L^2}
\lesssim \|\rho_0v_{xx}\|_{L^2}\|\bar{\eta}_{xx}\|_{L^\infty}
\leq Ct P(\sup_{0\leq s\leq t}E^{1/2}(s,\bar{v})),
\end{aligned}
\end{equation*}
due to \eqref{existence-add-12}, and replace \eqref{I-Priori-ellip-14} by 
\begin{equation*}
\begin{aligned}
\|(\rho_0\bar{\eta}_x^{-2})_{xx}v_x\|_{L^2}
&\lesssim \|v_x\|_{L^2}(1+\|\bar{\eta}_{xx}\|_{L^\infty}
+\|\bar{\eta}_{xx}\|_{L^\infty}^2
+\|\rho_0\partial_x^3\bar{\eta}\|_{L^\infty})\\
&\lesssim (\|\rho_0v_x\|_{L^2}+\|\rho_0v_{xx}\|_{L^2})\\
&\quad\times(1+\|\bar{\eta}_{xx}\|_{L^\infty}
+\|\bar{\eta}_{xx}\|_{L^\infty}^2
+\|\rho_0\partial_x^3\bar{\eta}\|_{L^\infty})\\
&\leq [M_0+Ct P(\sup_{0\leq s\leq t}E^{1/2}(s,\bar{v}))]^{1/2},
\end{aligned}
\end{equation*}
due to \eqref{existence-add-8} and \eqref{existence-add-12}, 
and then repeat the argument for \eqref{I-Priori-ellip-third} to get 
\begin{equation}\label{existence-add-23}
\begin{aligned}
\|\rho_0^{3/2}\partial_x^3v\|_{L^2(I)}^2\leq M_0+Ct P(\sup_{0\leq s\leq t}E^{1/2}(s,\bar{v})). 
\end{aligned}
\end{equation}

Similarly, one may repeat the arguments for \eqref{I-Priori-ellip-22} and \eqref{I-Priori-ellip-31} to obtain   
\begin{equation}\label{existence-add-24}
\begin{aligned}
\|\rho_0\partial_tv_{xx}\|_{L^2}^2+\|\rho_0^2\partial_x^4v\|_{L^2}^2
\leq M_0+Ct P(\sup_{0\leq s\leq t}E^{1/2}(s,\bar{v})). 
\end{aligned}
\end{equation}

\bigskip
\noindent{\bf{Estimate of \(\|\sqrt{\rho_0}\partial_t^3v\|_{L^2(I)}\).}} Next,  \(\|\sqrt{\rho_0}\partial_t^3v\|_{L^2(I)}\) can be estimated due to \eqref{existence-add-23} and \eqref{existence-add-24}. Indeed, one can apply \(\partial_t^3\) to \eqref{existence-add-10} and multiply it by \(\partial_t^3v\), after some elementary computations, to obtain that
\begin{equation}\label{existence-add-26}
\begin{aligned}
&\frac{1}{2}\int_I \rho_0(\partial_t^3v)^2\,\diff x
+\int_0^t\int_I\frac{\rho_0(\partial_t^3v_x)^2}{\bar{\eta}_x^2}\,\diff x\diff s\\
&\quad=\frac{1}{2}\int_I \rho_0(\partial_t^3v)^2(x,0)\,\diff x+\int_0^t\int_I\partial_t^3\bigg({\frac{\rho_0^2}{\bar{\eta}_x^2}}\bigg)\partial_t^3v_x\,\diff x\diff s\\
&\qquad-\int_0^t\int_I\bigg[\partial_t^3\bigg(\frac{\rho_0v_x}{\bar{\eta}_x^2}\bigg)\partial_t^3v_x
-\frac{\rho_0(\partial_t^3v_x)^2}{\bar{\eta}_x^2}\bigg]\,\diff x\diff s.
\end{aligned}
\end{equation}
Similar to \eqref{Add-1} and \eqref{Add-2},  one gets from \eqref{eta-bound-2} that
\begin{equation*}
\begin{aligned}
\bigg|\partial_t^3\bigg(\frac{1}{\bar{\eta}_x^2}\bigg)\bigg|\lesssim |\partial_t^2\bar{v}_x|+|\bar{v}_x\partial_t\bar{v}_x|+|\bar{v}_x^3|,
\end{aligned}
\end{equation*}
and
\begin{equation*}
\begin{aligned}
\bigg|\partial_t^3\bigg(\frac{v_x}{\bar{\eta}_x^2}\bigg)\partial_t^3v_x
-\frac{(\partial_t^3v_x)^2}{\bar{\eta}_x^2}\bigg|
&\lesssim \big[|\bar{v}_x\partial_t^2v_x|+|\partial_tv_x|(|\bar{v}_x|^2+|\partial_t\bar{v}_x|)+|v_x||\bar{v}_x|^3\\
&\quad+|v_x\partial_t^2\bar{v}_x|+|\partial_t\bar{v}_x||\bar{v}_xv_x|\big]|\partial_t^3v_x|.
\end{aligned}
\end{equation*}
Then one uses Cauchy's inequality to obtain
\begin{equation}\label{existence-add-27}
\begin{aligned}
&\bigg|\int_0^t\int_I\partial_t^3\bigg({\frac{\rho_0^2}{\bar{\eta}_x^2}}\bigg)\partial_t^3v_x\,\diff x\diff s\bigg|\\
&\quad\leq \frac{1}{100}\int_0^t\int_I\rho_0(\partial_t^3v_x)^2\,\diff x\diff s
+C\int_0^t\int_I\rho_0(\partial_t^2\bar{v}_x)^2\,\diff x\diff s\\
&\qquad+C\int_0^t\|\bar{v}_x\|_{L^\infty}^2\int_I\rho_0(\partial_t\bar{v}_x)^2\,\diff x\diff s
+C\int_0^t\|\bar{v}_x\|_{L^\infty}^4\int_I\rho_0\bar{v}_x^2\,\diff x\diff s\\
&\quad\leq \frac{1}{100}\int_0^t\int_I\rho_0(\partial_t^3v_x)^2\,\diff x\diff s+Ct P(\sup_{0\leq s\leq t}E^{1/2}(s,\bar{v})),
\end{aligned}
\end{equation}
and
\begin{equation}\label{existence-add-28}
\begin{aligned}
&\bigg|\int_0^t\int_I\bigg[\partial_t^3\bigg(\frac{\rho_0v_x}{\bar{\eta}_x^2}\bigg)\partial_t^3v_x
-\frac{\rho_0(\partial_t^3v_x)^2}{\bar{\eta}_x^2}\bigg]\,\diff x\diff s\bigg|\\
&\quad\leq \frac{1}{100}\int_0^t\int_I\rho_0(\partial_t^3v_x)^2\,\diff x\diff s
+C\sup_{0\leq s\leq t}\|\bar{v}_x\|_{L^\infty}^2\int_0^t\int_I\rho_0(\partial_t^2v_x)^2\,\diff x\diff s\\
&\qquad+C\int_0^t\|\bar{v}_x\|_{L^\infty}^4\int_I\rho_0(\partial_tv_x)^2\,\diff x\diff s+C\int_0^t\|\partial_t\bar{v}_x\|_{L^\infty}^2\int_I\rho_0(\partial_tv_x)^2\,\diff x\diff s\\
&\qquad+C\int_0^t\|\bar{v}_x\|_{L^\infty}^6\int_I\rho_0v_x^2\,\diff x\diff s+C\int_0^t\|v_x\|_{L^\infty}^2\int_I\rho_0(\partial_t^2\bar{v}_x)^2\,\diff x\diff s\\
&\qquad+C\int_0^t\|\bar{v}_x\|_{L^\infty}^2\|v_x\|_{L^\infty}^2\int_I\rho_0(\partial_t\bar{v}_x)^2\,\diff x\diff s\\
&\quad\leq \frac{1}{100}\int_0^t\int_I\rho_0(\partial_t^3v_x)^2\,\diff x\diff s+Ct P(\sup_{0\leq s\leq t}E^{1/2}(s,\bar{v})),
\end{aligned}
\end{equation}
where one has used \eqref{existence-add-8}, \eqref{existence-add-12}, \eqref{existence-add-23} and \eqref{existence-add-24} to estimate 
\begin{equation*}
\begin{aligned}
\|v_x\|_{L^\infty}&\lesssim \|v_x\|_{L^2}+\|v_{xx}\|_{L^2}\\
&\lesssim (\|\rho_0v_x\|_{L^2}+\|\rho_0v_{xx}\|_{L^2})+(\|\rho_0v_{xx}\|_{L^2}+\|\rho_0\partial_x^3v\|_{L^2})\\
&\lesssim \|\rho_0v_x\|_{L^2}+\|\rho_0v_{xx}\|_{L^2}+\|\rho_0^2\partial_x^3v\|_{L^2}+\|\rho_0^2\partial_x^4v\|_{L^2}\\
&\leq M_0+Ct P(\sup_{0\leq s\leq t}E^{1/2}(s,\bar{v})). 
\end{aligned}
\end{equation*}

It follows from \eqref{existence-add-26}-\eqref{existence-add-28} that
\begin{equation}\label{existence-add-29}
\begin{aligned}
\int_I \rho_0(\partial_t^3v)^2\,\diff x
+\int_0^t\int_I\rho_0(\partial_t^3v_x)^2\,\diff x\diff s
\leq M_0+Ct P(\sup_{0\leq s\leq t}E^{1/2}(s,\bar{v})).
\end{aligned}
\end{equation}

\bigskip
\noindent{\bf{Estimate of \(\|\sqrt{\rho_0}\partial_t^2v_x\|_{L^2(I)}\).}}  In view of \eqref{existence-add-29}, similar to \eqref{II-Priori-time-10}, one can derive that
\begin{equation}\label{existence-add-30}
\begin{aligned}
\int_0^t\int_I \rho_0(\partial_t^3v)^2\,\diff x\diff s
+\int_I\rho_0(\partial_t^2v_x)^2\,\diff x
\leq M_0+Ct P(\sup_{0\leq s\leq t}E^{1/2}(s,\bar{v})).
\end{aligned}
\end{equation}

\bigskip
\noindent{\bf{Estimate of \(\|\rho_0^{3/2}\partial_t\partial_x^3v\|_{L^2},\|\rho_0^{5/2}\partial_x^5v\|_{L^2}, \|\rho_0^2\partial_t\partial_x^4v\|_{L^2(I)}, \|\rho_0^3\partial_x^6v\|_{L^2(I)}\).}}
By \eqref{existence-add-29} and \eqref{existence-add-30}, one may repeat the arguments for \eqref{II-Priori-ellip-4}, \eqref{II-Priori-ellip-11}, \eqref{II-Priori-ellip-18} and \eqref{II-Priori-ellip-25} to obtain   
\begin{equation}\label{existence-add-31}
\begin{aligned}
&\|\rho_0^{3/2}\partial_t\partial_x^3v\|_{L^2}^2+\|\rho_0^{5/2}\partial_x^5v\|_{L^2}^2+\|\rho_0^2\partial_t\partial_x^4v\|_{L^2(I)}^2+\|\rho_0^3\partial_x^6v\|_{L^2(I)}^2\\
&\quad\leq M_0+Ct P(\sup_{0\leq s\leq t}E^{1/2}(s,\bar{v})). 
\end{aligned}
\end{equation}

Finally, \eqref{existence-add-0} follows from \eqref{solution bound}, \eqref{existence-add-5}, \eqref{existence-add-8}, \eqref{existence-add-12}, \eqref{existence-add-16}, \eqref{existence-add-21}, \eqref{existence-add-23}, \eqref{existence-add-24}, \eqref{existence-add-29}, \eqref{existence-add-30} and \eqref{existence-add-31}.

\end{proof}

\subsection{Existence of a classical solution to the problem \eqref{existence-3}.}\label{classical solution}
In order to show that
there exists a classical solution to the problem \eqref{existence-3}, we will construct its approximate solutions and show the approximate solutions converge uniformly by a contraction mapping method. Therefore we consider the following iteration problem:
\begin{equation}\label{existence-21}
	\begin{cases}
		\rho_0v_t^{(n)}+\bigg[\frac{\rho_0^2}{(\eta_x^{(n-1)})^2}\bigg]_x
		=\bigg[\frac{\rho_0v_x^{(n)}}{(\eta_x^{(n-1)})^2}\bigg]_x &\quad \mbox{in}\ I\times (0,T],\\
		v^{(n)}=u_0 &\quad \mbox{on}\ I\times \{t=0\},\\
		v_x^{(n)}=0 &\quad \mbox{on}\ \Gamma\times (0,T].
	\end{cases}
\end{equation}
For \(n=1\), we impose 
\(\eta^{(0)}(t,x)=x+tu_0(x)\). We then solve the problem \eqref{existence-21} for \(n=1,2,...\) iteratively. 
Given \(T>0\) sufficiently small, 
in view of Lemma \ref{Regularity lemma}, one can obtain \(\{v^{(n)}\}_{n=1}^\infty\subset \mathcal{C}_T(M_1)\) for any \(n\geq 1\). 

In the following, we will show that the approximate solutions \(\{v^{(n)}\}_{n=1}^\infty\) 
are contractive in some appropriate energy space. To this end, 
setting \(\sigma(v^{(n)}):=v^{(n+1)}-v^{(n)}\), one deduces
\begin{equation}\label{existence-22}
	\begin{cases}
		\rho_0\partial_t\sigma(v^{(n)})+\bigg[\frac{\rho_0^2}{(\eta_x^{(n)})^2}\bigg]_x-\bigg[\frac{\rho_0^2}{(\eta_x^{(n-1)})^2}\bigg]_x\\
		\qquad\qquad\qquad=\bigg[\frac{\rho_0v_x^{(n+1)}}{(\eta_x^{(n)})^2}\bigg]_x
		-\bigg[\frac{\rho_0v_x^{(n)}}{(\eta_x^{(n-1)})^2}\bigg]_x &\quad \mbox{in}\ I\times (0,T],\\
	\sigma(v^{(n)})=0 &\quad \mbox{on}\ I\times \{t=0\},\\
	\sigma_x(v^{(n)})=0 &\quad \mbox{on}\ \Gamma\times (0,T].
	\end{cases}
\end{equation}

\begin{lemma} It holds that
	\begin{equation}\label{existence-23}
		\begin{aligned}
			&\frac{\diff}{\diff t}\int_I\rho_0[\sigma(v^{(n)})]^2\,\diff x
			+\int_I\rho_0[\sigma_x(v^{(n)})]^2\,\diff x\\
			&\quad\leq C(M_1^{1/2}+1)t\int_0^t\int_I\rho_0[\sigma_x(v^{(n-1)})]^2\,\diff x\diff s.
		\end{aligned}
	\end{equation}		
\end{lemma}

\begin{proof} Multiplying Equation \(\eqref{existence-22}_1\) by \(\sigma(v^{(n)})\) and integrating by parts with respect to \(x\) yield
	\begin{equation}\label{existence-24}
		\begin{aligned}
			&\frac{1}{2}\frac{\diff}{\diff t}\int_I\rho_0[\sigma(v^{(n)})]^2\,\diff x
			+\int_I\bigg[\frac{\rho_0v_x^{(n+1)}}{(\eta_x^{(n)})^2}
			-\frac{\rho_0v_x^{(n)}}{(\eta_x^{(n-1)})^2}\bigg]\sigma_x(v^{(n)})\,\diff x\\
			&\quad=\int_I\bigg[\frac{\rho_0^2}{(\eta_x^{(n)})^2}-\frac{\rho_0^2}{(\eta_x^{(n-1)})^2}\bigg]\sigma_x(v^{(n)})\,\diff x.		 
		\end{aligned}
	\end{equation}	
Note that
	\begin{equation}\label{existence-25}
		\begin{aligned}
			&\int_I\bigg[\frac{\rho_0v_x^{(n+1)}}{(\eta_x^{(n)})^2}
			-\frac{\rho_0v_x^{(n)}}{(\eta_x^{(n-1)})^2}\bigg]\sigma_x(v^{(n)})\,\diff x\\
			&\quad=\int_I\bigg[\frac{\rho_0v_x^{(n+1)}}{(\eta_x^{(n)})^2}
			-\frac{\rho_0v_x^{(n)}}{(\eta_x^{(n)})^2}\bigg]\sigma_x(v^{(n)})\,\diff x\\
			&\qquad+\int_I\bigg[\frac{\rho_0v_x^{(n)}}{(\eta_x^{(n)})^2}
			-\frac{\rho_0v_x^{(n)}}{(\eta_x^{(n-1)})^2}\bigg]\sigma_x(v^{(n)})\,\diff x\\ 
			&\quad=\colon I_1+I_2.
		\end{aligned}
	\end{equation}	
Direct estimates yield
\begin{equation}\label{existence-26}
	\begin{aligned}
	I_1=\int_I\frac{\rho_0}{(\eta_x^{(n)})^2}
	[\sigma_x(v^{(n)})]^2\,\diff x\geq \frac{4}{9}\int_I\rho_0[\sigma_x(v^{(n)})]^2\,\diff x,
	\end{aligned}
\end{equation}	
and
	\begin{equation}\label{existence-27}
	\begin{aligned}
		|I_2|&=\bigg|\int_I\bigg[\frac{\rho_0v_x^{(n)}}{(\eta_x^{(n)})^2(\eta_x^{(n-1)})^2}(\eta_x^{(n)}+\eta_x^{(n-1)})
		\int_0^t\sigma_x(v^{(n-1)})\,\diff s\bigg]\sigma_x(v^{(n)})\,\diff x\bigg|\\
		&\leq \frac{1}{100}\int_I\rho_0[\sigma_x(v^{(n)})]^2\,\diff x+Ct\|v_x^{(n)}\|_{L^\infty}\int_0^t\int_I\rho_0[\sigma_x(v^{(n-1)})]^2\,\diff x\diff s.
	\end{aligned}
\end{equation}	
Hence it follows from \eqref{existence-25}-\eqref{existence-27} that
	\begin{equation}\label{existence-28}
	\begin{aligned}
		&-\int_I\bigg[\bigg(\frac{\rho_0v_x^{(n+1)}}{(\eta_x^{(n)})^2}\bigg)_x
		-\bigg(\frac{\rho_0v_x^{(n)}}{(\eta_x^{(n-1)})^2}\bigg)_x\bigg]\sigma(v^{(n)})\,\diff x\\
		&\quad\geq \frac{1}{3}\int_I\rho_0[\sigma_x(v^{(n)})]^2\,\diff x-Ct\|v_x^{(n)}\|_{L^\infty}\int_0^t\int_I\rho_0[\sigma_x(v^{(n-1)})]^2\,\diff x\diff s.
	\end{aligned}
\end{equation}		
Similar to \eqref{existence-27}, one has
	\begin{equation}\label{existence-29}
		\begin{aligned}
			&\int_I\bigg[\frac{\rho_0^2}{(\eta_x^{(n)})^2}-\frac{\rho_0^2}{(\eta_x^{(n-1)})^2}\bigg]\sigma_x(v^{(n)})\,\diff x\\
			&\quad=\int_I\bigg[\frac{\rho_0^2}{(\eta_x^{(n)})^2(\eta_x^{(n-1)})^2}(\eta_x^{(n)}+\eta_x^{(n-1)})
			\int_0^t\sigma_x(v^{(n-1)})\,\diff s\bigg]\sigma_x(v^{(n)})\,\diff x\\
			&\quad\leq\frac{1}{100}\int_I\rho_0[\sigma_x(v^{(n)})]^2\,\diff x+Ct\int_0^t\int_I\rho_0[\sigma_x(v^{(n-1)})]^2\,\diff x\diff s. 
		\end{aligned}
	\end{equation}	

Substituting \eqref{existence-28} and \eqref{existence-29} into \eqref{existence-24} gives 
	\begin{equation*}
		\begin{aligned}
			&\frac{\diff}{\diff t}\int_I\rho_0[\sigma(v^{(n)})]^2\,\diff x
			+\int_I\rho_0[\sigma_x(v^{(n)})]^2\,\diff x\\
			&\quad\leq C(\|v_x^{(n)}\|_{L^\infty(I)}+1)t\int_0^t\int_I\rho_0[\sigma_x(v^{(n-1)})]^2\,\diff x\diff s
			\\
			&\quad\leq C(M_1^{1/2}+1)t\int_0^t\int_I\rho_0[\sigma_x(v^{(n-1)})]^2\,\diff x\diff s,
		\end{aligned}
	\end{equation*}	
where one has used \(\{v^{(n)}\}_{n=1}^\infty\subset \mathcal{C}_T(M_1)\) in the last line. Hence \eqref{existence-23} follows.
\end{proof}

Integrating \eqref{existence-23} with respect to \(t\) on \([0,T]\), we deduce 
	\begin{equation*}
		\begin{aligned}
			&\underset{0\leq t\leq T}{\sup}\|\rho_0^{1/2}\sigma(v^{(n)})\|_{L^2(I)}^2+\|\rho_0^{1/2}\sigma_x(v^{(n)})\|_{L^2([0,T];L^2(I))}^2\\
			&\quad\leq \|\rho_0^{1/2}\sigma(v^{(n)})(0)\|_{L^2(I)}^2+C(M_1^{1/2}+1)T\|\rho_0^{1/2}\sigma_x(v^{(n-1)})\|_{L^2([0,T];L^2(I))}^2\\
			&\quad=C(M_1^{1/2}+1)T\|\rho_0^{1/2}\sigma_x(v^{(n-1)})\|_{L^2([0,T];L^2(I))}^2\\
			&\quad\leq \frac{1}{4}\big(\underset{0\leq t\leq T}{\sup}\|\rho_0^{1/2}\sigma(v^{(n-1)})\|_{L^2(I)}^2+
			\|\rho_0^{1/2}\sigma_x(v^{(n-1)})\|_{L^2([0,T];L^2(I))}^2\big),
		\end{aligned}
	\end{equation*}
since \(T>0\) is sufficiently small.
Hence for any \(n\geq 1\)
\begin{equation}\label{existence-30}
	\begin{aligned}
		&\underset{0\leq t\leq T}{\sup}\|\rho_0^{1/2}\sigma(v^{(n)})\|_{L^2(I)}+\|\rho_0^{1/2}\sigma_x(v^{(n)})\|_{L^2([0,T];L^2(I))}\\
		&\quad\leq \frac{1}{2}\big(\underset{0\leq t\leq T}{\sup}\|\rho_0^{1/2}\sigma(v^{(n-1)})\|_{L^2(I)}+
		\|\rho_0^{1/2}\sigma_x(v^{(n-1)})\|_{L^2([0,T];L^2(I))}\big).
	\end{aligned}
\end{equation}

The estimates \eqref{existence-30} imply that \(\{v^{(n)}\}_{n=1}^\infty\) is a Cauchy sequence in the space \(L^2([0,T],L^2(I))\) by using the weighted Sobolev inequality \eqref{ineq:weighted Sobolev}. According to this fact and the a priori bound \eqref{I-Priori-ellip-33} (see \eqref{Preliminary-1} that this a priori bound \eqref{I-Priori-ellip-33} controls \(H^3(I)\)-bound of \(v\)), one may furthermore deduce that \(\{v^{(n)}\}_{n=1}^\infty\) is a Cauchy sequence in the space \(L^2([0,T],H^s(I))\) \((0<s<3)\) by using the standard Gagliardo-Nirenberg interpolation inequality for functions in spatial variables (see \cite{MR3813967}). 
However this is insufficient for us to pass limit in \(n\) in Equation \(\eqref{existence-3}_1\) for time pointwisely. To get around this difficulty, we need the following weighted interpolation inequality:

\begin{lemma}[Weighted Interpolation Inequality]\label{le:Preliminary-5} The following weighted interpolation holds
	\begin{align}\label{Weighted Interpolation}
		\|g\|_{L^2(I)}\lesssim \|g\|_{L_{\rho_0}^2(I)}^{1/2}\|g\|_{H_{\rho_0}^1(I)}^{1/2},
	\end{align}
	where
	\begin{equation*}
		\begin{aligned}
			\|g\|_{L_{\rho_0}^2(I)}^2=\int_I\rho_0g^2\,\diff x\quad {\rm{and}}\quad 	\|g\|_{H_{\rho_0}^1(I)}^2=\int_I\rho_0(g^2+g_x^2)\,\diff x.
		\end{aligned}
	\end{equation*}	
\end{lemma}

\begin{proof} Due to the assumption \eqref{eq:intro-3} on \(\rho_0\), it suffices to prove \eqref{Weighted Interpolation}  for \(\rho_0\) with \(\rho_0(x)=x\) on \([0,1/2]\) and \(1-x\) on \([1/2,1]\). Note that
	\begin{equation*}
		\begin{aligned}
			\int_Ig^2\,\diff x= \int_0^{1/2}g^2\,\diff x+\int_{1/2}^1g^2\,\diff x.
		\end{aligned}
	\end{equation*}	
	Integration by parts yields 
	\begin{equation}\label{existence-31}
		\begin{aligned}
			\int_0^{1/2}g^2\,\diff x&=xg^2(x)\big|_{x=0}^{x=1/2}-2\int_0^{1/2}xgg_x\,\diff x\\
			&=\frac{1}{2}g^2(\frac{1}{2})-2\int_0^{1/2}\rho_0gg_x\,\diff x.
		\end{aligned}
	\end{equation}
To estimate \(g(\frac{1}{2})\), one has
	\begin{equation}\label{existence-32}
		\begin{aligned}
			&\int_0^{1/2}\rho_0g^2\,\diff x=\int_0^{1/2}xg^2\,\diff x\\
			&\quad=\frac{1}{2}x^2g^2(x)\big|_{x=0}^{x=1/2}-\int_0^{1/2}x^2gg_x\,\diff x\\
			&\quad=\frac{1}{8}g^2(\frac{1}{2})-\int_0^{1/2}\rho_0^2gg_x\,\diff x.
		\end{aligned}
	\end{equation}
It follows from \eqref{existence-31} and \eqref{existence-32} that
	\begin{equation}\label{existence-33}
		\begin{aligned}
			\int_0^{1/2}g^2\,\diff x
			&=4\int_0^{1/2}\rho_0g^2\,\diff x+4\int_0^{1/2}\rho_0^2gg_x\,\diff x-2\int_0^{1/2}\rho_0gg_x\,\diff x\\
			&\lesssim \int_0^{1/2}\rho_0g^2\,\diff x+\bigg(\int_0^{1/2}\rho_0^2g^2\,\diff x\bigg)^{1/2}\bigg(\int_0^{1/2}\rho_0^2g_x^2\,\diff x\bigg)^{1/2}\\
			&\quad+\bigg(\int_0^{1/2}\rho_0g^2\,\diff x\bigg)^{1/2}\bigg(\int_0^{1/2}\rho_0g_x^2\,\diff x\bigg)^{1/2}\\
			&\lesssim \bigg(\int_I\rho_0g^2\,\diff x\bigg)^{1/2}\bigg(\int_I\rho_0(g^2+g_x^2)\,\diff x\bigg)^{1/2}.
		\end{aligned}
	\end{equation}
Similarly, one can obtain
	\begin{equation}\label{existence-34}
		\begin{aligned}
			\int_{1/2}^1g^2\,\diff x
			\lesssim \bigg(\int_I\rho_0g^2\,\diff x\bigg)^{1/2}\bigg(\int_I\rho_0(g^2+g_x^2)\,\diff x\bigg)^{1/2}.
		\end{aligned}
	\end{equation}
	Finally, \eqref{Weighted Interpolation} follows from \eqref{existence-33} and \eqref{existence-34}.
	
\end{proof}

 Taking \(g(\cdot)=\sigma(v^{(n)})(\cdot,t)\) in \eqref{Weighted Interpolation}, one has that for each \(t\in[0,T]\) 
	\begin{equation}\label{existence-35}
		\begin{aligned}
	\|\sigma(v^{(n)})(\cdot,t)\|_{L^2(I)}\lesssim \|\sigma(v^{(n)})(\cdot,t)\|_{L_{\rho_0}^2(I)}^{1/2}\|\sigma(v^{(n)})(\cdot,t)\|_{H_{\rho_0}^1(I)}^{1/2}.
	\end{aligned}
\end{equation}
It follows from \eqref{existence-30} that \(\|\sigma(v^{(n)})(\cdot,t)\|_{L_{\rho_0}^2(I)}\to 0\) as \(n\to \infty\).
And \eqref{I-Priori-ellip-33} implies that \(\|\sigma(v^{(n)})(\cdot,t)\|_{H_{\rho_0}^1(I)}\) is uniformly bounded in \(n\geq 1\).  Hence \eqref{existence-35} implies that as \(n\to \infty\)
\begin{align}\label{existence-36}
		v^{(n)}\rightarrow v\quad \text{in}\ C([0,T];L^2(I)).
\end{align}	
Then the standard Gagliardo-Nirenberg interpolation inequality on a bounded domain (see \cite{MR3813967}) shows  that for any \(s\in(0,3)\)
\begin{equation}\label{existence-37}
	\begin{aligned}
		\|\sigma(v^{(n)})(\cdot,t)\|_{H^s(I)}\lesssim \|\sigma(v^{(n)})(\cdot,t)\|_{L^2(I)}^{1-\frac{s}{3}}
		\|\sigma(v^{(n)})(\cdot,t)\|_{H^3(I)}^{\frac{s}{3}}.
	\end{aligned}
\end{equation}
Since \(\|\sigma(v^{(n)})(\cdot,t)\|_{H^3(I)}\) is uniformly bounded in \(n\geq 1\),  it follows from \eqref{existence-36} and \eqref{existence-37} that as \(n\to \infty\)
\begin{align*}
	v^{(n)}\rightarrow v\quad \text{in}\ C([0,T];H^s(I)), \quad \forall\ s\in(0,3),
\end{align*}
which furthermore implies by Sobolev embedding that as \(n\to \infty\)
\begin{align}\label{existence-38}
	v^{(n)}\rightarrow v\quad \text{in}\ C([0,T];C^2(I)).
\end{align}
According to \(\eqref{existence-21}_1\), one has 
\begin{equation*}
	\begin{aligned}
		\rho_0v_t^{(n)}=-\bigg[\frac{\rho_0^2}{(\eta_x^{(n-1)})^2}\bigg]_x
		+\bigg[\frac{\rho_0v_x^{(n)}}{(\eta_x^{(n-1)})^2}\bigg]_x,
	\end{aligned}
\end{equation*}
which, together with \eqref{existence-38}, yields that as \(n\to \infty\)
\begin{align}\label{existence-39}
	\rho_0v_t^{(n)}\rightarrow -\bigg(\frac{\rho_0^2}{\eta_x^2}\bigg)_x
	+\bigg(\frac{\rho_0v_x}{\eta_x^2}\bigg)_x\quad \text{in}\ C([0,T];C(I)).
\end{align}
Due to \eqref{existence-39}, the
distribution limit of \(v_t^{(n)}\) must be \(v_t\) as \(n\to \infty\), so, in particular, \(v\) is a classical solution to the problem \eqref{eq:main-2}.
Moreover, following the standard argument (see \cite{MR1867882}), one may show \(v\in C([0,T]; H^3(I))\cap C^1([0,T]; H^1(I))\).

\section{Proof of Theorem \ref{th:main-1}: Uniqueness}\label{Uniqueness Part}

The following observation will be useful in showing the uniqueness of the classical solution to the problem \eqref{eq:main-2}.

\subsection{A lower-order energy function}\label{lower-order energy function}

Define the following lower-order energy functional:
\begin{equation}\label{energy function-2}
	\begin{aligned}
		\mathcal{E}(t,v)&=\sum_{k=0}^2\|\sqrt{\rho_0}\partial_t^kv\|_{L^2(I)}^2
		+\sum_{k=0}^1\|\sqrt{\rho_0}\partial_t^kv_x\|_{L^2(I)}^2\\
		&+\|\rho_0\partial_tv_{xx}\|_{L^2(I)}^2+\sum_{k=2}^4\big\|\sqrt{\rho_0^k}\partial_x^kv\big\|_{L^2(I)}^2.
	\end{aligned}
\end{equation}
Then one can also close the energy estimates, namely \(\mathcal{E}(t,v)\) satisfies 	
\begin{equation}\label{a priori bound-2}
	\begin{aligned}
		\mathcal{E}(t,v)
		\leq \mathcal{M}_0+CtP(\sup_{0\leq s\leq t}\mathcal{E}^{1/2}(s,v))\quad \mathrm{for\ all}\ t\in [0,T]
	\end{aligned}
\end{equation}
with \(\mathcal{M}_0\) given by
\begin{align*}
	\mathcal{M}_0=P(\mathcal{E}(0,v_0)),
\end{align*}
where \(P\) denotes a generic polynomial of its arguments, and \(C\) is an absolutely constant depending only on \(\|\partial_x^l\rho_0\|_{L^\infty(I)}\ (l=0,1,2,3)\).\\

In fact, \eqref{a priori bound-2} follows from \eqref{I-Priori-time-8}, \eqref{I-Priori-time-12}, \eqref{II-Priori-time-4} in Section \ref{Energy Estimates},  \eqref{I-Priori-time-21}, \eqref{II-Priori-time-10}, \eqref{I-Priori-ellip-5}, \eqref{I-Priori-ellip-third}, \eqref{I-Priori-ellip-22}, and \eqref{I-Priori-ellip-31} in Section \ref{Elliptic estimates}. Indeed, \eqref{a priori bound-2} can be proved in a similar way as \eqref{I-Priori-ellip-32} with some modifications as follows.
In this case, the estimates on highest order
time-derivatives are \eqref{I-Priori-time-8} and
\eqref{I-Priori-time-21}, which can be obtained straightforwardly as \eqref{I-Priori-time-3} and \eqref{I-Priori-time-16}, respectively. Lemma \ref{le:Preliminary-1} and Lemma \ref{le:Preliminary-2} should be replaced by
\begin{lemma}\label{le:Preliminary-3} It holds that
	\begin{align}\label{Preliminary-10}
		\|v(\cdot,t)\|_{H^2(I)}\lesssim \mathcal{E}^{1/2}(t,v).
	\end{align}
Hence,
	\begin{align}
		&\|\eta_{xx}(\cdot,t)\|_{L^2(I)}\lesssim t\sup_{0\leq s\leq t}\mathcal{E}^{1/2}(t,v),\label{Preliminary-11}\\
		&\|v_x(\cdot,t)\|_{L^\infty(I)}
		\lesssim \mathcal{E}^{1/2}(t,v).\label{Preliminary-12}
	\end{align}
\end{lemma}
\begin{lemma}\label{le:Preliminary-4} It holds that
	{\small \begin{align}\label{Preliminary-13}
			\|\rho_0\partial_x^3v(\cdot,t)\|_{L^2(I)}\lesssim \mathcal{E}^{1/2}(t,v).
	\end{align}}
Consequently,
	{\small \begin{align}
			&\|\rho_0\partial_x^3\eta(\cdot,t)\|_{L^2(I)}
			\lesssim t\sup_{0\leq s\leq t}\mathcal{E}^{1/2}(t,v),\label{Preliminary-14}\\
			&\|\rho_0v_{xx}(\cdot,t)\|_{L^\infty(I)}\lesssim \mathcal{E}^{1/2}(t,v),\label{Preliminary-15}\\
			&\|\rho_0\eta_{xx}(\cdot,t)\|_{L^\infty(I)}
			\lesssim t\sup_{0\leq s\leq t}\mathcal{E}^{1/2}(s,v)\label{Preliminary-16}.
	\end{align}}
\end{lemma}
In elliptic estimates, one can use Lemma \ref{le:Preliminary-3} and Lemma \ref{le:Preliminary-4} to replace Lemma \ref{le:Preliminary-1} and Lemma \ref{le:Preliminary-2}. On the one hand, the second term on the RHS of \eqref{additial estimate} can be estimated as follows:
\begin{equation*}
	\begin{aligned}
		\bigg|\int_I\rho_0(\rho_0)_x\eta_x^{-5}\eta_{xx}v_x^2\,\diff x\bigg|
		\lesssim
		\|\rho_0\eta_{xx}\|_{L^\infty}\|v_x\|_{L^2}^2
		\leq Ct P(\sup_{0\leq s\leq t}\mathcal{E}^{1/2}(s,v)).
	\end{aligned}
\end{equation*}	
One may also handle the similar term in \eqref{I-Priori-ellip-17} as
\begin{equation*}
	\begin{aligned}
		\bigg|\int_I\rho_0^2(\rho_0)_x\eta_x^{-5}\eta_{xx}v_{xx}^2\,\diff x\bigg|
		\lesssim
		\|\rho_0\eta_{xx}\|_{L^\infty}\|v_{xx}\|_{L^2}^2
		\leq Ct P(\sup_{0\leq s\leq t}\mathcal{E}^{1/2}(s,v)),
	\end{aligned}
\end{equation*}
and the one in \eqref{I-Priori-ellip-30} as
\begin{equation*}
	\begin{aligned}
		\bigg|\int_I\rho_0^3(\rho_0)_x\eta_x^{-5}\eta_{xx}(\partial_x^3v)^2\,\diff x\bigg|
		\lesssim
		\|\rho_0\eta_{xx}\|_{L^\infty}\|\rho_0\partial_x^3v\|_{L^2}^2
		\leq Ct P(\sup_{0\leq s\leq t}\mathcal{E}^{1/2}(s,v)).
	\end{aligned}
\end{equation*}

On the other hand, one can use \(\|\rho_0\eta_{xx}\|_{L^\infty}\) to replace \(\|\eta_{xx}\|_{L^\infty}\) in \eqref{I-Priori-ellip-7}, \eqref{I-Priori-ellip-12}, \eqref{I-Priori-ellip-14}, \eqref{I-Priori-ellip-25}, \eqref{I-Priori-ellip-27.2}, \eqref{I-Priori-ellip-27.5} and  \eqref{I-Priori-ellip-28.5}; and use
\(\|\rho_0v_{xx}\|_{L^\infty}\) to replace \(\|v_{xx}\|_{L^\infty}\) in \eqref{I-Priori-ellip-27.5}; and use
\(\|\rho_0\partial_x^3\eta\|_{L^2}\) to replace \(\|\partial_x^3\eta\|_{L^2}\) in \eqref{I-Priori-ellip-7}, \eqref{I-Priori-ellip-25}, \eqref{I-Priori-ellip-27.5} and  \eqref{I-Priori-ellip-28.5}; and use \(\|\rho_0\partial_x^3v\|_{L^2}\) to replace \(\|\partial_x^3v\|_{L^2}\) in \eqref{I-Priori-ellip-27.2}; and use \(\|\rho_0^2\partial_x^4\eta\|_{L^2}\) to replace \(\|\rho_0\partial_x^4\eta\|_{L^2}\) in \eqref{I-Priori-ellip-25} and \eqref{I-Priori-ellip-28.5}. All of these replacements are possible due to the suitable choice of weights in the corresponding formulae.

\begin{remark} The main reason that we use \(E(t,v)\) instead of \(\mathcal{E}(t,v)\) to define the solution space is to achieve the regularity  \(v\in L^\infty([0,T]; H^3(I))\) which is needed to define the classical solutions. The energy functional \(\mathcal{E}(t,v)\) only gives us the regularity \(v\in L^\infty([0,T]; H^2(I))\), however, which will play an important role in showing the uniqueness of the classical solution to the problem \eqref{eq:main-2} in the next section.
	
\end{remark}

\subsection{Uniqueness of the classical solution to the problem \eqref{eq:main-2}}\label{Uniqueness of the classical solution}

Let \(v\) and \(w\) be two solutions to the problem \eqref{eq:main-2} on \([0,T]\) with initial data \((\rho_0,u_0)\) satisfying the same estimate. Their corresponding flow maps are:
\begin{equation*}
	\begin{aligned}
		\eta(x,t)=x+\int_0^tv(x,s)\,\diff s,\\
		\zeta(x,t)=x+\int_0^tw(x,s)\,\diff s.
	\end{aligned}
\end{equation*}

Set
\begin{equation*}
	\begin{aligned}
		\delta_{vw}=v-w.
	\end{aligned}
\end{equation*}
Then \(\delta_{vw}\) satisfies:
\begin{equation*}
	\begin{cases}
		\rho_0(\delta_{vw})_t+\big[\rho_0^2\big(\frac{1}{\eta_x^2}
		-\frac{1}{\zeta_x^2}\big)\big]_x
		=\big[\rho_0\big(\frac{v_x}{\eta_x^2}-\frac{w_x}{\zeta_x^2}\big)\big]_x&\quad \mbox{in}\ I\times (0,T],\\
		(\delta_{vw},\eta)=(0,e) &\quad \mbox{on}\ I\times \{t=0\},\\
		(\delta_{vw})_x=0 &\quad \mbox{on}\ \Gamma\times (0,T].
	\end{cases}
\end{equation*}
Note that
\begin{equation*}
	\begin{aligned}
		\bigg[\rho_0^2\bigg(\frac{1}{\eta_x^2}-\frac{1}{\zeta_x^2}\bigg)\bigg]_x
		=-\bigg[\frac{\rho_0^2}{\eta_x^2\zeta_x^2}
		\bigg(\int_0^t(\delta_{vw})_x\,\diff s\int_0^t(v_x+w_x)\,\diff s\bigg)\bigg]_x,
	\end{aligned}
\end{equation*}
and
\begin{equation*}
	\begin{aligned}
		\bigg[\rho_0\bigg(\frac{v_x}{\eta_x^2}-\frac{w_x}{\zeta_x^2}\bigg)\bigg]_x
		&=\bigg[\frac{\rho_0}{\eta_x^2\zeta_x^2}
		\bigg((\delta_{vw})_x+2(\delta_{vw})_x\int_0^tw_x\,\diff s
		-2w_x\int_0^t(\delta_{vw})_x\,\diff s\\
		&\quad
		-\int_0^t(\delta_{vw})_x\,\diff s\int_0^t(v_x+w_x)\,\diff s\bigg)\bigg]_x,
	\end{aligned}
\end{equation*}
which contain some additional error terms:
\begin{equation*}
	\begin{aligned}
(\delta_{vw})_x\int_0^tw_x\,\diff s,\quad w_x\int_0^t(\delta_{vw})_x\,\diff s\quad {\rm{and}}\quad \int_0^t(\delta_{vw})_x\,\diff s\int_0^t(v_x+w_x)\,\diff s.
	\end{aligned}
\end{equation*}
Unfortunately, it can be checked that these additional error terms make it difficult to derive an inequality as \eqref{I-Priori-ellip-32} for 
\(E(t,\delta_{vw})\). In other words, it needs some higher-order energy functionals than \(E(t,v)\) and \(E(t,w)\) to control these error terms if one works with \(E(t,\delta_{vw})\).

So we instead work with the lower-order energy functional \(\mathcal{E}(t,\delta_{vw})\) defined by \eqref{energy function-2}, and find that all the error terms can be easily controlled by the energy functionals \(E(t,v)\) and \(E(t,w)\). Therefore we may  get finally (see Subsection \ref{lower-order energy function} for more details) that
\begin{equation}\label{uniques}
	\begin{aligned}
		\sup_{0\leq s\leq t}\mathcal{E}(t,\delta_{vw})
		\leq CtP(\sup_{0\leq s\leq t}\mathcal{E}^{1/2}(s,\delta_{vw}))\quad \mathrm{for\ all}\ t\in [0,T],
	\end{aligned}
\end{equation}
where \(C\) depends on \(E(t,v)\) and \(E(t,w)\). Hence \(\delta_{vw}=0\) follows from \eqref{uniques}.

\section*{Acknowledgment}
Li's research was supported by the National Natural Science Foundation of China (Grant Nos. 11931010 and 11871047), and by the key research project of Academy for Multidisciplinary Studies, Capital Normal University, and by the Capacity Building for Sci-Tech Innovation-Fundamental Scientific Research Funds 007/20530290068. Wang's research was supported by the Grant No. 830018 from China. 
Xin's research was supported by the Zheng Ge Ru Foundation and by Hong Kong RGC Earmarked Research Grants CUHK14301421, CUHK14300819, CUHK14302819, CUHK14300917,  Basic and Applied Basic Research Foundations of Guangdong Province 20201\
31515310002, and the Key Project of National Nature Science Foundation of China (Grant No. 12131010).

\end{document}